\documentclass[11pt]{amsart}%
\usepackage{palatino, mathpazo}
\usepackage{amsfonts}
\usepackage{amsmath}
\usepackage{amssymb,latexsym}
\usepackage{graphicx}
\usepackage[mathscr]{eucal}
\usepackage{amssymb}%
\providecommand{\U}[1]{\protect \rule{.1in}{.1in}}

\newtheorem{theorem}{Theorem}[section]

\newtheorem{corollary}[theorem]{Corollary}

\newtheorem{definition}[theorem]{Definition}

\newtheorem{lemma}[theorem]{Lemma}

\newtheorem{proposition}[theorem]{Proposition}
\theoremstyle{remark}
\newtheorem{remark}[theorem]{Remark}

\numberwithin{equation}{section}
\begin{document}
\title[geometry of generalized Lam\'{e} equation, III]{The geometry of generalized Lam\'{e} equation, III:\\ One-to-one of the Riemann-Hilbert correspondence}
\author{Zhijie Chen}
\address{Department of Mathematical Sciences, Yau Mathematical Sciences Center,
Tsinghua University, Beijing, 100084, China }
\email{zjchen2016@tsinghua.edu.cn}
\author{Ting-Jung Kuo}
\address{Department of Mathematics, Taiwan Normal University, Taipei 11677, Taiwan }
\email{tjkuo1215@ntnu.edu.tw}
\author{Chang-Shou Lin}
\address{Center for Advanced Study in
Theoretical Sciences (CASTS), Taiwan University, Taipei 10617, Taiwan }
\email{cslin@math.ntu.edu.tw}

\begin{abstract}
In this paper, the third in a series, we continue to study the
generalized Lam\'{e} equation H$(n_0,n_1,n_2,n_3;B)$ with the Darboux-Treibich-Verdier potential
\begin{equation*}
y^{\prime \prime }(z)=\bigg[  \sum_{k=0}^{3}n_{k}(n_{k}+1)\wp(z+\tfrac{
\omega_{k}}{2}|\tau)+B\bigg]  y(z),\quad n_{k}\in \mathbb{Z}_{\geq0}
\end{equation*}
and a related linear ODE with additional singularities $\pm p$
from the monodromy aspect.
We establish the uniqueness of these ODEs with respect to the global monodromy data. Surprisingly, our result shows that the Riemann-Hilbert correspondence from the set \[\{\text{H}(n_0,n_1,n_2,n_3;B)|B\in\mathbb{C}\}\cup \{\text{H}(n_0+2,n_1,n_2,n_3;B) | B\in\mathbb{C}\}\] to the set of group representations $\rho:\pi_1(E_{\tau})\to SL(2,\mathbb{C})$ is one-to-one.
We emphasize that this result is not trivial at all. There is an example that for $\tau=\frac12+i\frac{\sqrt{3}}{2}$, there are $B_1,B_2$ such that the monodromy representations of H$(1,0,0,0;B_1)$ and H$(4,0,0,0;B_2)$ are {\bf the same}, namely the Riemann-Hilbert correspondence from the set \[\{\text{H}(n_0,n_1,n_2,n_3;B)|B\in\mathbb{C}\}\cup \{\text{H}(n_0+3,n_1,n_2,n_3;B) | B\in\mathbb{C}\}\] to the set of group representations is {\bf not} necessarily one-to-one. This example shows that our result is completely different from the classical one concerning linear ODEs defined on $\mathbb{CP}^1$ with finite singularities.
\end{abstract}

\maketitle

\section{Introduction}

Throughout the paper, we use the notations $\omega_{0}=0$, $\omega_{1}=1$,
$\omega_{2}=\tau$, $\omega_{3}=1+\tau$ and $\Lambda_{\tau}=\mathbb{Z+Z}\tau$,
where $\tau \in \mathbb{H}=\{  \tau|\operatorname{Im}\tau>0\}  $.
Define $E_{\tau}:=\mathbb{C}/\Lambda_{\tau}$ to be a flat torus and $E_{\tau}[2]:=\{ \frac{\omega_{k}}{2}|k=0,1,2,3\}+\Lambda
_{\tau}$ to be the set consisting of the lattice points and $2$-torsion points
in $E_{\tau}$. For $z\in\mathbb{C}$ we denote $[z]:=z \ (\text{mod}\ \Lambda_{\tau}) \in E_{\tau}$.
For a point $[z]$ in $E_{\tau}$ we often write $z$ instead of $[z]$ to
simplify notations when no confusion arises.

Let $\wp(z)=\wp(z|\tau)$ be the
Weierstrass elliptic function with periods $\Lambda_{\tau}$ and define $e_k(\tau):=\wp(\frac{\omega_k}{2}|\tau)$, $k=1,2,3$. Let $\zeta(z)=\zeta(z|\tau):=-\int^{z}\wp(\xi|\tau)d\xi$
be the Weierstrass zeta function with two quasi-periods $\eta_{k}(\tau)$, $%
k=1,2$:
\begin{equation}
\eta_{k}(\tau):=2\zeta(\tfrac{\omega_{k}}{2} |\tau)=\zeta(z+\omega_{k} |\tau)-\zeta(z|\tau),\quad k=1,2,
\label{40-2}
\end{equation}
and $\sigma(z)=\sigma(z|\tau):=\exp \int^{z}\zeta(\xi)d\xi$ be the Weierstrass sigma function. Notice that $\zeta(z)$ is an odd
meromorphic function with simple poles at $\Lambda_{\tau}$ and $\sigma(z)$
is an odd entire function with simple zeros at $\Lambda_{\tau}$.

This is the third in a series of papers, initiated in Part I \cite{CKL1}, to study the generalized Lam\'{e} equation
(denoted by GLE$(\mathbf{n},p,A,\tau)$):
\begin{equation}
y^{\prime \prime}(z)=I_{\mathbf{n}}(z;p,A,\tau)y(z),\quad z\in\mathbb{C},
\label{54-12}%
\end{equation}
where the potential $I_{\mathbf{n}}(z;p,A,\tau)$ is given by%
\begin{equation}
I_{\mathbf{n}}(z;p,A,\tau) =\left[
\begin{array}
[c]{l}%
\sum_{k=0}^{3}n_{k}(n_{k}+1)\wp(z+\frac{\omega_{k}}{2}|\tau)+\frac{3}{4}%
(\wp(z+p|\tau)+\\
\wp(z-p|\tau))+A(\zeta(z+p|\tau)-\zeta(z-p|\tau))+B
\end{array}
\right]  \label{potential1}%
\end{equation}
with $\mathbf{n}=(n_0, n_1, n_2, n_3)$, $n_{k}\in \mathbb{Z}_{\geq 0}$ for all $k$, $\pm [p]  \not \in E_{\tau}[2]$ and%
\begin{equation}
B=A^{2}-\zeta(2p|\tau)A-\frac{3}{4}\wp(2p|\tau)-\sum_{k=0}^{3}n_{k}(n_{k}+1)\wp(
p+\tfrac{\omega_{k}}{2}|\tau)  . \label{apparent-B}%
\end{equation}
The (\ref{apparent-B}) is
equivalent to that $\pm[p]$ are apparent singularities (i.e.
non-logarithmic); see \cite{Chen-Kuo-Lin} for a proof and also \cite{CKL-JGP,CKLW,Takemura} for recent studies on (\ref{54-12}).
Remark that all singularities of GLE$(\mathbf{n},p,A,\tau)$ are apparent and
\begin{align}
&  \text{GLE}(\mathbf{n},p,A,\tau)\; \text{\textit{is independent of any
representative }}\tilde{p}\in p+\Lambda_{\tau}\label{GLE-invariant}\\
&  \text{\textit{and} GLE}(\mathbf{n},p,A,\tau)=\text{GLE}(\mathbf{n}%
,-p,-A,\tau).\nonumber
\end{align}
For convenience, we often omit some of $\{\mathbf{n},p,A,\tau\}$ in the notations
when no confusion should arise.

Our motivation of studying GLE (\ref{54-12}) is inspired by the so-called \emph{elliptic form} of Painlev\'{e} VI equation (denoted by EPVI$(\alpha_{0},\alpha_{1},\alpha
_{2},\alpha_{3})$):
\begin{equation}
\frac{d^{2}p(\tau)}{d\tau^{2}}=\frac{-1}{4\pi^{2}}\sum_{k=0}^{3}\alpha_{k}%
\wp^{\prime}\left(  \left.  p(\tau)+\frac{\omega_{k}}{2}\right \vert
\tau \right)  , \label{124}%
\end{equation}
where
\begin{equation}
\alpha_{k}=\tfrac{(2n_k+1)^2}{8},\quad n_k\in\mathbb{Z}_{\geq 0},\quad k=0,1,2,3.
\label{125}%
\end{equation}
In \cite{Chen-Kuo-Lin} we proved that GLE (\ref{54-12}) with $(p, A)=(p(\tau),A(\tau))$ preserves the monodromy as $\tau$ deforms if and only if $(p(\tau),A(\tau))$ satisfies the following Hamiltonian system
\begin{equation}
\left \{
\begin{array}
[c]{l}%
\frac{dp(\tau)}{d\tau}=\frac{\partial \mathcal{H}}{\partial A}=\frac{-i}{4\pi
}(2A-\zeta(2p|\tau)+2p\eta_{1}(\tau))\\
\frac{dA(\tau)}{d\tau}=-\frac{\partial \mathcal{H}}{\partial p}=\frac{i}{4\pi
}\left(
\begin{array}
[c]{l}%
(2\wp(2p|\tau)+2\eta_{1}(\tau))A-\frac{3}{2}\wp^{\prime}(2p|\tau)\\
-\sum_{k=0}^{3}n_{k}(n_{k}+1)\wp^{\prime}(p+\frac{\omega_{k}}{2}|\tau)
\end{array}
\right)
\end{array}
\right.  , \label{142-0}%
\end{equation}
with
\begin{align*}
\mathcal{H}  &  =\frac{-i}{4\pi}\left[
\begin{array}
[c]{l}%
A^{2}+(2p\eta_{1}(\tau)-\zeta(2p|\tau))A-\frac{3}{4}\wp(2p|\tau)\\
-\sum_{k=0}^{3}n_{k}(n_{k}+1)\wp(p+\frac{\omega_{k}}{2}|\tau)
\end{array}
\right]\\
&  =\frac{-i}{4\pi}(B+2p\eta_{1}(\tau)A),
\end{align*}
or equivalently $p(\tau)$ is a solution of EPVI$(\alpha_{0},\alpha_{1},\alpha
_{2},\alpha_{3})$.

Since the local exponents of GLE (\ref{54-12}) at $\frac{\omega_{k}}{2}$ (resp. at $\pm p$) are $-n_{k}$, $n_{k}+1$ (resp. $-\frac12$, $\frac32$), the local monodromy matrix at $\frac{\omega_k}{2}$ (resp. at $\pm p$) is the identity matrix $I_2$ (resp. is $-I_2$). Denote by $L$ the straight segment connecting $\pm p$. Then any solution $y(z)$ of GLE (\ref{54-12}) can be viewed as a single-valued
meromorphic function in $\mathbb{C}\backslash(L+\Lambda_{\tau})$, and in this
region $y(-z)$ and $y(z+\omega_j)$ are well-defined.
See
\cite{Chen-Kuo-Lin,Takemura} or Section 2. Let $(y_1, y_2)$ be any linearly independent solutions of  GLE (\ref{54-12}). Then there are monodromy matrices $N_1, N_2\in SL(2,\mathbb{C})$ such that
\begin{equation}\label{eqee}
\begin{pmatrix}
y_{1}(z+\omega_j)\\
y_{2}(z+\omega_j)
\end{pmatrix}
=N_j
\begin{pmatrix}
y_{1}(z)\\
y_{2}(z)
\end{pmatrix}
,\quad j=1,2,\quad\text{and }
\end{equation}
\begin{equation}
\label{eq-Scommute}
N_1N_2=N_2N_1.
\end{equation}
Furthermore, the monodromy group of GLE (\ref{54-12}) is generated by $-I_2, N_1, N_2$.
By (\ref{eq-Scommute}), clearly there are two cases (see Part I \cite{CKL1}):

\begin{itemize}
\item[Case (a)] Completely reducible (i.e. all the monodromy matrices have two
linearly independent common eigenfunctions). Up to a common conjugation,
$N_1$ and $N_2$ can be expressed as%
\begin{equation}
N_1=%
\begin{pmatrix}
e^{-2\pi is} & 0\\
0 & e^{2\pi is}%
\end{pmatrix}
,\text{ \  \  \ }N_2=%
\begin{pmatrix}
e^{2\pi ir} & 0\\
0 & e^{-2\pi ir}%
\end{pmatrix}
\label{Mono-1}%
\end{equation}
for some $(r,s)\in \mathbb{C}^{2}\backslash \frac{1}{2}\mathbb{Z}^{2}$. In
particular,
\begin{equation}
(\text{tr}N_1,\text{tr}N_2)=(2\cos2\pi s,2\cos2\pi
r)\not \in \{ \pm(2,2),\pm(2,-2)\}. \label{complete-rs}%
\end{equation}

\item[Case (b)] Not completely reducible (i.e. the space of common eigenfunctions
is of dimension $1$). Up to a common conjugation, $N_1$ and
$N_2$ can be expressed as%
\begin{equation}
N_1=\varepsilon_{1}%
\begin{pmatrix}
1 & 0\\
1 & 1
\end{pmatrix}
,\text{ \  \  \ }N_2=\varepsilon_{2}%
\begin{pmatrix}
1 & 0\\
\mathcal{C} & 1
\end{pmatrix}
, \label{Mono-21}%
\end{equation}
where $\varepsilon_{1},\varepsilon_{2}\in \{ \pm1\}$ and $\mathcal{C}%
\in \mathbb{C}\cup \{ \infty \}$. In particular,
\begin{equation}
(\text{tr}N_1,\text{tr}N_2)=(2\varepsilon_{1}%
,2\varepsilon_{2})\in \{ \pm(2,2),\pm(2,-2)\}. \label{notcompleteC}%
\end{equation}
Remark that if $\mathcal{C}=\infty$, then (\ref{Mono-21}) should be understood
as%
\begin{equation}
N_1=\varepsilon_{1}%
\begin{pmatrix}
1 & 0\\
0 & 1
\end{pmatrix}
,\text{ \  \  \ }N_2=\varepsilon_{2}%
\begin{pmatrix}
1 & 0\\
1 & 1
\end{pmatrix}
. \label{Mono-31}%
\end{equation}

\end{itemize}

\noindent For later usage we will briefly review it in
Section \ref{Monodromy}.
In this paper, GLE  (\ref{54-12}) (and also the H$(\mathbf{n},B,\tau)$ below) is called \emph{completely reducible} if Case (a) occurs; \emph{not completely reducible} if Case (b) occurs.

In \cite{Chen-Kuo-Lin} we proved that if $p(\tau)$ is a solution of EPVI$(\alpha_{0},\alpha_{1},\alpha
_{2},\alpha_{3})$ and
\[p(\tau)\to \tfrac{\omega_k}{2}=\tfrac{\omega_k(\tau_0)}{2},\quad\text{as  }\tau\to \tau_0,\]
then  the potential $I_{\mathbf{n}}(z;p(\tau),A(\tau),\tau)$ converges to the well-known \emph{Darboux-Treibich-Verdier potential } $I_{\mathbf{n}_k^{\pm}}(z;B,\tau_0)$ for some $B\in\mathbb{C}$, where the Darboux-Treibich-Verdier potential is defined as (\cite{Darboux,Treibich,TV})
\begin{equation}
I_{\mathbf{n}}(z;B,\tau):=\sum_{k=0}^3n_k(n_k+1)\wp(z+\tfrac{\omega_k}{2}|\tau)+B,
\end{equation}
and $\mathbf{n}^\pm_k$ is defined by replacing $n_k$ in
$\mathbf{n}$ with $n_k\pm 1$.
That is, by considering the corresponding generalized Lam\'{e} equation (denoted by H$(\mathbf{n},B,\tau)$ or simply H$(\mathbf{n},B)$)
\begin{equation}  \label{eq21}
y^{\prime \prime }(z)=I_{\mathbf{n}}(z;B,\tau)y(z),\quad z\in\mathbb{C},
\end{equation}
we have that GLE$(\mathbf{n},p(\tau),A(\tau),\tau)$ converges to H$(\mathbf{n}_k,B,\tau_0)$.

For H$(\mathbf{n},B,\tau)$ we always assume $\max_k n_k\geq 1$.
H$(\mathbf{n},B,\tau)$ is the elliptic form of the well-known Heun's equation and the Darboux-Treibich-Verdier potential is known
as an elliptic algebro-geometric solution of the
KdV hierarchy \cite{GW1,Treibich,TV}. See also a
series of papers \cite{Takemura1,Takemura2,Takemura3,Takemura4,Takemura5} by Takemura,
where H$(\mathbf{n},B,\tau)$ was studied as the eigenvalue problem for the
Hamiltonian of the $BC_{1}$ (one particle) Inozemtsev model.
When $\mathbf{n}=(n,0,0,0)$, the potential $n(n+1)\wp(z|\tau)$ is the well-known
Lam\'{e} potential and (\ref{eq21}) becomes the Lam\'{e} equation
\begin{equation}  \label{Lame}
y^{\prime \prime }(z)=[n(n+1)\wp(z|\tau)+B]y(z),\quad z\in\mathbb{C}.
\end{equation}
Ince \cite{Ince} first discovered that the Lam\'{e} potential is a
finite-gap potential. See also the classic texts \cite{Halphen,Poole,Whittaker-Watson} and recent works \cite{CLW,Dahmen,LW2,Maier} for more details about (\ref{Lame}).

Like GLE$(\mathbf{n},p,A,\tau)$, the local monodromy matrix of H$(\mathbf{n}, B, \tau)$
at $\frac{\omega_{k}}{2}$ is also $I_{2}$. Thus the monodromy representation $\rho:\pi_{1}(E_{\tau})  \to
SL(2,\mathbb{C})$ is abelian, i.e. the same Cases (a) or (b) occurs.

The main purpose of this paper is to study the natural problem: {\it Whether  H$(\mathbf{n}, B)$ or GLE$(\mathbf{n},p,A,\tau)$ is unique with respect to the monodromy representation}, or equivalently, {\it whether the Riemann-Hilbert correspondence from the set $\{\text{H}(\mathbf{n},B)|B\in\mathbb{C}\}$ or $\{\text{GLE}(\mathbf{n},p,A,\tau)|p\notin E_{\tau}[2], A\in\mathbb{C}\}$ to the set of group representations $\rho:\pi_1(E_{\tau})\to SL(2,\mathbb{C})$ is one-to-one (i.e. injective)}?

\begin{remark}
By letting
$x=\wp(z)$, H$(\mathbf{n},B)$ can be projected to the Heun's equation on
$\mathbb{CP}^{1}$, for which the monodromy representation is {\it irreducible} if and only if Case (a) occurs,
and {\it reducible} if and only if Case (b) occurs. In other words, the monodromy of H$(\mathbf{n},B)$ is easier to compute than that of the Heun's equation on
$\mathbb{CP}^{1}$. This is an advantage of studying H$(\mathbf{n},B)$.
Most of the references in the
literature are devoted to irreducible representation on
$\mathbb{CP}^{1}$, but very few are devoted to reducible representation. In this paper we deal with the both two cases for H$(\mathbf{n}, B)$.
\end{remark}

For the completely reducible case (a), the one-to-one of the Riemann-Hilbert correspondence was proved in \cite[Theorem 3.3]{LW2} for the Lam\'{e} case and later in
Part II \cite[Lemma 2.3]{CKL2} for the Darboux-Treibich-Verdier case (See also \cite{CKL2,LW2} for important applications of such results). However, the proofs in \cite{CKL2,LW2} can \emph{not} work for the not completely reducible case (b). In this paper, we develop a new approach, which applies the deep relation with Painlev\'{e} VI equation and seems more sophisticated but works for the not completely reducible case and also GLE$(\mathbf{n},p,A,\tau)$.

Remark that although the monodromy matrices $N_j$'s depend on the choice
of linearly independent solutions, they are unique up to a common
conjugation. In particular, tr$N_j$ is \emph{independent} of the
choice of solutions, i.e. tr$N_j$ is
uniquely determined by GLE$(\mathbf{n},p,A)$ or H$(\mathbf{n},B)$.
We say%
\begin{equation}
(r_{1},s_{1})\sim(r_{2},s_{2})\text{ if }(r_{1},s_{1})\equiv \pm(r_{2}%
,s_{2})\operatorname{mod}\mathbb{Z}^{2}. \label{rs-equivalent}%
\end{equation}
Then in Case (a), $(r,s)$ is uniquely determined in $(\mathbb{C}^{2}%
\backslash \frac{1}{2}\mathbb{Z}^{2})/\sim$.

\begin{definition}
\label{MD} Given GLE$(\mathbf{n},p,A,\tau)$ (resp. H$(\mathbf{n},B,\tau)$), we call%
\[
\left \{
\begin{array}
[c]{l}%
(r,s)\in(\mathbb{C}^{2}\backslash \frac{1}{2}\mathbb{Z}^{2})/\sim \text{ \ if the
monodromy is completely reducible}\\
(\text{tr}N_1,\text{tr}N_2,\mathcal{C})\text{ if the
monodromy is not completely reducible}%
\end{array}
\right.
\]
to be its global monodromy data.
\end{definition}

The main purpose of this paper is to establish the uniqueness of such ODEs with respect to the global monodromy data. For $k\in \{0,1,2,3\}$ and $\mathbf{n}=(n_0,n_1,n_2,n_3)$, we define $\mathbf{n}_k$ by replacing $n_k$ in $\mathbf{n}$ with $n_k+2$, i.e. \begin{equation}\label{n-0-0}\mathbf{n}_0=(n_0+2,n_1,n_2,n_3),\quad
\mathbf{n}_1=(n_0,n_1+2,n_2,n_3)\end{equation}
and so on. The main result of this paper is the following uniqueness theorem.

\begin{theorem}
\label{thm1}Fix any $\mathbf{n}$ and $\tau$. Then the following hold.
\begin{itemize}
\item[(1)] If GLE$(\mathbf{n},p_{1},A_{1})$
and GLE$(\mathbf{n},p_{2},A_{2})$ have the same global monodromy data, then
$
\text{GLE}(\mathbf{n},p_{1},A_{1})=\text{GLE}(\mathbf{n},p_{2},A_{2}).
$
\item[(2)] If H$(\mathbf{n},B_1)$ and H$(\mathbf{n},B_2)$ have the same global monodromy data, then $\text{H}(\mathbf{n},B_1)=\text{H}(\mathbf{n},B_2)$.
\item[(3)] Fix any $k\in\{0,1,2,3\}$. Then the global monodromy datas of H$(\mathbf{n}, B_1, \tau)$ and H$(\mathbf{n}_k, B_2, \tau)$ can not be the same for any $B_1, B_2\in\mathbb{C}$.
\end{itemize}
\end{theorem}

\begin{remark}\label{rm1.4} H$(\mathbf{n}, B_1, \tau)$ and H$(\mathbf{n}_k, B_2, \tau)$ have {\it different local exponents} at the singularity $\frac{\omega_k}{2}$. Therefore, it is quite surprising to us that for fixed $\mathbf{n}$, $\tau$ and $k$, {\it the Riemann-Hilbert correspondence from the set $\{\text{H}(\mathbf{n},B,\tau) | B\in\mathbb{C}\}\cup \{\text{H}(\mathbf{n}_k,B,\tau) | B\in\mathbb{C}\}$ to the set of group representations $\rho:\pi_1(E_{\tau})\to SL(2,\mathbb{C})$ is one-to-one}. We emphasize that this result is not trivial at all. For example, we can not expect the one-to-one correspondence from $\{\text{H}(\mathbf{n},B,\tau) | B\in\mathbb{C}\}\cup \{\text{H}((n_0+3,n_1,n_2,n_3),B,\tau) | B\in\mathbb{C}\}$ to the set of group representations. Indeed, Wang and the third author \cite[Theorem 4.5]{LW2} proved the existence of a pre-modular form $Z_{r,s}^{(n)}(\tau)$ such that the global monodromy data of H$((n,0,0,0),B,\tau)$ for some $B$ is given by $(r,s)\notin\frac12\mathbb{Z}^2$ if and only if $Z_{r,s}^{(n)}(\tau)=0$. Now for $\tau_0=\frac12+i\frac{\sqrt{3}}{2}$,
it was proved in \cite[Example 2.6]{LW} that
\[Z_{\frac13,\frac13}^{(1)}(\tau_0)=0,\quad \wp(\tfrac{1+\tau_0}{3}|\tau_0)=0.\]
Inserting these and $g_2(\tau_0)=0$ into the expression of $Z_{r,s}^{(4)}(\tau)$ (see \cite[(5.8)]{LW2}), we obtain $Z_{\frac13,\frac13}^{(4)}(\tau_0)=Z_{\frac13,\frac13}^{(1)}(\tau_0)=0$, so there are $B_1,B_2$ such that the global monodromy datas of H$((1,0,0,0)$, $B_1,\tau_0)$ and H$((4,0,0,0),B_2,\tau_0)$ are both $(\frac13,\frac13)$.

\end{remark}

\begin{remark}
The uniqueness with respect to the same monodromy group does not necessarily
hold. For example, our later argument shows that given $\mathbf{n}$ and
$m\in \mathbb{N}_{\geq3}$, there exist $(p_{j},A_{j})$, $j=1,2$ and the same
$\tau$ such that for GLE$(\mathbf{n},p_{1},A_{1})$,%
\[
N_1=%
\begin{pmatrix}
e^{-2\pi i/m} & 0\\
0 & e^{2\pi i/m}%
\end{pmatrix}
,\text{ }N_2=%
\begin{pmatrix}
e^{2\pi i/m} & 0\\
0 & e^{-2\pi i/m}%
\end{pmatrix},
\]
i.e. $($tr$N_1, $tr$N_2)=(2\cos \frac{2\pi}%
{m},2\cos \frac{2\pi}{m})$, and for GLE$(\mathbf{n},p_{2},A_{2})$,%
\[
\tilde{N}_1=%
\begin{pmatrix}
e^{-2\pi i/m} & 0\\
0 & e^{2\pi i/m}%
\end{pmatrix}
,\text{ }\tilde{N}_2=%
\begin{pmatrix}
e^{4\pi i/m} & 0\\
0 & e^{-4\pi i/m}%
\end{pmatrix},
\]
i.e. $($tr$\tilde{N}_1, $tr$\tilde{N}_2)=(2\cos \frac{2\pi}%
{m},2\cos \frac{4\pi}{m})$. Thus, these two GLEs have different global monodromy datas (or equivalently, different monodromy representations). However,
they have the same monodromy group (i.e. the images of the monodromy representations are the same)
\[
\left \langle -I_{2},N_1,N_2\right \rangle
=\left \langle -I_{2},\tilde{N}_1,\tilde{N}_2\right \rangle
=\left \langle -I_{2},N_1\right \rangle .
\]
\end{remark}

\begin{remark}\label{rm1.6}
For a class of linear ODEs defined on $\mathbb{CP}^1$ with finite singularities, classically there is a one-to-one correspondence of such ODEs and their monodromy datas; see e.g. \cite[Proposition 2.2]{FIK}. However, the set of monodromy datas for this classical result contains connection matrices at each singularities. Hence, our Theorem \ref{thm1} is different from the classical one because no apriori information about the connection matrices are assumed in Theorem \ref{thm1}. Remark that due to the inclusion of connection matrices, the class of ODEs in the classical result has no restrictions. However, Theorem \ref{thm1} has a strict restriction on the class of ODEs. For example, as mentioned in Remark \ref{rm1.4}, the one-to-one correspondence fails if the class of ODEs contains H$((1,0,0,0)$, $B_1,\tau_0)$ and H$((4,0,0,0),B_2,\tau_0)$.
\end{remark}

The rest of the paper is organized as follows. In Section \ref{Monodromy}, we briefly review the monodromy theory of GLE$(\mathbf{n},A,p)$. Our proof of Theorem \ref{thm1} relies on the connection between GLE$(\mathbf{n},A,p)$ and Painlev\'{e} VI equation established in \cite{Chen-Kuo-Lin}, which is briefly reviewed in Section \ref{GLE-PVI}. In Sections \ref{Hitchin-case}-\ref{General-case-Ok}, we establish the uniqueness of solutions of certain Painlev\'{e} VI equations with respect to the global monodromy datas of GLE$(\mathbf{n},A,p)$. This theory will be applied to prove Theorem \ref{thm1} in Section \ref{unique-u}. An application of Theorem \ref{thm1} will be given in Section \ref{sec-app}. One can see that our proof of Theorem \ref{thm1} is purely analytic. Recently Prof. Treibich communicated with us and he conjectured that there should be a different proof of Theorem \ref{thm1} via algebraic geometry. This is a very interesting question and deserves further study elsewhere.

\section{Preliminaries}

\label{Monodromy}

In this section, we briefly review the basic theory about the monodromy
representation of GLE$(\mathbf{n},A,p)$ and H$(\mathbf{n}, B)$ from \cite{CKL1,Takemura}, which will be applied in the proof of Theorem \ref{thm1}.

\subsection{The unique even elliptic solution}
Let $y_{1},y_{2}$ be any two solutions of GLE$(\mathbf{n},A,p)$ and set
$\Phi(z)=y_{1}(z)y_{2}(z).$ Then $\Phi(z)$ satisfies the second
symmetric product equation for GLE$(\mathbf{n},A,p)$:%
\begin{equation}
\Phi^{\prime \prime \prime}(z)-4I(z)\Phi^{\prime}(z)-2I^{\prime}(z)\Phi(z)=0,
\label{303-1}%
\end{equation}
where $I(z)=I_{\mathbf{n}}(z; p, A, \tau)$. The following lemma follows from \cite[Propositions 2.1 and 2.9]{Takemura}. For later usage, we sketch
the proof of the existence here, and refer the proof of the uniqueness to
\cite[Proposition 2.9]{Takemura} or Part I \cite[Proposition 2.3]{CKL1}.

\begin{lemma}
\label{lem6.3} \cite{Takemura} Equation (\ref{303-1}) has a unique (up to
multiplying a nonzero constant) even elliptic solution $\Phi_{e}(z)$.
\end{lemma}

\begin{proof}
Fix any base point $q_{0}\in E_{\tau}\backslash(E_{\tau}[2]\cup \{ \pm \lbrack
p]\})$. Since the local monodromy matrice at $\frac{\omega_k}{2}$ is $I_2$, the monodromy representation of GLE (\ref{54-12}) is reduced to $\rho:\pi_{1}(  E_{\tau}%
\backslash\{  \pm \lbrack p]\}  ,q_{0})  \rightarrow
SL(2,\mathbb{C})$. Let
$\gamma_{\pm}\in \pi_{1}(  E_{\tau}%
\backslash\{  \pm \lbrack p]\}  ,q_{0})   $ be a
simple loop encircling $\pm p$ counterclockwise respectively, and $\ell_{j}%
\in \pi_{1}(  E_{\tau}%
\backslash\{  \pm \lbrack p]\}  ,q_{0})  $, $j=1,2$, be two
fundamental cycles of $E_{\tau}$ connecting $q_{0}$ with $q_{0}+\omega_{j}$
such that $\ell_{j}$ does not intersect with $L+\Lambda_{\tau}$ (here $L$ is
the straight segment connecting $\pm p$) and satisfies%
\begin{equation}
\gamma_{-}\gamma_{+}=\ell_{1}\ell_{2}\ell_{1}^{-1}\ell_{2}^{-1}\text{ \  \ in
}\pi_{1}\left(  E_{\tau}\backslash \left \{  \pm \lbrack p]\right \}
,q_{0}\right)  . \label{II-iv1}%
\end{equation}
Since
\begin{equation}
\rho(\gamma_{\pm})=-I_{2}, \label{89-21}%
\end{equation}
we have $N_j=\rho(\ell_{j})$, $N_1N_2=N_2N_1$
and the monodromy group of (\ref{54-12}) is generated by $\{-I_{2},N_1,N_2\}$, namely is abelian. So there is a
common eigenfunction (or called \emph{eigen-solution}) $y_{1}(z)$ of all
monodromy matrices. Let
$\varepsilon_{i}$ be the eigenvalue: $\ell_{i}^{\ast}y_{1}(z)=\varepsilon
_{i}y_{1}(z)$, where $\ell^{\ast}y(z)$ denotes the analytic continuation of
$y(z)$ along the loop $\ell$. Note that $y_{1}(z)$ have branch points only at
$\pm p+\Lambda_{\tau}$. By (\ref{89-21}), $y_{1}(z)$ can be viewed as a single-valued
meromorphic function in $\mathbb{C}\backslash(L+\Lambda_{\tau})$, and in this
region, $y_{1}(-z)$ is well-defined and%
\begin{equation}
y_{1}(z+\omega_{i})=\ell_{i}^{\ast}y_{1}(z)=\varepsilon_{i}y_{1}(z),\text{
}i=1,2, \label{304-111}%
\end{equation}
since the fundamental circles are chosen not to intersect with $L+\Lambda
_{\tau}$.

Let $y_{2}(z)=y_{1}(-z)$ in $\mathbb{C}\backslash(L+\Lambda_{\tau})$. Clearly
$y_{2}(z)$ is also a solution of (\ref{54-12}) and (\ref{304-111}) implies
\begin{equation}
y_{2}(z+\omega_{i})=\ell_{i}^{\ast}y_{2}(z)=\varepsilon_{i}^{-1}%
y_{2}(z),\text{ }i=1,2, \label{304-12}%
\end{equation}
i.e. $y_{2}(z)$ is also an eigenfunction with eigenvalue $\varepsilon_{i}^{-1}$.
Define
\[
\Phi_{e}(z):=y_{1}(z)y_{2}(z)=y_{1}(z)y_{1}(-z).
\]
Obviously, $\pm \lbrack p]$ are no longer branch points of $\Phi
_{e}(z)$, which implies that $\Phi_{e}(z)$ is single-valued meromorphic in
$\mathbb{C}$. By (\ref{304-111})-(\ref{304-12}), $\Phi_{e}(z)$ is an even
elliptic function. This proves the existence part.
\end{proof}

Since $\Phi_{e}(z)$ have poles at most at $\frac{\omega_{k}}{2}$ with order
$2n_{k}$ and at $\pm p$ with order $2$, we have%
\[
\Phi_{e}(z)=C_{0}+\sum_{k=0}^{3}\sum_{j=0}^{n_{k}-1}b_{j}^{(k)}\wp
(z+\tfrac{\omega_{k}}{2})^{n_{k}-j}+\frac{d}{\wp(z)-\wp(p)},
\]
where $C_{0},b_{j}^{(k)}$ and $d$ are constants depending on $\mathbf{n}%
,A,p,\tau$. By a careful computation, it was proved in \cite{Takemura,Takemura1}
that\medskip

\noindent \textbf{Theorem 2.A.} \cite{Takemura,Takemura1} \emph{After a normalization of multiplying a nonzero constant depending on $\mathbf{n}%
,A,p,\tau$,}%
\begin{equation}\label{even-ell}
\Phi_{e}(z)=C_{0}(A)  +\sum_{k=0}^{3}\sum_{j=0}%
^{n_{k}-1}b_{j}^{(k)}(A)\wp(z+\tfrac{\omega_{k}}{2})^{n_{k}-j}+\frac{d(A)}%
{\wp(z)-\wp(p)},
\end{equation}
\emph{where }$C_{0}(A)=C_{0}(A;p,\tau),$
$b_{j}^{(k)}(A)=b_{j}^{(k)}(A;p,\tau)$\emph{ and }$d(A)=d(A;p,\tau)$\emph{ are
all polynomials of }$A$ \emph{with cofficients being rational functions of
}$\wp(p)$, $\wp^{\prime}(p)$, $e_{k}(\tau)^{\prime}$\emph{s, and they do not
have common zeros, and the leading coefficient of }$C_{0}(A)  $\emph{ can be chosen to be
$\frac{1}{2}$. Moreover, }
\[
g:=\deg_{A}C_{0}(A)>\max \left\{ \deg_{A}b_{j}^{(k)}(A),\deg_{A}d(A)\right\}.
\]

Theorem 2.A will be applied in the proof of Theorems \ref{thm-II-8}-\ref{thm-II-8-0} below.

\subsection{The Hermite-Halphen ansatz}

Let $N=\sum_{k=0}^{3}n_{k}+1$ in this section. For any $\boldsymbol{a}=(  a_{1}
,\cdot \cdot \cdot,a_{N})\in \mathbb{C}^{N}$, we consider the
Hermite-Halphen ansatz
\begin{equation}
y_{\boldsymbol{a}}(z):=\frac{e^{cz}\prod_{i=1}^{N}\sigma(z-a_{i})}%
{\sqrt{\sigma(z-p)\sigma(z+p)}\prod_{k=0}^{3}\sigma(z-\frac{\omega_{k}}%
{2})^{n_{k}}},\text{ }c\in \mathbb{C}. \label{exp1}%
\end{equation}
In Part I \cite{CKL1} we proved
that the common eigen-solution of GLE$(\mathbf{n},A,p)$ must be of the form
$y_{\boldsymbol{a}}(z)$.\medskip

\noindent \textbf{Theorem 2.B.} \cite{CKL1} \textit{Let }$y_{1}(z)$ \textit{be
the common eigen-solution in Lemma \ref{lem6.3}. Then up to a nonzero
constant, }%
\[
y_{1}(z)=y_{\boldsymbol{a}}(z)
\]
\textit{for some }$\boldsymbol{a}=(a_{1},\cdot \cdot
\cdot,a_{N})\in \mathbb{C}^{N}$\textit{ and }$c=c(
\boldsymbol{a})  \in \mathbb{C}$.\medskip

\begin{proof} We sketch the proof here for the reader's convenience. Define
\begin{equation}\label{psipz}
\Psi_{p}(z):=\frac{\sigma(z)}{\sqrt{\sigma(z+p)\sigma(z-p)}}.
\end{equation}
Since $\Psi_{p}(z)^{2}$ is even elliptic and $\ell_{j}$ is chosen to have no
intersection with $L+\Lambda_{\tau}$, we proved in Part I \cite[Lemma 2.2]{CKL1} that $\Psi
_{p}(z)$ is invariant under analytic continuation along $\ell_{j}$, i.e.%
\begin{equation}\label{psii}
\ell_{j}^{\ast}\Psi_{p}(z)=\Psi_{p}(z),\quad j=1,2.
\end{equation}
Since $y_1(z)$ has branch points at $\pm p$, we set
$\tilde{y}(z):=y_1(z)/\Psi_p(z)$.
Then $\tilde{y}(z)$ is
meromorphic, and it follows from (\ref{304-111}) and \eqref{psii} that
\begin{equation}\label{exp1111}
\tilde{y}(z+\omega_i)=\varepsilon_{i}\tilde{y}(z),\quad i=1,2,
\end{equation}%
namely $\tilde{y}(z)$ is \emph{elliptic of the second kind} with periods $1$ and $\tau$. Then a classic theorem says that up to a constant, $\tilde{y}(z)$ can be written as
\begin{equation}\label{ytilde}
\tilde{y}(z)=\frac{e^{cz}\prod_{i=1}^{N}\sigma (z-a_{i})}{\sigma
(z)\prod_{k=0}^{3}\sigma (z-\frac{\omega _{k}}{2})^{n_{k}}},
\end{equation}
for some $\boldsymbol{a}=(a_{1},\cdot \cdot
\cdot,a_{N})\in \mathbb{C}^{N}$ and $c\in\mathbb{C}$,
because $\tilde{y}(z)$ have poles at most at $0$ with order $n_0+1$ and at $\omega _{k}/2$ with order $n_k$, $k=1,2,3$. The constant $c=c(\boldsymbol{a})$ can be determined; see (\ref{61-38}) below.
\end{proof}

\begin{remark} Generically $\{[a_{1}],\cdot \cdot
\cdot, [a_{N}]\}$ is precisely the zero set of $y_1(z)=y_{\boldsymbol{a}}(z)$. For some special $A$'s, the local exponent of $y_1(z)$ at $p$ might be $\frac{3}{2}$, so there are two points in $\{[a_{1}],\cdot \cdot
\cdot, [a_{N}]\}$ being $[p]$, say $[a_{N-1}]=[a_{N}]=[p]$ for example, and in this case the zero set of $y_1(z)$ is contained in $\{[a_{1}],\cdot \cdot
\cdot, [a_{N-2}]\}$. Similarly, $\{[a_{1}],\cdot \cdot
\cdot, [a_{N}]\}$ might contain $\frac{\omega_k}{2}$'s for special $A$'s.
\end{remark}

Although $y_{\boldsymbol{a}}(z)$ is a multi-valued function in $\mathbb{C}$,
$y_{\boldsymbol{a}}(-z)$ can be well-defined as shown in the proof of
Lemma \ref{lem6.3}, and $y_{\boldsymbol{a}}(-z)$ is also a common
eigen-solution. By using the transformation law (let $\eta_3=\eta_1+\eta_2$)
\begin{equation}
\sigma(z+\omega_{k})=-e^{\eta_{k}(z+\frac{\omega_{k}}{2})}\sigma(z),\quad k=1,2,3,
\label{518}%
\end{equation}
it is easy to see that in $\mathbb{C}\backslash(L+\Lambda_{\tau})$,
\begin{equation}
y_2(z)=y_{\boldsymbol{a}}(-z)=y_{-\boldsymbol{a}}(z)\text{ \ up to a nonzero
constant,} \label{a-a}%
\end{equation}
which infers
\begin{equation}
\Phi_{e}(z)=y_{\boldsymbol{a}}(z)y_{-\boldsymbol{a}}(z)\text{ \ up to a
nonzero constant.} \label{a-a1}%
\end{equation}
By the uniqueness of $\Phi_{e}(z)$, we easily see that $\pm \boldsymbol{a}%
\operatorname{mod}\Lambda_{\tau}$ is unique, i.e.%
\begin{equation}
\pm \{[a_{1}],\cdot \cdot \cdot,[a_{N}]\} \text{ \emph{is unique for given GLE}%
}(\mathbf{n},p,A,\tau), \label{aunique}%
\end{equation}
and for different representatives $\boldsymbol{a}, \tilde{\boldsymbol{a}}\in\mathbb{C}^N$ of the same $\{[a_{1}],\cdots,[a_{N}]\}$,
\begin{equation}
y_{\boldsymbol{a}}(z)=y_{\tilde{\boldsymbol{a}}}(z)\text{ \ up to a nonzero
constant.} \label{aunique11}%
\end{equation}

If $y_{\boldsymbol{a}}(z)$ and $y_{-\boldsymbol{a}}(z)$ are linearly
independent, then the monodromy is completely reducible by definition. The
following result shows that the converse assertion also holds, and in this case the monodromy data can be easily computed.

\begin{theorem}
\label{thm0} \cite{CKL1} If the monodromy of GLE$(\mathbf{n},p,A,\tau)$ is completely
reducible, then $y_{\boldsymbol{a}}(z)$ and
$y_{-\boldsymbol{a}}(z)$ are linearly independent and there exists
$(r,s)\in \mathbb{C}^{2}\backslash \frac{1}{2}\mathbb{Z}^{2}$ such
that with respect to $y_{\boldsymbol{a}}(z)$ and $y_{-\boldsymbol{a}}(z)$,
\begin{equation}
N_1=\rho(\ell_{1})=%
\begin{pmatrix}
e^{-2\pi is} & 0\\
0 & e^{2\pi is}%
\end{pmatrix}
,\text{ }N_2=\rho(\ell_{2})=%
\begin{pmatrix}
e^{2\pi ir} & 0\\
0 & e^{-2\pi ir}%
\end{pmatrix},
\label{61.35}%
\end{equation}
and
\begin{equation}
\sum_{i=1}^{N}a_{i}-\sum_{k=1}^{3}\frac{n_{k}\omega_{k}}{2}=r+s\tau, \quad c(\boldsymbol{a})=r\eta_{1}+s\eta_{2}.
\label{61-37}
\end{equation}
Furthermore, if $[a_j]\neq \pm [p]$ for all $j$, then (recall $\eta_3=\eta_1+\eta_2$)
\begin{equation}
c({\boldsymbol a})=\frac{1}{2}\sum_{i=1}^{N}(\zeta
(a_{i}+p)+\zeta(a_{i}-p))-\sum_{k=1}^{3}\frac{n_{k}\eta_{k}}{2}. \label{61-38}
\end{equation}
\end{theorem}

\begin{proof}
This result was proved in Part I \cite{CKL1}. Here we
sketch the proof for later usage.
Let $y_{3}(z)$ be another common eigen-solution which is linearly independent
to $y_{\boldsymbol{a}}(z)$. Clearly $y_{3}(z)y_{3}(-z)$ is also an even
elliptic solution of (\ref{303-1}), so up to nonzero constants,
\begin{equation}
y_{3}(z)y_{3}(-z)=\Phi_{e}(z)=y_{\boldsymbol{a}}(z)y_{-\boldsymbol{a}}(z).
\label{kk--k}%
\end{equation}
Then a zero of $y_{3}(z)$ must be a zero of $y_{-\boldsymbol{a}}(z)$ and vice
versa, so $y_{3}(z)=y_{-\boldsymbol{a}}(z)$ up to a
nonzero constant, namely $y_{\boldsymbol{a}}(z)$ and $y_{-\boldsymbol{a}}(z)$ are linearly independent.

Rewrite%
\begin{equation}
y_{\boldsymbol{a}}(z)=\frac{e^{c(\boldsymbol{a})z}\prod_{j=1}^{N}
\sigma(z-a_{j})}{\sigma(z)\prod_{k=0}^{3}\sigma(z-\tfrac{\omega_{k}}
{2})^{n_{k}}}\cdot \Psi_{p}(z), \label{61-11}
\end{equation}
where $\Psi_{p}(z)$ is defined in (\ref{psipz}).
Then by applying (\ref{psii}) and the transformation law (\ref{518}) to $y_{\boldsymbol{a}}(z)/\Psi_{p}(z)$, we have
\begin{equation}\label{translaw}
\ell_{j}^{\ast}y_{\boldsymbol{a}}(z)=\exp \bigg(  c(\boldsymbol{a})\omega
_{j}-\eta_{j}\bigg(  \sum_{i=1}^{N}a_{i}-\sum_{k=1}^{3}\frac{n_{k}\omega_{k}
}{2}\bigg)  \bigg)  y_{\boldsymbol{a}}(z),\;j=1,2.
\end{equation}
Define $(r,s)\in \mathbb{C}^{2}$ by%
\[
c(\boldsymbol{a})-\eta_{1}\bigg(  \sum_{i=1}^{N}a_{i}-\sum_{k=1}^{3}%
\frac{n_{k}\omega_{k}}{2}\bigg)  =-2\pi is,
\]%
\begin{equation}\label{ca-rs}
c(\boldsymbol{a})\tau-\eta_{2}\bigg(  \sum_{i=1}^{N}a_{i}-\sum_{k=1}^{3}%
\frac{n_{k}\omega_{k}}{2}\bigg)  =2\pi ir.
\end{equation}
Then (\ref{61-37}) follows by using $\tau\eta_{1}-\eta_{2}=2\pi i$. Recalling the eigenvalues
$\varepsilon_{1},\varepsilon_{2}$ in Lemma \ref{lem6.3}, we see from Theorem
2.B and (\ref{translaw})-(\ref{a-a}) that $(\varepsilon_{1},\varepsilon_{2})=(e^{-2\pi
is},e^{2\pi ir})$ and hence (\ref{61.35}) holds. If both $e^{2\pi is}$
and $e^{2\pi ir}\in \{ \pm1\}$, then $y_{\boldsymbol{a}}(z)+y_{-\boldsymbol{a}%
}(z)$ is also a common eigen-solution, and the same argument as (\ref{kk--k})
gives $y_{\boldsymbol{a}}(z)+y_{-\boldsymbol{a}}(z)=c_{\pm}%
y_{\pm \boldsymbol{a}}(z)$ for some constant $c_{\pm}$, a contradiction.
So either $e^{2\pi ir}\not \in \{ \pm1\}$ or $e^{2\pi is}\not \in \{
\pm1\}$, i.e. $(r,s)\not \in \frac{1}{2}\mathbb{Z}^{2}$. Finally, (\ref{61-38}) follows by inserting
(\ref{exp1}) into GLE$({\bf n},p,A)$ and computing the leading terms at
singularities $\pm p$. This completes the proof.
\end{proof}

Now we consider the not completely reducible case.

\begin{theorem}
\label{thm0-1}Suppose the monodromy of GLE$(\mathbf{n}, p, A, \tau)$ is not completely
reducible. Then
\begin{equation}
\{[a_{1}],\cdot \cdot \cdot,[a_{N}]\}=\{-[a_{1}],\cdot \cdot \cdot,-[a_{N}]\},
\label{a=-a}%
\end{equation}
and there exists
$(r,s)\in \frac{1}{2}\mathbb{Z}^{2}$ such that
\begin{equation}\label{61-37-5}
\sum_{i=1}^{N}a_{i}-\sum_{k=1}^{3}\frac{n_{k}\omega_{k}}{2}=r+s\tau,
\quad c(\boldsymbol{a})=r\eta_{1}+s\eta_{2}.
\end{equation}%
Furthermore, there exist linearly independent solutions such that $\rho
(\ell_{1})$ and $\rho(\ell_{2})$ can be expressed as%
\begin{equation}
\rho(\ell_{1})=\varepsilon_{1}%
\begin{pmatrix}
1 & 0\\
1 & 1
\end{pmatrix}
,\text{ \  \  \ }\rho(\ell_{2})=\varepsilon_{2}%
\begin{pmatrix}
1 & 0\\
\mathcal{C} & 1
\end{pmatrix}
, \label{Mono-211}%
\end{equation}
with $\mathcal{C}
\in \mathbb{C}\cup \{ \infty \}$ and
\begin{equation}
(\varepsilon_{1},\varepsilon_{2})=\left \{
\begin{array}
[c]{l}%
(1,1),\text{ if }(r,s)\equiv (0,0)\mod\mathbb{Z}^2,\\
(1,-1),\text{ if }(r,s)\equiv (\frac{1}{2},0)\mod\mathbb{Z}^2,\\
(-1,1),\text{ if }(r,s)\equiv (0,\frac{1}{2})\mod\mathbb{Z}^2,\\
(-1,-1),\text{ if }(r,s)\equiv (\frac{1}{2},\frac{1}{2})\mod\mathbb{Z}^2.
\end{array}
\right.  \label{trace00}%
\end{equation}
Remark that if $\mathcal{C}=\infty$, then
(\ref{Mono-211}) should be understood as%
\begin{equation}
\rho(\ell_{1})=\varepsilon_{1}%
\begin{pmatrix}
1 & 0\\
0 & 1
\end{pmatrix}
,\text{ \  \  \ }\rho(\ell_{2})=\varepsilon_{2}%
\begin{pmatrix}
1 & 0\\
1 & 1
\end{pmatrix}
. \label{Mono-212}%
\end{equation}

\end{theorem}

\begin{proof}
Since the monodromy is not completely reducible and $y_{\pm \boldsymbol{a}}(z)$
are both common eigen-solutions, we have $y_{\boldsymbol{a}}%
(z)=y_{-\boldsymbol{a}}(z)$ up to a nonzero constant, which implies: (1)
$\varepsilon_{j}=$ $\varepsilon_{j}^{-1}$, i.e. $\varepsilon_{j}=\pm1$ for
$j=1,2$; (2) (\ref{a=-a}) holds by using (\ref{exp1}); (3) $\Phi_{e}(z)=y_{\boldsymbol{a}
}(z)^{2}$ up to a nonzero constant. Again by the same argument as (\ref{61-11})-(\ref{ca-rs}), we easily obtain (\ref{61-37-5}) and (\ref{trace00}).

To prove (\ref{Mono-211}), we let $y_{2}(z)$ be a linearly independent solution of GLE (\ref{54-12}) to
$y_{\boldsymbol{a}}(z)$ and define $\chi(z):={y_{2}(z)}/{y_{\boldsymbol{a}}(z)}$.
Then $\chi(z)\not \equiv $const has no branch points, namely $\chi(z)$ is
single-valued meromorphic. Furthermore, inserting $y_{2}(z)=\chi
(z)y_{\boldsymbol{a}}(z)$ into GLE (\ref{54-12}) leads to%
\[
\frac{\chi^{\prime \prime}(z)}{\chi^{\prime}(z)}+2\frac{y_{\boldsymbol{a}%
}^{\prime}(z)}{y_{\boldsymbol{a}}(z)}=0,\text{ i.e. }\chi^{\prime
}(z)=\text{const}\cdot \Phi_{e}(z)^{-1}\text{ is even elliptic.}%
\]
Thus $\chi(z)$ is quasi-periodic, namely there exist two constants $\chi_{1}$
and $\chi_{2}$ such that%
\[
\chi(z+\omega_{j})=\chi(z)+\chi_{j},\text{ \ }j=1,2.
\]
Since $y_{2}(z)$ is not a common eigen-solution, $\chi_{1}$ and $\chi_{2}$ can
not vanish simultaneously. Define%
\begin{equation}
\mathcal{C}:={\chi_{2}}/{\chi_{1}}. \label{mono-C}%
\end{equation}
If $\chi_{1}=0$, then $\chi_{2}\not =0$, $\mathcal{C}=\infty$ and a direct
computation gives
\[
\ell_{1}^{\ast}%
\begin{pmatrix}
\chi_{2}y_{\boldsymbol{a}}(z)\\
y_{2}(z)
\end{pmatrix}
=\varepsilon_{1}%
\begin{pmatrix}
\chi_{2}y_{\boldsymbol{a}}(z)\\
y_{2}(z)
\end{pmatrix}
,
\]%
\[
\ell_{2}^{\ast}%
\begin{pmatrix}
\chi_{2}y_{\boldsymbol{a}}(z)\\
y_{2}(z)
\end{pmatrix}
=\varepsilon_{2}%
\begin{pmatrix}
1 & 0\\
1 & 1
\end{pmatrix}%
\begin{pmatrix}
\chi_{2}y_{\boldsymbol{a}}(z)\\
y_{2}(z)
\end{pmatrix}
,
\]
which is precisely (\ref{Mono-212}). If $\chi_{1}\not =0$, then $\mathcal{C}%
\not =\infty$ and we easily obtain%
\begin{equation}
\ell_{1}^{\ast}%
\begin{pmatrix}
\chi_{1}y_{\boldsymbol{a}}(z)\\
y_{2}(z)
\end{pmatrix}
=\varepsilon_{1}%
\begin{pmatrix}
1 & 0\\
1 & 1
\end{pmatrix}%
\begin{pmatrix}
\chi_{1}y_{\boldsymbol{a}}(z)\\
y_{2}(z)
\end{pmatrix}
, \label{mono-D1}%
\end{equation}%
\begin{equation}
\ell_{2}^{\ast}%
\begin{pmatrix}
\chi_{1}y_{\boldsymbol{a}}(z)\\
y_{2}(z)
\end{pmatrix}
=\varepsilon_{2}%
\begin{pmatrix}
1 & 0\\
\mathcal{C} & 1
\end{pmatrix}%
\begin{pmatrix}
\chi_{1}y_{\boldsymbol{a}}(z)\\
y_{2}(z)
\end{pmatrix}
, \label{mono-D}%
\end{equation}
which is precisely (\ref{Mono-211}). This completes the proof.
\end{proof}

\begin{corollary}
\label{coro4}The monodromy of GLE$(\mathbf{n}, p, A, \tau)$ is completely reducible if and
only if
\begin{equation}
(\text{tr}\rho(\ell_{1}),\text{tr}\rho(\ell_{2}))\not \in \{ \pm
(2,2),\pm(2,-2)\}. \label{II-110}%
\end{equation}
\end{corollary}

\subsection{The monodromy theory for H$(\mathbf{n},B)$}

Now we recall the counterpart of the above monodromy theory for H$(\mathbf{n},B)$ from Part I \cite{CKL1}, the proof of which is simpler due to the absence of singularities $\pm [p]$.
In this section we denote $\tilde{N}=\sum_{k}n_k\geq 1$. By changing variable $z\to z+\frac{\omega_k}{2}$ if necessary, we always assume $n_0\geq 1$.

(i) Any solution of $H({\bf n},B,\tau)$ is meromorphic in $\mathbb{C}$. The corresponding second symmetric product equation
\[\Phi^{\prime \prime \prime }(z;B)-4I_{\mathbf{n}}(z;B,\tau)\Phi^{\prime
}(z;B)-2I_{\mathbf{n}}^{\prime }(z;B,\tau)\Phi(z;B)=0\]
has a unique even elliptic solution $\Phi_{e}(z;B)$ expressed by
\begin{equation}
\Phi_{e}(z;B)=C_{0}(B)  +\sum_{k=0}^{3}\sum
_{j=0}^{n_{k}-1}b_{j}^{(k)}(B)\wp(z+\tfrac{\omega_{k}}{2})^{n_{k}-j}
\label{3rd1}%
\end{equation}
where $C_{0}(B), b_{j}^{(k)}(B)$ are all polynomials in $B$ with $\deg C_{0}>\max_{j,k} \deg b_{j}^{(k)}$ and the leading coefficient of $C_{0}(B)$
being $\frac{1}{2}$. Moreover, $\Phi_{e}(z;B)=y_1(z;B)$ $y_1(-z;B)$, where $y_1(z;B)$ is a common eigenfunction of the monodromy matrices of $H({\bf n},B,\tau)$ and up to a constant, can be written as
\begin{equation}\label{yby}y_1(z;B)=\tilde{y}_{\boldsymbol{a}}(z)
:=\frac{e^{c({\boldsymbol{a}})z}\prod_{i=1}^{\tilde{N}}\sigma(z-a_i)}{\prod_{k=0}^3\sigma(z-\frac{\omega_k}{2})^{n_k}}\end{equation}
with some $\boldsymbol{a}=(a_1,\cdots, a_{\tilde{N}})$ and $c(\boldsymbol{a})\in\mathbb{C}$.
See (\ref{61-381}) for the expression of $c(\boldsymbol{a})$ in the completely reducible case.
By (\ref{yby}) and the transformation law (\ref{518}), it is easy to see that
$y_1(-z;B)=\tilde{y}_{-\boldsymbol{a}}(z)$ up to a sign $(-1)^{n_1+n_2+n_3}$.

(ii) Let $W$ be the Wroskian of $y_1(z;B)$ and $y_1(-z;B)$, then $W^2=Q_{\bf n}(B;\tau)$, where
\[Q_{\bf n}(B;\tau):=\Phi_{e}'(z;B)^2-2\Phi_{e}(z;B)\Phi_{e}^{\prime \prime }(z;B)+4I_{\mathbf{n}%
}(z;B,\tau)\Phi_{e}(z;B)^2\]
is a monic polynomial in $B$ with \emph{odd degree} and independent of $z$.

(iii) The monodromy of H$(\mathbf{n}, B, \tau)$ is completely reducible if and only if $y_1(z;B)=\tilde{y}_{\boldsymbol{a}}(z)$ and $y_1(-z;B)=\tilde{y}_{-\boldsymbol{a}}(z)$ are linearly independent, which is also equivalent to
\begin{equation}\label{a-0a}\{[a_1],\cdots,[a_{\tilde{N}}]\}\cap \{-[a_1],\cdots,-[a_{\tilde{N}}]\}=\emptyset.
\end{equation}
In this case, since $a_j\neq 0$ in $E_{\tau}$ for all $j$ and $n_0\not =0$, we have
\begin{equation}
c(\boldsymbol{a})=\sum_{i=1}^{\tilde{N}}\zeta(a_{i})-\sum_{k=1}^{3}\frac{n_{k}\eta
_{k}}{2}, \label{61-381}%
\end{equation}
which follows by inserting (\ref{yby}) into H$(\mathbf{n},B,\tau)$ and computing the leading terms at the singularity $0$.
Besides, the $(r,s)$ defined by
\begin{equation}
\left \{
\begin{array}
[c]{l}%
\sum_{i=1}^{\tilde{N}}a_{i}-\sum_{k=1}^{3}\frac{n_{k}\omega_{k}}{2}=r+s\tau \\
\sum_{i=1}^{\tilde{N}}\zeta(a_{i})  -\sum_{k=1}^{3}\frac{n_{k}\eta_{k}
}{2}=r\eta_{1}+s\eta_{2}
\end{array}
\right. \label{rs2}%
\end{equation}
satisfies $(r,s)\notin \frac{1}{2}\mathbb{Z}^2$. Furthermore, with respect to $\tilde{y}_{\boldsymbol{a}}(z)$ and $\tilde{y}_{-\boldsymbol{a}}(z)$,
\begin{equation}
N_1=\rho(\ell_{1})=%
\begin{pmatrix}
e^{-2\pi is} & 0\\
0 & e^{2\pi is}%
\end{pmatrix}
,\text{ }N_2=\rho(\ell_{2})=%
\begin{pmatrix}
e^{2\pi ir} & 0\\
0 & e^{-2\pi ir}%
\end{pmatrix}.
\label{61.3512}%
\end{equation}

(iv) For the not completely reducible case, Theorem \ref{thm0-1} and so Corollary \ref{coro4} also hold for H$(\mathbf{n}, B, \tau)$.

\section{GLE and Painlev\'{e} VI equation}

\label{GLE-PVI}

In order to prove Theorem \ref{thm1},
we need to apply the deep connection \cite{Chen-Kuo-Lin} between GLE and Painlev\'{e} VI
equation.
The well-known Painlev\'{e} VI equation with four free parameters
$(\alpha,\beta,\gamma,\delta)$ (denoted by PVI$(\alpha,\beta,\gamma,\delta)$) is written
as
{\allowdisplaybreaks
\begin{align}
\frac{d^{2}\lambda}{dt^{2}}=  &  \frac{1}{2}\left(  \frac{1}{\lambda}+\frac
{1}{\lambda-1}+\frac{1}{\lambda-t}\right)  \left(  \frac{d\lambda}{dt}\right)
^{2}-\left(  \frac{1}{t}+\frac{1}{t-1}+\frac{1}{\lambda-t}\right)
\frac{d\lambda}{dt}\nonumber \\
&  +\frac{\lambda(\lambda-1)(\lambda-t)}{t^{2}(t-1)^{2}}\left[  \alpha
+\beta \frac{t}{\lambda^{2}}+\gamma \frac{t-1}{(\lambda-1)^{2}}+\delta
\frac{t(t-1)}{(\lambda-t)^{2}}\right]  . \label{46}%
\end{align}
}%
Due to its connection with many different disciplines in mathematics and
physics, PVI has been extensively studied in the past several
decades. We refer the readers to the text \cite{GP} for a detailed introduction of PVI.

One of the fundamental properties for PVI is the so-called
\emph{Painlev\'{e} property}, which says that any solution $\lambda(t)$ of
PVI has neither movable branch points nor movable essential
singularities; in other words, for any $t_{0}\in \mathbb{C}\backslash \{0,1\}$,
either $\lambda(t)$ is holomorphic at $t_{0}$ or $\lambda(t)$ has a pole at $t_{0}$. Therefore, it is reasonable to lift PVI to
the universal covering space $\mathbb{H}=\{ \tau|\operatorname{Im}\tau>0\}$ of
$\mathbb{C}\backslash \{0,1\}$ by the following transformation:%
\begin{equation}
t=\frac{e_{3}(\tau)-e_{1}(\tau)}{e_{2}(\tau)-e_{1}(\tau)},\text{ \ }%
\lambda(t)=\frac{\wp(p(\tau)|\tau)-e_{1}(\tau)}{e_{2}(\tau)-e_{1}(\tau)}.
\label{II-130}%
\end{equation}
Then it is known (cf. \cite{Babich-Bordag,Y.Manin}) that $\lambda(t)$ solves PVI if and only if $p(\tau)$ satisfies the
\emph{elliptic form} (\ref{124})
with parameters given by%
\begin{equation}
\left(  \alpha_{0},\alpha_{1},\alpha_{2},\alpha_{3}\right)  =\left(
\alpha,-\beta,\gamma,\tfrac{1}{2}-\delta \right)  . \label{126-0}%
\end{equation}
The Painlev\'{e} property implies that
function $\wp(p(\tau)|\tau)$ is a single-valued meromorphic function in
$\mathbb{H}$. This is an advantage of making the transformation (\ref{II-130}).

\begin{remark}
\label{identify}Clearly for any $m_{1},m_{2}\in \mathbb{Z}$, $\pm p(\tau
)+m_{1}+m_{2}\tau$ is also a solution of the elliptic form (\ref{124}). Since
they all give the same $\lambda(t)$ via (\ref{II-130}), we
always identify all these $\pm p(\tau)+m_{1}+m_{2}\tau$ with the
same one $p(\tau)$.
\end{remark}

Another important feature of PVI is that it is closely related to
the isomonodromy theory of a second order Fuchsian ODE on $\mathbb{CP}^{1}$,
which has five regular singular points $\{0,1,t,\lambda(t),\infty\}$. Among
them, $\lambda(t)$ (as a solution of PVI) is an apparent
singularity. In fact, PVI (\ref{46}) is equivalent to the following
Hamiltonian system%
\begin{equation}
\frac{d\lambda (t)  }{dt}=\frac{\partial K}{\partial \mu},\text{
\ }\frac{d\mu(t)  }{dt}=-\frac{\partial K}{\partial \lambda},
\label{aa}%
\end{equation}
where $K=K(\lambda,\mu,t)$ is given by%
\begin{equation}
K=\frac{1}{t(t-1)}\left \{
\begin{array}
[c]{l}%
\lambda(\lambda-1)(\lambda-t)\mu^{2}+\theta_{0}(\theta_{0}+\theta_{4}%
)(\lambda-t)\\
-\left[
\begin{array}
[c]{l}%
\theta_{1}(\lambda-1)(\lambda-t)+\theta_{2}\lambda(\lambda-t)\\
+(\theta_{3}-1)\lambda(\lambda-1)
\end{array}
\right]  \mu
\end{array}
\right \}  , \label{98}%
\end{equation}
and the relation of parameters is given by%
\begin{equation}
(  \alpha,\beta,\gamma,\delta )  =\left(  \tfrac{1}{2}\theta_{4}%
^{2},\,-\tfrac{1}{2}\theta_{1}^{2},\, \tfrac{1}{2}\theta_{2}^{2},\, \tfrac
{1}{2}\left(  1-\theta_{3}^{2}\right)  \right)  , \label{46-2}%
\end{equation}
\begin{equation}
2\theta_{0}+\theta_{1}+\theta_{2}+\theta_{3}+\theta_{4}=1. \label{46-3}%
\end{equation}
For the Hamiltonian system (\ref{aa}), we consider a second order Fuchsian
differential equation on $\mathbb{CP}^{1}$ as follows:%
\begin{equation}
\frac{d^{2}f}{dx^{2}}+p_{1}(x)\frac{df}{dx}+p_{2}(x)f=0, \label{90}%
\end{equation}
which has five regular singular points at $\{0,1,t,\lambda,\infty \}$ with the
Riemann scheme%
\begin{equation}
\left(
\begin{array}
[c]{ccccc}%
0 & 1 & t & \lambda & \infty \\
0 & 0 & 0 & 0 & \theta_{0}\\
\theta_{1} & \theta_{2} & \theta_{3} & 2 & \theta_{0}+\theta_{4}%
\end{array}
\right)  , \label{91}%
\end{equation}
and $\lambda$ is an apparent singularity. Under these conditions, we have
\begin{equation}
p_{1}(x)=\frac{1-\theta_{3}}{x-t}+\frac{1-\theta_{1}}{x}+\frac{1-\theta_{2}%
}{x-1}-\frac{1}{x-\lambda}, \label{96}%
\end{equation}%
\begin{equation}
p_{2}(x)=\frac{\theta_{0}\left(  \theta_{0}+\theta_{4}\right)  }{x(x-1)}%
-\frac{t(t-1)K}{x(x-1)(x-t)}+\frac{\lambda(\lambda-1)\mu}{x(x-1)(x-\lambda)},
\label{97}%
\end{equation}
where $K=K(\lambda,\mu,t)$ is given by (\ref{98}); see e.g. \cite{GP}. The
following result was proved in \cite{Fuchs, Okamoto2}: \emph{Suppose that
}$\theta_{1},\theta_{2},\theta_{3},\theta_{4}\notin \mathbb{Z}$ \emph{and
}$\lambda$\emph{\ is an apparent singularity of (\ref{90}). Then (\ref{90}) is
monodromy preserving as }$t$\emph{\ deforms if and only if }$\left(
\lambda(t),\mu(t)\right)  $\emph{\ satisfies the Hamiltonian system
(\ref{aa}). In particular, }$\lambda(t)$ \emph{is a solution of PVI
(\ref{46}).}

On the other hand, there are works studying the isomonodromic deformation on elliptic curves and its Hamiltonian structure; see e.g. \cite{Kawai} and references therein.
Recently, we \cite{Chen-Kuo-Lin} developed an analogous isomonodromy theory
for the elliptic form (\ref{124}). First we proved that the elliptic form
(\ref{124}) is equivalent to the new Hamiltonian system (\ref{142-0}).
Then we proved that this Hamiltonian system governs the isomonodromic
deformation of GLE$(\mathbf{n},$ $p(\tau),$ $A(\tau),\tau)$.\medskip

\noindent \textbf{Theorem 3.A.} \cite{Chen-Kuo-Lin} \emph{GLE$(\mathbf{n}
,p(\tau),A(\tau),\tau)$ with $p(\tau)$ being
an apparent singularity is monodromy preserving as $\tau$ deforms if
and only if $(p(\tau),A(\tau))$ satisfies the Hamiltonian system
(\ref{142-0}). In particular, $p(\tau)$ is a solution of the elliptic
form (\ref{124}) with parameter (\ref{125})}.\medskip

Remark that Theorem 3.A holds for
any $n_{k}\in \mathbb{C}\setminus(\frac{1}{2}+\mathbb{Z})$ (i.e. non-resonant
condition), but we only consider $n_{k}\in \mathbb{Z}_{\geq0}$ in this
paper.

Given any solution $p(\tau)$ of the elliptic form (\ref{124}) with parameter
(\ref{125}), we define $A(\tau)$ by the first equation of (\ref{142-0}). Then
for any $\tau$ such that $p(\tau)\not \in E_{\tau}[2]$, $A(\tau)$ is finite
and so GLE$(\mathbf{n},p(\tau),A(\tau),\tau)$ is well-defined, which is called
\emph{the associated GLE} of $p(\tau)$ in this paper.

In view of Theorem 3.A and the monodromy theory of GLE discussed in Section
2, we give the following definition for convenience.

\begin{definition}\label{definition1}
A solution $p(\tau)$ of the elliptic form (\ref{124}) with parameter
(\ref{125}) is called a completely reducible solution if the monodromy of the
associated GLE$(\mathbf{n}$, $p(\tau),A(\tau),\tau)$ is completely
reducible; otherwise, $p(\tau)$ is called a not completely reducible solution.
\end{definition}

A natural problem is \emph{how to classify (not) completely reducible
solutions $p(\tau)$ in terms of the global monodromy data of the associated GLE$(\mathbf{n},p(\tau),A(\tau),\tau)$}. This is crucial for
us to prove Theorem \ref{thm1}. In Sections \ref{Hitchin-case}-\ref{General-case-Ok}, we answer this
question for the special case $\mathbf{n}=\boldsymbol{0}$, i.e. $n_{k}=0$ for
all $k$ and the general case $\mathbf{n}\not =\boldsymbol{0}$, respectively.

\section{The special case $\mathbf{n}=\boldsymbol{0}$}
\label{Hitchin-case}

Note from (\ref{125}) that $\alpha_{k}=\frac{1}{8}$ for all $k$ if $\mathbf{n}=\boldsymbol{0}$.
This section is devoted to the classification of all solutions of EPVI$(\frac{1}{8},\frac{1}{8},\frac{1}{8},\frac{1}{8})$
\begin{equation}
\frac{d^{2}p(\tau)}{d\tau^{2}}=\frac{-1}{32\pi^{2}}\sum_{k=0}^{3}\wp^{\prime
}\left(  \left.  p(\tau)+\frac{\omega_{k}}{2}\right \vert
\tau \right)  ,\label{hit}%
\end{equation}
or equivalently PVI$(\frac{1}{8},\frac{-1}{8},\frac{1}{8},\frac{3}{8})$, in terms of the global monodromy data of the associated GLE$(\mathbf{0},p(\tau),A(\tau),\tau)$. PVI$(\frac{1}{8},\frac{-1}{8},\frac{1}{8},\frac{3}{8})$ was first studied by Hitchin \cite{Hit1} and later by Takemura \cite{Takemura}. Therefore, part of the results in this section do overlap with the existing literature.
However, there are a number of issues which we were unable to locate
satisfactory in the literature. Here we attempt to provide a self-contained
account of solutions of PVI$(\frac{1}{8},\frac{-1}{8},\frac{1}{8},\frac{3}{8})$ for later usage in Section \ref{General-case-Ok}.

First we recall Hitchin's famous formula.
For any $(r,s)  \in \mathbb{C}^{2}\backslash \frac
{1}{2}\mathbb{Z}^{2}$, let $p_{r,s}^{\mathbf{0}}(\tau
)$ be defined by
\begin{equation}
\wp(p_{r,s}^{\mathbf{0}}(\tau)|\tau):=\wp(r+s\tau|\tau)+\frac{\wp^{\prime
}(r+s\tau|\tau)}{2(\zeta(r+s\tau|\tau)-r\eta_{1}(\tau)-s\eta_{2}%
(\tau))  }. \label{513-1}%
\end{equation}
In \cite{Hit1} Hitchin proved the following remarkable result for PVI$(\frac{1}{8},\frac{-1}{8},\frac{1}{8},\frac{3}{8})$.

\medskip

\noindent \textbf{Theorem 4.A. \cite{Hit1}} \textit{For any $(
r,s)  \in \mathbb{C}^{2}\backslash \frac{1}{2}\mathbb{Z}^{2}$\textit{,
}$p_{r,s}^{\mathbf{0}}(\tau)$ given by (\ref{513-1}) is a solution to
EPVI$(\frac{1}{8},\frac{1}{8},\frac{1}{8},\frac{1}{8})$;
or equivalently, $\lambda_{r,s}^{\mathbf{0}}(t):=
\frac{\wp(p_{r,s}^{\mathbf{0}}(\tau)|\tau)-e_{1}(  \tau)  }%
{e_{2}(\tau)-e_{1}(\tau)}$ via (\ref{513-1}) is a solution to PVI$(\frac{1}{8},$ $\frac{-1}{8},\frac{1}{8},\frac{3}{8})$}.\medskip

The following result shows that $p_{r,s}^{\mathbf{0}}(\tau)$ represents the completely reducible solutions in the sense of Definition \ref{definition1}.

\begin{theorem}
\label{thm5}Suppose $p^{\mathbf{0}}(\tau)$ is a solution of
(\ref{hit}). Then

\begin{itemize}
\item[(i)] $p^{\mathbf{0}}(\tau)$ is completely reducible if and
only if there is a complex pair $(r,s)\in \mathbb{C}^{2}\backslash \frac{1}%
{2}\mathbb{Z}^{2}$ such that $p^{\mathbf{0}}(\tau)=p_{r,s}%
^{\mathbf{0}}(\tau)$ given by (\ref{513-1}). In this case, the
monodromy of the associated GLE$(\mathbf{0},p^{\mathbf{0}}(\tau),A(\tau),\tau)$ satisfies (\ref{61.35}).

\item[(ii)]
$\wp(p_{r_{1},s_{1}}^{\mathbf{0}}(\tau)|\tau)\equiv \wp(p_{r_{2},s_{2}%
}^{\mathbf{0}}(\tau)|\tau)\text{ }\Longleftrightarrow \text{ }(
r_{1},s_{1})  \equiv \pm (r_{2},s_{2})  \operatorname{mod}%
\mathbb{Z}^{2}.$
\end{itemize}
\end{theorem}

\begin{proof}
(i) Take $\tau_{0}\in \mathbb{H}$ such that $p^{\mathbf{0}}(\tau)$
$\not \in E_{\tau}[2]$ in a neighborhood $U$ of $\tau_{0}$. We only need to
prove $p^{\mathbf{0}}(\tau)=p_{r,s}%
^{\mathbf{0}}(\tau)$ in a neighborhood $U$ for some $(r,s)\notin\frac{1}{2}\mathbb{Z}^2$ and then the result follows by analytic
continuation.

First we prove the necessary part. Since $p^{\mathbf{0}}(\tau)$
is completely reducible, the associated GLE$(\mathbf{0},p^{\mathbf{0}}%
(\tau),A(\tau),\tau)$ is well-defined in $U$ and preserves its completely
reducible monodromy for $\tau \in U$. Then by Theorem \ref{thm0} and
(\ref{aunique})-(\ref{aunique11}), there exists $(r,s)\in\mathbb{C}^{2}\backslash \frac
{1}{2}\mathbb{Z}^{2}$ independent of $\tau$ such that
\begin{equation}
y_{a_{1}(\tau)}(  z)  =\frac{e^{c(\tau)z}\sigma(  z-a_{1}%
(\tau))  }{\sqrt{\sigma(z-p^{\mathbf{0}}(\tau))\sigma(z+p^{\mathbf{0}%
}(\tau))}} \label{yac}%
\end{equation}
is a solution to GLE$(\mathbf{0},p^{\mathbf{0}}(\tau),A(\tau),\tau)$, where%
\begin{equation}
a_{1}(\tau)=r+s\tau, \label{61-370}%
\end{equation}%
\begin{align}
c(\tau)&=r\eta_{1}(\tau)+s\eta_{2}(\tau)\nonumber\\
&=\tfrac{1}{2}\left[  \zeta(a_{1}%
(\tau)+p^{\mathbf{0}}(\tau))+\zeta(a_{1}(\tau)-p^{\mathbf{0}}(\tau))\right]  .
\label{kkkll}%
\end{align}
Here $[a_1(\tau)]\neq \pm [p^{\mathbf{0}}(\tau)]$ because the local exponents are $\frac{-1}{2}, \frac{3}{2}$ at $\pm p^{\mathbf{0}}(\tau)$.  Applying the addition formula%
\begin{equation}
\zeta(u+v)+\zeta(u-v)-2\zeta(u)=\frac{\wp^{\prime}(u)}{\wp(u)-\wp(v)},
\label{81}%
\end{equation}
it is easy to see that the second equality in (\ref{kkkll}) is equivalent to
\begin{equation}
\wp \left(  p^{\mathbf{0}}(\tau)|\tau \right)  =\wp(r+s\tau|\tau)+\frac
{\wp^{\prime}(r+s\tau|\tau)}{2(\zeta \left(  r+s\tau|\tau \right)  -r\eta
_{1}(\tau)-s\eta_{2}(\tau))}, \label{a}%
\end{equation}
i.e. $\wp(  p^{\mathbf{0}}(\tau)|\tau)  =\wp(  p_{r,s}%
^{\mathbf{0}}(\tau)|\tau)  $ for $\tau \in U$. This proves $p^{\mathbf{0}%
}(\tau)=p_{r,s}^{\mathbf{0}}(\tau)$ by Remark \ref{identify}.

Next we prove the sufficient part. Since $p^{\mathbf{0}}(\tau)$ $=p_{r,s}^{\mathbf{0}}(\tau)$, the
above argument shows the validity of the second equality of (\ref{kkkll}) by
defining $a_{1}(\tau)$ $=r+s\tau$. Since $(r,s)\not \in \frac{1}{2}%
\mathbb{Z}^{2}$, we may assume $a_{1}(\tau)$ $\not \in $ $E_{\tau}[2]$ and
hence $a_{1}(\tau)\not \equiv \pm p^{\mathbf{0}}(\tau)\operatorname{mod}\Lambda_{\tau}$
for $\tau \in U$. Then we define $c(\tau)$ by (\ref{kkkll}) and $y_{a_{1}%
(\tau)}(z)$ by (\ref{yac}) in $U$. Consequently, a direct computation shows that
$y_{a_{1}(\tau)}(z)$ is a solution to GLE$(\mathbf{0},$ $p^{\mathbf{0}}%
(\tau),$ $\tilde{A}(\tau),\tau)$ with%
\begin{equation}
\tilde{A}(\tau):=\tfrac{1}{2}\left[  \zeta(a_{1}(\tau)+p^{\mathbf{0}}%
(\tau))-\zeta(a_{1}(\tau)-p^{\mathbf{0}}(\tau))-\zeta(2p^{\mathbf{0}}%
(\tau))\right]  . \label{Aap}%
\end{equation}
Indeed, since
\[\frac{y_{a_{1}}'(z)}{y_{a_{1}}(z)}=c(\tau)+\zeta(z-a_1)-\tfrac{1}{2}[\zeta(z+p^{\mathbf{0}})
+\zeta(z-p^{\mathbf{0}})],\]
\[\left(\frac{y_{a_{1}}'(z)}{y_{a_{1}}(z)}\right)'=-\wp(z-a_1)+\tfrac{1}{2}
[\wp(z+p^{\mathbf{0}})
+\wp(z-p^{\mathbf{0}})],\]
are all elliptic functions, we have
\begin{align*}
\frac{y_{a_{1}}''(z)}{y_{a_{1}}(z)}&=\left(\frac{y_{a_{1}}'(z)}{y_{a_{1}}(z)}\right)'+
\left(\frac{y_{a_{1}}'(z)}{y_{a_{1}}(z)}\right)^2\\
&=\tfrac{3}{4}
[\wp(z+p^{\mathbf{0}})
+\wp(z-p^{\mathbf{0}})]+\tilde{A}[\zeta(z+p^{\mathbf{0}})
-\zeta(z-p^{\mathbf{0}})]+\tilde{B},
\end{align*}
with some $\tilde{B}\in\mathbb{C}$ and
$\tilde{A}=-c(\tau)+\zeta(p^{\mathbf{0}}+a_1)-\tfrac{1}{2}\zeta(2p^{\mathbf{0}})$,
i.e. (\ref{Aap}) holds by using the second equality of (\ref{kkkll}).

By (\ref{kkkll}) and $a_{1}(\tau)$ $=r+s\tau$, the same
argument as Theorem \ref{thm0} implies that (\ref{61.35}) holds with respect
to $y_{a_{1}(\tau)}(z)$ and $y_{-a_{1}(\tau)}(z)$, i.e. the monodromy of
GLE$(\mathbf{0},$ $p^{\mathbf{0}}(\tau),$ $\tilde{A}(\tau),\tau)$ is
completely reducible and preserves for $\tau \in U$. Then Theorem 3.A implies that $(p^{\mathbf{0}}%
(\tau),$ $\tilde{A}(\tau))$ satisfies the Hamiltonian system (\ref{142-0}),
namely $\tilde{A}(\tau)=A(\tau)$ and so the monodromy of the associated
GLE$(\mathbf{0},$ $p^{\mathbf{0}}(\tau),$ $A(\tau),\tau)$ of $p^{\mathbf{0}}(\tau)$ is completely
reducible. This proves that $p^{\mathbf{0}}(\tau)$ is a completely reducible solution.

(ii) The sufficient part is trivial so we prove the necessary part. Suppose
$\wp(p_{r_{1},s_{1}}^{\mathbf{0}}(\tau)|\tau)$ $\equiv$ $\wp(p_{r_{2},s_{2}%
}^{\mathbf{0}}(\tau)|\tau)$. Take $\tau_{0}\in \mathbb{H}$ such that
$p_{r_{i},s_{i}}^{\mathbf{0}}(\tau)$ $\not \in E_{\tau}[2]$, $i=1,2,$ in a
neighborhood $U$ of $\tau_{0}$. Then $p_{r_{1},s_{1}}^{\mathbf{0}}(\tau)$
$=\pm p_{r_{2},s_{2}}^{\mathbf{0}}(\tau)$ $+m$ $+n\tau$ for $\tau \in U$. Let
$A_{i}(\tau)$ be defned by the first equation of the Hamiltonian system
(\ref{142-0}), then $A_{1}(\tau)=\pm A_{2}(\tau)$. Together with
(\ref{GLE-invariant}), we conclude that these two associated GLE($\mathbf{0},$
$p_{r_{i},s_{i}}^{\mathbf{0}}(\tau),$ $A_{i}(\tau),\tau$) must be the same.
Consequently, it follows from the assertion (i) that%
\[
e^{2\pi is_{1}}=e^{\pm2\pi is_{2}}\text{ and }e^{2\pi ir_{1}}=e^{\pm2\pi
ir_{2}},
\]
which is precisely $\left(  r_{1},s_{1}\right)  \equiv \pm \left(  r_{2}%
,s_{2}\right)  \operatorname{mod}\mathbb{Z}^{2}$. The proof is complete.
\end{proof}

Next we study the not completely reducible solutions of
EPVI$(\frac{1}{8},\frac{1}{8},\frac{1}{8},\frac{1}{8})$. Recall (\ref{98}) that
the corresponding Hamiltonian $K=K(\lambda,\mu,t)$ is given by%
\begin{equation}
K=\frac{1}{t(t-1)}\left \{  \lambda(\lambda-1)(\lambda-t)\mu^{2}-\tfrac{1}%
{2}(\lambda^{2}-2t\lambda+t)\mu \right \}  . \label{98-0}%
\end{equation}
In general, solutions of PVI($\frac{1}{8},\frac{-1}{8},\frac{1}{8},\frac{3}{8}$) might also come from Riccati equations. It is easy to see from
(\ref{98-0}) that the Hamiltonian system (\ref{aa}) has four families of
solutions $(\lambda(t),\mu(t))$, where $\lambda(t)$ satisfies four different
Riccati equations as follows:%
\begin{equation}
\frac{d\lambda}{dt}=-\frac{1}{2t(t-1)}(\lambda^{2}-2t\lambda+t),\quad
\mu \equiv0; \label{1000}%
\end{equation}%
\begin{equation}
\frac{d\lambda}{dt}=\frac{1}{2t(t-1)}(\lambda^{2}-2\lambda+t),\quad
\mu \equiv \frac{1}{2\lambda}; \label{1001}%
\end{equation}%
\begin{equation}
\frac{d\lambda}{dt}=\frac{1}{2t(t-1)}(\lambda^{2}-t),\quad\mu \equiv \frac
{1}{2(\lambda-1)}; \label{1002}%
\end{equation}%
\begin{equation}
\frac{d\lambda}{dt}=\frac{1}{2t(t-1)}(\lambda^{2}+2(t-1)\lambda-t),\quad
\mu \equiv \frac{1}{2(\lambda-t)}. \label{1003}%
\end{equation}

\begin{theorem}
\label{thm3.3}Suppose $p(\tau)$ is a solution of
EPVI$(\frac{1}{8},\frac{1}{8},\frac{1}{8},\frac{1}{8})$. Then $p(\tau)$ is not completely reducible if and only if the
corresponding solution $\lambda(t)$ (via (\ref{II-130})) of PVI$(\frac{1}{8},\frac{-1}{8},\frac{1}{8},\frac{3}{8})$ solves one of the four
Riccati equations (\ref{1000})-(\ref{1003}).
\end{theorem}

\begin{proof}
Let $p(\tau)$ be a solution of the elliptic form (\ref{hit}). We can take
$\tau_{0}\in \mathbb{H}$ such that
\begin{equation}
\lbrack p(\tau)]\not \in E_{\tau}[2]\text{ and }A(\tau)\text{ is finite in a
neighborhood }U\text{ of }\tau_{0}, \label{1113}%
\end{equation}
namely the associated GLE$(\mathbf{0},p(\tau),A(\tau),\tau)$ is
well-defined and preserves the monodromy for $\tau \in U$. Recalling (\ref{Aap}), we let $\pm a_{1}(\tau)$ be
defined by
\begin{equation}
A(\tau)=\frac{1}{2}[\zeta(a_{1}(\tau)+p(\tau))-\zeta(a_{1}(\tau)-p(\tau
))-\zeta(2p(\tau))],\tau \in U. \label{135-2}
\end{equation}
Then (\ref{1113}) gives
\begin{equation}
\lbrack a_{1}(\tau)]\not =\pm \lbrack p(\tau)],\text{ }\tau \in U. \label{1112}%
\end{equation}
Consequently, the same argument as that in the proof of Theorem \ref{thm5}-(i)
shows that%
\[
y_{\pm a_{1}(\tau)}(z)  =\frac{e^{\pm c(\tau)z}\sigma (  z\mp
a_{1}(\tau))  }{\sqrt{\sigma(z-p(\tau))\sigma(z+p(\tau))}}%
\]
with%
\[
c(\tau)=\frac{1}{2}\left[  \zeta(a_{1}(\tau)+p(\tau))+\zeta(a_{1}(\tau
)-p(\tau))\right]
\]
are both solutions of GLE$(\mathbf{0},p(\tau),A(\tau),\tau)$. By Theorem \ref{thm0}, the
monodromy is not completely reducible if and only if $y_{a_{1}(\tau)}(
z)$ and $y_{-a_{1}(\tau)}(z)$ are linearly dependent,
which is equivalent to $a_{1}(\tau)\equiv-a_{1}(\tau)\operatorname{mod}%
\Lambda_{\tau}$, i.e.
\begin{equation}
\left[  a_{1}(\tau)\right]  =[\tfrac{\omega_{k}}{2}]\text{ for }\tau \in U\text{
and some }k\in \{0,1,2,3\}. \label{1115}%
\end{equation}

On the other hand, by the
addition formula (\ref{81}) and $\frac{\wp^{\prime \prime}(p)}{2\wp^{\prime
}(p)}=$ $\zeta(2p)-2\zeta(p)$, we can rewrite (\ref{135-2}) as
\begin{equation}
A(\tau)=\frac{\wp^{\prime}(p(\tau))}{2\left[  \wp(p(\tau))-\wp(a_{1}%
(\tau))\right]  }-\frac{\wp^{\prime \prime}(p(\tau))}{4\wp^{\prime}(p(\tau))}.
\label{135-10}%
\end{equation}
Recall that $\lambda(t)$ defined via (\ref{II-130}) is a solution of
PVI$(\frac{1}{8},\frac{-1}{8},\frac{1}{8},\frac{3}{8})$. Then by defining
$\mu(t)$ via the first equation of the Hamiltonian system (\ref{aa}),
$(\lambda(t),\mu(t))$ satisfies the Hamiltonian system (\ref{aa}). It follows
from (\ref{II-106}) below that the relation of $\mu(t)$ and $A(\tau)$ is given
by%
\begin{equation}
\mu(t(\tau))=\frac{1}{8}\frac{\mathfrak{p}^{\prime}(\lambda)}{\mathfrak{p}%
(\lambda)}+\frac{A\wp^{\prime}(p)}{\left(  e_{2}(\tau)-e_{1}(\tau)\right)
^{2}\mathfrak{p}(\lambda)}, \label{bb}%
\end{equation}
where
\begin{equation}
\mathfrak{p}(x)=4x(x-1)(x-t). \label{123}%
\end{equation}
Notice from (\ref{123}), (\ref{II-130}) and $\wp^{\prime}(z)^{2}=4\prod
_{k=1}^{3}(\wp(z)-e_{k})$ that {\allowdisplaybreaks%
\[
\mathfrak{p}(\lambda(t))=\frac{\wp^{\prime}(p(\tau))^{2}}{(  e_{2}%
(\tau)-e_{1}(\tau))  ^{3}},\quad\mathfrak{p}^{\prime}(\lambda
(t))=\frac{2\wp^{\prime \prime}(p(\tau))}{(  e_{2}(\tau)-e_{1}%
(\tau))  ^{2}}.
\]
}Inserting these and (\ref{135-10}) into (\ref{bb}), we easily
obtain{\allowdisplaybreaks%
\begin{align}
\mu(t)  &  =\frac{\left(  e_{2}(\tau)-e_{1}(\tau)\right)  \left(  4A(\tau
)\wp^{\prime}(p(\tau))+\wp^{\prime \prime}(p(\tau))\right)  }{4\wp^{\prime
}(p(\tau))^{2}}\nonumber \\
&  =\frac{e_{2}(\tau)-e_{1}(\tau)}{2\left[  \wp(p(\tau))-\wp(a_{1}%
(\tau))\right]  }. \label{135}%
\end{align}
Remark that (\ref{135}) always holds no matter with whether }$p(\tau)$ is a
completely reducible solution or not.

Recall that the monodromy is not completely reducible if and only if
(\ref{1115}) holds. By (\ref{II-130}) and (\ref{135}), this is equivalent to
\begin{equation}
\mu(t)=\left \{
\begin{array}
[c]{l}%
0,\text{ if }k=0,\\
\frac{1}{2\lambda(t)},\text{ if }k=1,\\
\frac{1}{2(\lambda(t)-1)},\text{ if }k=2,\\
\frac{1}{2(\lambda(t)-t)},\text{ if }k=3,
\end{array}
\right.  \text{ in a neighborhood of }t(\tau_{0})\text{,} \label{kakak}%
\end{equation}
namely one of (\ref{1000})-(\ref{1003}) holds after the analytic continuation.
The proof is complete.
\end{proof}

Now we want to find the expression of a not completely reducible solution
$p(\tau)$. Assume $[  a_{1}]  =[\frac{\omega_{k}}{2}]\in E_{\tau
}[2]$ by (\ref{1115}), and recall (\ref{135-2}) that%
\begin{equation}
A(\tau)=\frac{1}{2}[\zeta(\tfrac{\omega_{k}}{2}+p(\tau))-\zeta(\tfrac{\omega_{k}}{2}-p(\tau
))-\zeta(2p(\tau))]. \label{44}%
\end{equation}
By using $2\zeta(z)-\zeta(2z)=-\frac{1}{2}\frac{\wp^{\prime \prime}(z)}%
{\wp^{\prime}(z)}$, (\ref{44}) is equivalent to%
\begin{equation}
A(\tau)=-\frac{1}{4}\frac{\wp^{\prime \prime}(p(\tau)-\frac{\omega_{k}}{2})}{\wp^{\prime
}(p(\tau)-\frac{\omega_{k}}{2})}. \label{43-2}%
\end{equation}
As in Theorem \ref{thm0-1}, we let
\begin{equation}
y_{1}(z)=y_{a_{1}}(z)=\frac{e^{\frac{1}{2}[  \zeta(a_{1}+p)+\zeta
(a_{1}-p)]  z}\sigma(  z-a_{1})  }{\sqrt{\sigma
(z-p)\sigma(z+p)}},\quad a_1=\tfrac{\omega_k}{2}, \label{43}%
\end{equation}
and $y_{2}(z)=\chi(z)y_{1}(z)$ be linearly independent solutions of the
associated GLE$(\mathbf{0},p(\tau),A(\tau),\tau)$, where
\begin{equation}
\chi^{\prime}(z)=\text{const}\cdot y_{1}(z)^{-2}. \label{43-1}%
\end{equation}
Define%
\begin{equation}
(\varepsilon_{k,1},\varepsilon_{k,2})=\left \{
\begin{array}
[c]{l}%
(1,1),\text{ if }k=0,\\
(1,-1),\text{ if }k=1,\\
(-1,1),\text{ if }k=2,\\
(-1,-1),\text{ if }k=3.
\end{array}
\right.  \label{trace}%
\end{equation}

First we consider the case $[a_{1}]=[0]$. Then $y_1(z)=\frac{\sigma(z)}{\sqrt{\sigma(z-p)\sigma(z+p)}}=\Psi_{p}(z)$ (see Theorem \ref{thm0} for $\Psi_{p}(z)$) and
\[
y_{1}(z)^{-2}=\frac{\sigma(z+p)\sigma(z-p)}{\sigma(z)^{2}}=c(\wp(z)-\wp(p)),\;c\neq 0.
\]
So (\ref{43-1}) yields that we can take $\chi(z)=\zeta(z)+\wp(p)z$, namely for any $c(\tau)\not =0$, $(c(\tau)y_{1},y_{2})$ with $y_{2}
(z)=(\zeta(z)+\wp(p)z)y_{1}(z)$ is a fundamental system of
solutions to GLE$(\mathbf{0},p(\tau),$ $A(\tau),\tau)$. In particular, (\ref{psii}) implies
\begin{equation}
\ell_{j}^{\ast}%
\begin{pmatrix}
c(\tau)y_{1}\\
y_{2}%
\end{pmatrix}
=%
\begin{pmatrix}
1 & 0\\
\frac{\eta_{j}+\wp(p)\omega_{j}}{c(\tau)} & 1
\end{pmatrix}%
\begin{pmatrix}
c(\tau)y_{1}\\
y_{2}%
\end{pmatrix}
,\text{ \ }j=1,2. \label{43-3}%
\end{equation}

\begin{proposition}
\label{Prop-II-8}The solutions of the Riccati equation (\ref{1000}) can be
parameterized by $\mathcal{C}\in \mathbb{C}\mathbb{P}^{1}:$%
\begin{equation}
\lambda_{0,\mathcal{C}}^{\mathbf{0}}(t)=\frac{\wp(p_{0,\mathcal{C}%
}^{\mathbf{0}}(\tau)|\tau)-e_{1}(\tau)}{e_{2}(\tau)-e_{1}(\tau)},\text{ }%
\wp(p_{0,\mathcal{C}}^{\mathbf{0}}(\tau)|\tau)=\frac{\eta_{2}(\tau
)-\mathcal{C}\eta_{1}(\tau)}{\mathcal{C}-\tau}. \label{43-4}%
\end{equation}
Moreover, the monodromy of the associated GLE satifies%
\begin{equation}
\rho(\ell_{1})=%
\begin{pmatrix}
1 & 0\\
1 & 1
\end{pmatrix}
,\quad\rho(\ell_{2})=%
\begin{pmatrix}
1 & 0\\
\mathcal{C} & 1
\end{pmatrix}
. \label{m1}%
\end{equation}
Here when $\mathcal{C}=\infty$, it should be understand as%
\begin{equation}
\rho(\ell_{1})=I_{2},\quad\rho(\ell_{2})=%
\begin{pmatrix}
1 & 0\\
1 & 1
\end{pmatrix}
. \label{m1-0}%
\end{equation}

\end{proposition}

\begin{proof}
In this proof, we omit $\boldsymbol{0},0$ in the notations.

\textbf{Step 1.} We prove that for any constant $\mathcal{C}\in \mathbb{C}%
\mathbb{P}^{1}$, $\lambda_{\mathcal{C}}(t)$ given by (\ref{43-4}) solves the
Riccati equation (\ref{1000}).

Fix any $\mathcal{C}\in \mathbb{C}\mathbb{P}^{1}$ and let $p(\tau
)=p_{\mathcal{C}}(\tau)$, $A(\tau)=-\frac{1}{4}\frac{\wp^{\prime \prime}%
(p(\tau))}{\wp^{\prime}(p(\tau))}$ in GLE$(\mathbf{0},p(\tau),$ $A(\tau
),\tau)$. If $\mathcal{C}=\infty$, then $\wp(p(\tau))=-\eta_{1}(\tau)$. Choose
$c(\tau)=\eta_{2}(\tau)+\wp(p(\tau))\tau$. By the Legendre relation $\tau
\eta_{1}(\tau)-\eta_{2}(\tau)=2\pi i$ we have $c(\tau)=-2\pi i$. Thus by
(\ref{43-3}), we obtain (\ref{m1-0}). That is, GLE$(\mathbf{0},p(\tau),$
$A(\tau),\tau)$ is monodromy preserving as $\tau$ deforms, so $p(\tau)=p_{\infty}%
(\tau)$ is a solution of EPVI$(\frac{1}{8},\frac{1}{8},\frac{1}{8},\frac{1}{8})$.

If $\mathcal{C}\not =\infty$, then (\ref{43-4}) gives $\eta_{1}(\tau
)+\wp(p(\tau))\not \equiv 0$ and $\mathcal{C}=\frac{\eta_{2}(\tau)+\wp
(p(\tau))\tau}{\eta_{1}(\tau)+\wp(p(\tau))}$. Choose $c(\tau)=\eta_{1}%
(\tau)+\wp(p(\tau))$. Clearly except a set of discrete points in $\mathbb{H}$,
$c(\tau)\not =0$ and so (\ref{43-3}) gives (\ref{m1}). Again we conclude
that $p(\tau)=p_{\mathcal{C}}(\tau)$ is a solution of EPVI$(\frac{1}{8},\frac{1}{8},\frac{1}{8},\frac{1}{8})$.
Formula (\ref{43-4}) can be found in \cite{Hit1,Takemura}. Here together with
$a_{1}=0$ and (\ref{kakak}), we note that $\lambda_{\mathcal{C}%
}(t)$ actually solves the Ricatti equation (\ref{1000}).

\textbf{Step 2.} Let $\lambda(t)$ be any solution of the Riccati equation
(\ref{1000}). We prove the existence of $\mathcal{C}\in \mathbb{C}%
\mathbb{P}^{1}$ such that $\lambda(t)=\lambda_{\mathcal{C}}(t)$.

Define $\pm[p(\tau)]$ by $\lambda(t)$ via (\ref{II-130}) and
$A(\tau)=-\frac{1}{4}\frac{\wp^{\prime \prime}(p(\tau))}{\wp^{\prime}(p(\tau
))}$. Then $p(\tau)$ is a solution of EPVI$(\frac{1}{8},\frac{1}{8},\frac{1}{8},\frac{1}{8})$ and
the associated GLE$(\mathbf{0}, p(\tau), A(\tau), \tau)$ is monodromy preserving as $\tau$ deforms.
So there exists a fundamental system of solutions $(\tilde{y}%
_{1}(z;\tau),\tilde{y}_{2}(z;\tau))$ such that the monodromy matrices $M_{1}$,
$M_{2}$, which are defined by%
\[
\ell_{j}^{\ast}%
\begin{pmatrix}
\tilde{y}_{1}\\
\tilde{y}_{2}%
\end{pmatrix}
=M_{j}%
\begin{pmatrix}
\tilde{y}_{1}\\
\tilde{y}_{2}%
\end{pmatrix}
,\text{ }j=1,2,
\]
are independent of $\tau$. We may assume $\wp(p(\tau)|\tau)\not \equiv
\wp(p_{\infty}(\tau)|\tau)$, otherwise we are done. Then $c(\tau):=\eta
_{1}(\tau)+\wp(p(\tau))\not \equiv 0$. For any $\tau$ such that $c(\tau
)\not =0$, $(c(\tau)y_{1},y_{2})$ given by (\ref{43})-(\ref{43-3}) is also a fundamental system of solutions,
so there is an invertible matrix $\gamma=%
\begin{pmatrix}
a & b\\
c & d
\end{pmatrix}
$ such that $%
\begin{pmatrix}
\tilde{y}_{1}\\
\tilde{y}_{2}%
\end{pmatrix}
=\gamma%
\begin{pmatrix}
c(\tau)y_{1}\\
y_{2}%
\end{pmatrix}
$. Clearly the monodromy matrices of $(c(\tau)y_{1},y_{2})$ is given by
(\ref{m1}), where
\begin{equation}
\mathcal{C}=\frac{\eta_{2}(\tau)+\wp(p(\tau)|\tau)\tau}{\eta_{1}%
(\tau)+\wp(p(\tau)|\tau)} \label{p-c}%
\end{equation}
might depend on $\tau$ at the moment. Then%
\[
M_{1}=\gamma%
\begin{pmatrix}
1 & 0\\
1 & 1
\end{pmatrix}
\gamma^{-1}=%
\begin{pmatrix}
1+\frac{bd}{ad-bc} & \frac{-b^{2}}{ad-bc}\\
\frac{d^{2}}{ad-bc} & 1-\frac{bd}{ad-bc}%
\end{pmatrix}
,
\]%
\[
M_{2}=\gamma%
\begin{pmatrix}
1 & 0\\
\mathcal{C} & 1
\end{pmatrix}
\gamma^{-1}=%
\begin{pmatrix}
1+\frac{bd}{ad-bc}\mathcal{C} & \frac{-b^{2}}{ad-bc}\mathcal{C}\\
\frac{d^{2}}{ad-bc}\mathcal{C} & 1-\frac{bd}{ad-bc}\mathcal{C}%
\end{pmatrix}
.
\]
Since $M_{1}$, $M_{2}$ are independent of $\tau$ and $|b|^2+|d|^2\neq 0$, we
conclude that $\mathcal{C}$ is a constant independent of $\tau$. Consequently,
(\ref{p-c}) implies $\wp(p(\tau)|\tau)=\wp(p_{\mathcal{C}}(\tau)|\tau)$ and so
$\lambda(t)=\lambda_{\mathcal{C}}(t)$.
\end{proof}

Similarly, we can prove that all solutions of the other three Riccati
equations can be parameterized by $\mathbb{CP}^{1}$. The calculation is as
follows. Fix $k\in \{1,2,3\}$. By (\ref{43}) it is easy to see that
\[
\chi(z):=-\frac{\wp(p)-e_{k}}{(e_{k}-e_{i})(e_{k}-e_{j})}\zeta(z-\tfrac
{\omega_{k}}{2})-\left(  1+e_{k}\frac{\wp(p)-e_{k}}{(e_{k}-e_{i})(e_{k}%
-e_{j})}\right)  z
\]
satisfies (\ref{43-1}), where $\{i,j\}=\{1,2,3\} \backslash \{k\}$. As before,
for any $c(\tau)\not =0$, $(c(\tau)y_{1}(z),y_{2}(z))$ with $y_{2}(z)=\chi(z)y_{1}(z)$ is a fundamental system
of solutions to GLE$(\mathbf{0},p(\tau),A(\tau),\tau)$. In particular, as in Theorem \ref{thm0-1} we easily obtain
\begin{equation}
\ell_{1}^{\ast}%
\begin{pmatrix}
c(\tau)y_{1}\\
y_{2}%
\end{pmatrix}
=\varepsilon_{k,1}%
\begin{pmatrix}
1 & 0\\
-\frac{D\eta_{1}+(1+De_{k})}{c(\tau)} & 1
\end{pmatrix}%
\begin{pmatrix}
c(\tau)y_{1}\\
y_{2}%
\end{pmatrix}
, \label{45-4}%
\end{equation}%
\[
\ell_{2}^{\ast}%
\begin{pmatrix}
c(\tau)y_{1}\\
y_{2}%
\end{pmatrix}
=\varepsilon_{k,2}%
\begin{pmatrix}
1 & 0\\
-\frac{D\eta_{2}+\tau(1+De_{k})}{c(\tau)} & 1
\end{pmatrix}%
\begin{pmatrix}
c(\tau)y_{1}\\
y_{2}%
\end{pmatrix}
,
\]
where $(\varepsilon_{k,1},\varepsilon_{k,2})$ is given by (\ref{trace}) and%
\begin{equation}
D:=\frac{\wp(p)-e_{k}}{(e_{k}-e_{i})(e_{k}-e_{j})}. \label{45-7}%
\end{equation}

\begin{proposition}
\label{Prop-II-9}For $k\in \{1,2,3\}$ and $\mathcal{C}\in \mathbb{CP}^{1}$, we
let
\[
\lambda_{k,\mathcal{C}}^{\mathbf{0}}(t)=\frac{\wp(p_{k,\mathcal{C}%
}^{\mathbf{0}}(\tau)|\tau)-e_{1}(\tau)}{e_{2}(\tau)-e_{1}(\tau)},
\]
where
\begin{equation}
\wp(p_{k,\mathcal{C}}^{\mathbf{0}}(\tau)|\tau):=\frac{e_{k}(\mathcal{C}\eta
_{1}(\tau)-\eta_{2}(\tau))+(\frac{g_{2}}{4}-2e_{k}^{2})(\mathcal{C}-\tau
)}{\mathcal{C}\eta_{1}(\tau)-\eta_{2}(\tau)+e_{k}(\mathcal{C}-\tau)}.
\label{45-1}%
\end{equation}
Then $\lambda_{k,\mathcal{C}}^{\mathbf{0}}(t)$ satisfies the Ricatti equation
(\ref{1001}) if $k=1$, (\ref{1002}) if $k=2$, (\ref{1003}) if $k=3$.
Conversely, such $\lambda_{k,\mathcal{C}}^{\mathbf{0}}(t)$ give all the
solutions of these three Riccati equations respectively. Furthermore, the
monodromy of its associated GLE satisfies%
\begin{equation}
\rho(\ell_{1})=\varepsilon_{k,1}%
\begin{pmatrix}
1 & 0\\
1 & 1
\end{pmatrix}
,\quad\rho(\ell_{2})=\varepsilon_{k,2}%
\begin{pmatrix}
1 & 0\\
\mathcal{C} & 1
\end{pmatrix}
, \label{inftn}%
\end{equation}
where as before, when $\mathcal{C}=\infty$, it should be understand as%
\begin{equation}
\rho(\ell_{1})=\varepsilon_{k,1}I_{2},\quad\rho(\ell_{2})=\varepsilon_{k,2}%
\begin{pmatrix}
1 & 0\\
1 & 1
\end{pmatrix}
. \label{inft}%
\end{equation}

\end{proposition}

\begin{proof}
We sketch the proof for fixed $k\in \{1,2,3\}$ and omit $\mathbf{0,}k$ in the
notations. For any $\mathcal{C}\in \mathbb{C}\mathbb{P}^{1}$, we let
$p(\tau)=p_{\mathcal{C}}(\tau)$, $A(\tau)=-\frac{1}{4}\frac{\wp^{\prime \prime
}(p(\tau)-\frac{\omega_{k}}{2})}{\wp^{\prime}(p(\tau)-\frac{\omega_{k}}{2})}$
in GLE$(\mathbf{0},p(\tau),A(\tau),\tau)$. If $\mathcal{C}=\infty$, i.e.
$D\eta_{1}+(1+De_{k})\equiv0$, then we choose $c(\tau)=-[D\eta_{2}%
+\tau(1+De_{k})]=\frac{-2\pi i}{\eta_{1}(\tau)+e_{k}(\tau)}\not \equiv 0$. By
(\ref{45-4}) we obtain (\ref{inft}). If $\mathcal{C}\not =\infty$, then
(\ref{45-1}) gives $D\eta_{1}+(1+De_{k})\not \equiv 0$ and $\mathcal{C}%
=\frac{D\eta_{2}+\tau(1+De_{k})}{D\eta_{1}+(1+De_{k})}$. Choose $c(\tau
)=-[D\eta_{1}+(1+De_{k})]$, then we immediately obtain (\ref{inftn}). In both cases, GLE$(\mathbf{0},p(\tau),A(\tau),\tau)$ is monodromy preserving, so $p(\tau)=p_{\mathcal{C}}(\tau)$ is a solution of EPVI$(\frac{1}{8},\frac{1}{8},\frac{1}{8},\frac{1}{8})$. Formula (\ref{45-1})
was first obtained in \cite{Takemura}. Here by $a_{1}
=\frac{\omega_{k}}{2}$ and (\ref{kakak}), we note that $\lambda_{\mathcal{C}}(t)$
actually satisfies the Ricatti equation (\ref{1001}) if $k=1$, (\ref{1002}) if
$k=2$, (\ref{1003}) if $k=3$. The rest of the proof is similar to that of
Proposition \ref{Prop-II-8}.
\end{proof}

Remark that the explict expression of $\wp(p_{k,\mathcal{C}}^{\mathbf{0}}%
(\tau)|\tau)$ immediately implies%
\begin{equation}
\wp(p_{k,\mathcal{C}_{1}}^{\mathbf{0}}(\tau)|\tau)\equiv \wp(p_{k,\mathcal{C}%
_{2}}^{\mathbf{0}}(\tau)|\tau)\Longleftrightarrow \mathcal{C}_{1}%
=\mathcal{C}_{2}. \label{uni-rec}%
\end{equation}

The above results completely classify all the solutions of EPVI$(\frac{1}{8},\frac{1}{8},\frac{1}{8},\frac{1}{8})$ in terms of the global monodromy data of the associated GLE. For a
completely reducible solution $p_{r,s}^{\mathbf{0}}(\tau)$, we denote the
corresponding $\mu(t)$ by $\mu_{r,s}^{\mathbf{0}}(t)$ and (\ref{135}) gives%
\begin{equation}
\mu_{r,s}^{\mathbf{0}}(t)=\frac{e_{2}(\tau)-e_{1}(\tau)}{2\left[  \wp
(p_{r,s}^{\mathbf{0}}(\tau)|\tau)-\wp(r+s\tau|\tau)\right]  }. \label{mu}%
\end{equation}
For a not completely reducible solution $p_{k,\mathcal{C}}^{\mathbf{0}}(\tau)$,
we denote the corresponding $\mu(t)$ by $\mu_{k,\mathcal{C}}^{\mathbf{0}}(t)$,
and by (\ref{1000})-(\ref{1003}) or (\ref{135}),
\[
\mu_{0,\mathcal{C}}^{\mathbf{0}}(t)\equiv0,\quad
\mu_{k,\mathcal{C}}^{\mathbf{0}}(t)=\frac{e_{2}(\tau)-e_{1}(\tau)}{2[
\wp(p_{k,\mathcal{C}}^{\mathbf{0}}(\tau)|\tau)-e_{k}(\tau)]  }, \quad k=1,2,3.
\]
We conclude this section by studying the precise relation between these two kinds of solutions.

\begin{theorem}
\label{comp-non}For $\mathcal{C}\not =\infty$, there holds%
\[
\wp(p_{k,\mathcal{C}}^{\mathbf{0}}(\tau)|\tau)=\left \{
\begin{array}
[c]{l}%
\lim_{s\rightarrow0}\wp(p_{-\mathcal{C}s,s}^{\mathbf{0}}(\tau)|\tau)\text{
\  \ if }k=0,\\
\lim_{s\rightarrow0}\wp(p_{\frac{1}{2}-\mathcal{C}s,s}^{\mathbf{0}}(\tau
)|\tau)\text{ \  \ if }k=1,\\
\lim_{s\rightarrow0}\wp(p_{\mathcal{C}s,\frac{1}{2}-s}^{\mathbf{0}}(\tau
)|\tau)\text{ \  \ if }k=2,\\
\lim_{s\rightarrow0}\wp(p_{\frac{1}{2}+\mathcal{C}s,\frac{1}{2}-s}%
^{\mathbf{0}}(\tau)|\tau)\text{ \  \ if }k=3,
\end{array}
\right.
\]
and the same holds for $\mu_{k,\mathcal{C}}^{\mathbf{0}}(t)$ as the limit of
$\mu_{r,s}^{\mathbf{0}}(t)$ for $(r,s)=(-\mathcal{C}s,s)$ if $k=0$, and so on.

For $\mathcal{C}=\infty$, there holds
\[
\wp(p_{k,\infty}^{\mathbf{0}}(\tau)|\tau)=\left \{
\begin{array}
[c]{l}%
\lim_{r\rightarrow0}\wp(p_{r,0}^{\mathbf{0}}(\tau)|\tau)\text{ \  \ if }k=0,\\
\lim_{r\rightarrow0}\wp(p_{\frac{1}{2}+r,0}^{\mathbf{0}}(\tau)|\tau)\text{
\  \ if }k=1,\\
\lim_{r\rightarrow0}\wp(p_{r,\frac{1}{2}}^{\mathbf{0}}(\tau)|\tau)\text{
\  \ if }k=2,\\
\lim_{r\rightarrow0}\wp(p_{\frac{1}{2}+r,\frac{1}{2}}^{\mathbf{0}}(\tau
)|\tau)\text{ \  \ if }k=3,
\end{array}
\right.
\]
and the same holds for $\mu_{k,\infty}^{\mathbf{0}}(t)$ as the limit of
$\mu_{r,s}^{\mathbf{0}}(t)$ for $(r,s)=(r,0)$ if $k=0$, and so on.
\end{theorem}

\begin{proof}
The proof is just by computations. For example, for $\mathcal{C}\not =\infty$,
we denote $u=-\mathcal{C}s+s\tau=s(\tau-\mathcal{C})$ for convenience. Then
$u\rightarrow0$ as $s\rightarrow0$, and it follows from the Laurent series of
$\zeta(\cdot|\tau)$ and $\wp(\cdot|\tau)$ that
\[
\zeta(-Cs+s\tau|\tau)=\frac{1}{u}-\frac{g_{2}}{60}u^{3}+O(|u|^{5}),
\]%
\[
\wp(-Cs+s\tau|\tau)=\frac{1}{u^{2}}+\frac{g_{2}}{20}u^{2}+O(|u|^{4}),
\]%
\[
\wp^{\prime}(-Cs+s\tau|\tau)=\frac{-2}{u^{3}}+\frac{g_{2}}{10}u+O(|u|^{3}),
\]
hold uniformly for $\tau$ in any compact subset $K\subset \mathbb{H}$ as $s\rightarrow0$. Inserting these into Hicthin's formula (\ref{513-1}), we
easily obtain that
\[
\lim_{s\rightarrow0}\wp(p_{-\mathcal{C}s,s}^{\mathbf{0}}(\tau)|\tau
)=\frac{\eta_{2}(\tau)-\mathcal{C}\eta_{1}(\tau)}{\mathcal{C}-\tau}%
=\wp(p_{0,\mathcal{C}}^{\mathbf{0}}(\tau)|\tau)
\]
holds uniformly for $\tau$ in any compact subset $K$. Therefore, as solutions of EPVI$(\frac{1}{8},\frac{1}{8},\frac{1}{8},\frac{1}{8})$, $\wp(p_{0,\mathcal{C}}^{\mathbf{0}}(\tau)|\tau)\to \wp(p_{-\mathcal{C}s,s}^{\mathbf{0}}(\tau)|\tau
)$ as $s\to 0$.
Furthermore, it follows from (\ref{mu}) that $\lim_{s\rightarrow0}%
\mu_{-\mathcal{C}s,s}^{\mathbf{0}}(t)=0=\mu_{0,\mathcal{C}}^{\mathbf{0}}(t)$.
The other formulas can be proved similarly and we omit the details here.
\end{proof}

In the next section, we will generalize the above results to the general
case $\mathbf{n}\not =\mathbf{0}$ via the well known B\"{a}cklund transformation.

\section{General case via the B\"{a}cklund transformation}

\label{General-case-Ok}

The purpose of this section is to classify all the solutions of the elliptic
form (\ref{124}) with parameters%
\begin{equation}
\alpha_{k}=\tfrac{(2n_{k}+1)^{2}}{8},\text{ }n_{k}\in
\mathbb{Z}_{\geq0}\text{ for all }k\text{ and }\mathbf{n}\not =\mathbf{0},
\label{parameter0}%
\end{equation}
or equivalently PVI with parameters%
\begin{align}
(\alpha,\beta,\gamma,\delta)=  &  \left(  \tfrac{(2n_{0}+1)^{2}}{8},\text{ }-\tfrac{(2n_{1}+1)^{2}}{8},\text{ }\tfrac{(2n_{2}+1)^{2}}{8},\right. \nonumber \\
&  \left.  \tfrac{1}{2}-\tfrac{(2n_{3}+1)^{2}}{8}\right)  ,\text{
}n_{k}\in \mathbb{Z}_{\geq0}\text{ for all }k\text{ and }\mathbf{n}%
\not =\mathbf{0}, \label{parameter}%
\end{align}
in terms of the global monodromy data of the associated GLE. The idea is to apply the B\"{a}cklund transformations.

It is known that solutions of PVI with parameter (\ref{parameter}) could be
obtained from solutions of PVI$(\frac{1}{8},\frac{-1}{8},\frac{1}{8},\frac
{3}{8})$ (i.e. $n_{k}=0$ for all $k$) via the B\"{a}cklund transformations
(\cite{Okamoto1}). By (\ref{46-2})-(\ref{46-3}), it is convenient to consider the parameter space of PVI (equivalently the Hamiltonian system
(\ref{aa})-(\ref{98})) as an affine space%
\[
\mathcal{K}=\left \{ \theta=(\theta_{0},\theta_{1},\theta_{2}%
,\theta_{3},\theta_{4})\in \mathbb{C}^{5}\text{ }:\text{ }2\theta_{0}%
+\theta_{1}+\theta_{2}+\theta_{3}+\theta_{4}=1\right \}  .
\]

\begin{definition}
\cite{Okamoto1} \textit{An B\"{a}cklund transformation }$\kappa$\textit{ is an
invertible mapping which maps solutions }$(\lambda(t),\mu(t),t)$\textit{ of
the Hamiltonian system (\ref{aa}) with parameter }$\theta$\textit{ to
solutions }$(\kappa(\lambda)(t),\kappa(\mu)(t),$ $t)$\textit{ of (\ref{aa})
with new parameter }$\kappa(\theta)\in K$\textit{ where both }%
$\kappa(\lambda)(t)$ \textit{and }$\kappa(\mu)(t)$\textit{ are rational
functions of }$\lambda, \mu, t$. \textit{In particular,
}$\kappa(\lambda)(t)$\textit{ is a solution to PVI (\ref{46}) with new
parameter }$\kappa(\theta)\in \mathcal{K}$\textit{.}
\end{definition}

The list of the B\"{a}cklund transformations $\kappa_{j}(0\leq j\leq4)$ is given in
the Table 1 (cf. \cite{Tsuda-Okamoto-Sakai}). \begin{table}[tbh]
\caption{B\"{a}cklund transformations}%
\centering
\par%
\begin{tabular}
[c]{|l|c|c|c|c|c|l|c|c|}\hline
& $\theta_{0}$ & $\theta_{1}$ & $\theta_{2}$ & $\theta_{3}$ & $\theta_{4}$ &
$t$ & $\lambda$ & $\mu$\\ \hline
$\kappa_{0}$ & $-\theta_{0}$ & $\theta_{1}+\theta_{0}$ & $\theta_{2}%
+\theta_{0}$ & $\theta_{3}+\theta_{0}$ & $\theta_{4}+\theta_{0}$ & $t$ &
$\lambda+\frac{\theta_{0}}{\mu}$ & $\mu$\\ \hline
$\kappa_{1}$ & $\theta_{0}+\theta_{1}$ & $-\theta_{1}$ & $\theta_{2}$ &
$\theta_{3}$ & $\theta_{4}$ & $t$ & $\lambda$ & $\mu-\frac{\theta_{1}}%
{\lambda}$\\ \hline
$\kappa_{2}$ & $\theta_{0}+\theta_{2}$ & $\theta_{1}$ & $-\theta_{2}$ &
$\theta_{3}$ & $\theta_{4}$ & $t$ & $\lambda$ & $\mu-\frac{\theta_{2}}%
{\lambda-1}$\\ \hline
$\kappa_{3}$ & $\theta_{0}+\theta_{3}$ & $\theta_{1}$ & $\theta_{2}$ &
$-\theta_{3}$ & $\theta_{4}$ & $t$ & $\lambda$ & $\mu-\frac{\theta_{3}%
}{\lambda-t}$\\ \hline
$\kappa_{4}$ & $\theta_{0}+\theta_{4}$ & $\theta_{1}$ & $\theta_{2}$ &
$\theta_{3}$ & $-\theta_{4}$ & $t$ & $\lambda$ & $\mu$\\ \hline
\end{tabular}
\end{table} Among them $\kappa_0$ is due to Okamoto \cite{Okamoto1} while the others are classically known. These transformations $\kappa_{j}$ $(0\leq j\leq4)$, which
satisfy $\kappa_{j}\circ \kappa_{j}=Id$ (i.e. $\kappa_j^{-1}=\kappa_j$), generate the affine Weyl group of type
$D_{4}^{(1)}$:%
\begin{equation}
W(D_{4}^{(1)})=\left \langle \kappa_{0},\kappa_{1},\kappa_{2},\kappa_{3}%
,\kappa_{4}\right \rangle . \label{II-100}%
\end{equation}

Denote $\theta^{\mathbf{0}}:=(-\tfrac{1}{2},\tfrac{1}{2},\tfrac{1}{2}%
,\tfrac{1}{2},\tfrac{1}{2})$ which corresponds to PVI$(\frac{1}{8},\frac
{-1}{8},\frac{1}{8},\frac{3}{8})$. By Table 1 there exists $\kappa
^{\mathbf{n}}\in W(D_{4}^{(1)})$ such that
\begin{equation}
\theta^{\mathbf{n}}:=\left(  -\frac{1+\sum n_{k}}{2},n_{1}+\tfrac
{1}{2},n_{2}+\tfrac{1}{2},n_{3}+\tfrac{1}{2},n_{0}+\tfrac{1}{2}\right)
=\kappa^{\mathbf{n}}(\theta^{\mathbf{0}}). \label{II-105}
\end{equation}
Note that \begin{equation}\label{IiI-105}(\kappa
^{\mathbf{n}})^{-1}\in W(D_{4}^{(1)})\quad\text{and }\; \theta^{\mathbf{0}}= (\kappa
^{\mathbf{n}})^{-1}(\theta^{\mathbf{n}}). \end{equation}
Consequently, there exist two \emph{rational functions} $R^{\mathbf{n}}%
(\cdot,\cdot,\cdot)$ and $\tilde{R}^{\mathbf{n}}(\cdot,\cdot,\cdot)$ of three
independent variables with coefficients in $\mathbb{Q}$ such that for any
solution $(\lambda^{\mathbf{0}}(t),\mu^{\mathbf{0}}(t))$ of the Hamiltonian
system (\ref{aa}) with parameter $\theta^{\mathbf{0}}$, $(\lambda^{\mathbf{n}%
}(t),\mu^{\mathbf{n}}(t))$ given by%
\begin{equation}
\lambda^{\mathbf{n}}(t):=\kappa(\lambda^{\mathbf{0}})(t)=R^{\mathbf{n}%
}(\lambda^{\mathbf{0}}(t),\mu^{\mathbf{0}}(t),t), \label{II-128}%
\end{equation}%
\begin{equation}
\mu^{\mathbf{n}}(t):=\kappa(\mu^{\mathbf{0}})(t)=\tilde{R}^{\mathbf{n}%
}(\lambda^{\mathbf{0}}(t),\mu^{\mathbf{0}}(t),t), \label{II-128-0}%
\end{equation}
is a solution of the Hamiltonian system (\ref{aa}) with parameter
$\theta^{\mathbf{n}}$, or equivalently, $\lambda^{\mathbf{n}}(t)$ is a
solution of PVI with parameter (\ref{parameter}).

Remark that by (\ref{IiI-105}), there are also two \emph{rational functions} $\mathcal{R}^{\mathbf{n}}
(\cdot,\cdot,\cdot)$ and $\tilde{\mathcal{R}}^{\mathbf{n}}(\cdot,\cdot,\cdot)$ of three
independent variables with coefficients in $\mathbb{Q}$ such that the rational map (\ref{II-128})-(\ref{II-128-0}) is invertible in the following sense
\begin{equation}
\lambda^{\mathbf{0}}(t)=\mathcal{R}^{\mathbf{n}
}(\lambda^{\mathbf{n}}(t),\mu^{\mathbf{n}}(t),t),\quad \mu^{\mathbf{0}}(t)=\tilde{\mathcal{R}}^{\mathbf{n}
}(\lambda^{\mathbf{n}}(t),\mu^{\mathbf{n}}(t),t). \label{IiI-128}%
\end{equation}%

In the literature, there are also references treating
the B\"{a}cklund transformations as
biholomorphic transformations on the space of initial conditions for solutions of Painlev\'{e} equations; see e.g. \cite{ST,Takemura09}. In this paper, (\ref{II-128})-(\ref{IiI-128}) are enough for our following arguments and so we do not need to discuss the space of initial conditions.

\medskip

\noindent \textbf{Notation}: Let $p^{\mathbf{n}}(\tau)$ be a solution of the
elliptic form (\ref{124}) with parameter (\ref{parameter0}). We denote it by
$p_{r,s}^{\mathbf{n}}(\tau)$ (resp. $p_{k,\mathcal{C}}^{\mathbf{n}}(\tau)$) if it comes
from the solution $p_{r,s}^{\mathbf{0}}(\tau)$ (resp. $p_{k,\mathcal{C}}^{\mathbf{0}}(\tau)$) of EPVI$(\frac{1}{8},\frac{1}{8},\frac{1}{8},\frac{1}{8})$
via (\ref{II-128}), i.e.%
\begin{equation}
\frac{\wp(p_{r,s}^{\mathbf{n}}(\tau)|\tau)-e_{1}(\tau)}{e_{2}(\tau)-e_{1}%
(\tau)}=R^{\mathbf{n}}\left(  \frac{\wp(p_{r,s}^{\mathbf{0}}(\tau)|\tau
)-e_{1}(\tau)}{e_{2}(\tau)-e_{1}(\tau)},\mu_{r,s}^{\mathbf{0}}(t),t\right),
\label{comred}%
\end{equation}%
\begin{equation}
\frac{\wp(p_{k,\mathcal{C}}^{\mathbf{n}}(\tau)|\tau)-e_{1}(\tau)}{e_{2}(\tau)-e_{1}%
(\tau)}=R^{\mathbf{n}}\left(  \frac{\wp(p_{k,\mathcal{C}}^{\mathbf{0}}(\tau)|\tau
)-e_{1}(\tau)}{e_{2}(\tau)-e_{1}(\tau)},\mu_{k,\mathcal{C}}^{\mathbf{0}}(t),t\right)  .
\label{ncomred}%
\end{equation}
We use similar notations $\mu_{r,s}^{\mathbf{n}}(t)$ and $\mu_{k,\mathcal{C}}^{\mathbf{n}}(t)$ via (\ref{II-128-0}). Consequently, it follows from (\ref{IiI-128}) that
\begin{equation}
\frac{\wp(p_{r,s}^{\mathbf{0}}(\tau)|\tau)-e_{1}(\tau)}{e_{2}(\tau)-e_{1}%
(\tau)}=\mathcal{R}^{\mathbf{n}}\left(  \frac{\wp(p_{r,s}^{\mathbf{n}}(\tau)|\tau
)-e_{1}(\tau)}{e_{2}(\tau)-e_{1}(\tau)},\mu_{r,s}^{\mathbf{n}}(t),t\right),
\label{ciomred}%
\end{equation}%
\begin{equation}
\frac{\wp(p_{k,\mathcal{C}}^{\mathbf{0}}(\tau)|\tau)-e_{1}(\tau)}{e_{2}(\tau)-e_{1}%
(\tau)}=\mathcal{R}^{\mathbf{n}}\left(  \frac{\wp(p_{k,\mathcal{C}}^{\mathbf{n}}(\tau)|\tau
)-e_{1}(\tau)}{e_{2}(\tau)-e_{1}(\tau)},\mu_{k,\mathcal{C}}^{\mathbf{n}}(t),t\right)  .
\label{nciomred}%
\end{equation}

\begin{remark} Given $(r,s)\in\mathbb{C}^2\setminus\frac12\mathbb{Z}^2$, we write $Z=Z_{r,s}(\tau)$, $\wp=\wp(r+s\tau|\tau)$ and $\wp'=\wp'(r+s\tau|\tau)$ for convenience. Then Hitchin's formula (\ref{513-1}) gives
\[
\wp(p_{r,s}^{\mathbf{0}}(\tau)|\tau)=\wp+\frac{\wp^{\prime
}}{2Z }.
\]
Consequently, we see from (\ref{135}) that
\[
\mu_{r,s}^{\mathbf{0}}(t) =\frac{e_{2}(\tau)-e_{1}(\tau)}{2[  \wp(p_{r,s}^{\mathbf{0}}(\tau)|\tau)-\wp]  }
=\frac{(e_{2}(\tau)-e_{1}(\tau))Z}{\wp^{\prime}}. \]
Inserting these and $t=\frac{e_3(\tau)-e_1(\tau)}{e_2(\tau)-e_1(\tau)}$ into (\ref{comred}), we conclude that
\[\wp(p_{r,s}^{\mathbf{n}}(\tau)|\tau)=\Xi_{\mathbf{n}}(Z,\wp,\wp',e_1(\tau),e_2(\tau),
e_3(\tau)),\]
where $\Xi_{\mathbf{n}}$ is a rational function of six independent variables with coefficients in $\mathbb{Q}$.
\end{remark}
Our main results of this section are as follows, which indicate that the B\"{a}cklund transformation preserves the global monodromy data (or equivalently the monodromy representation) in both completely reducible and not completely reducible cases.

\begin{theorem}
[Completely reducible solutions]\label{thm-II-8} \

\begin{itemize}
\item[(1)] $p^{\mathbf{n}}(\tau)$ is a completely reducible solution if and
only if there exists $\left(  r,s\right)  \in \mathbb{C}^{2}\backslash \frac
{1}{2}\mathbb{Z}^{2}$ such that $p^{\mathbf{n}}(\tau)=p_{r,s}^{\mathbf{n}%
}(\tau)$. In this case, for any $\tau$ satisfying $p^{\mathbf{n}}%
(\tau)\not \in E_{\tau}[2]$, the monodromy of the associated GLE$(\mathbf{n}$, $ p^{\mathbf{n}}(\tau), A^{\mathbf{n}}(\tau), \tau)$ satisfies
(\ref{61.35}), i.e. the global monodromy data is precisely this $(r,s)$.

\item[(2)] $\wp(p_{r_{1},s_{1}}^{\mathbf{n}}(\tau)|\tau)\equiv \wp
(p_{r_{2},s_{2}}^{\mathbf{n}}(\tau)|\tau)\Longleftrightarrow(r_{1},s_{1})  \equiv \pm(r_{2},s_{2})  \operatorname{mod}$
$\mathbb{Z}^{2}$.
\end{itemize}
\end{theorem}

\begin{theorem}
[Not completely reducible solutions]\label{thm-II-8-0} \

\begin{itemize}
\item[(1)] $p^{\mathbf{n}}(\tau)$ is a not completely reducible solution if
and only if there exist $k\in \{0,1,2,3\}$ and $\mathcal{C}\in \mathbb{C}\cup \{
\infty \}$ such that $p^{\mathbf{n}}(\tau)=p_{k,\mathcal{C}}^{\mathbf{n}}%
(\tau)$. In this case, for any $\tau$ satisfying $p^{\mathbf{n}}%
(\tau)\not \in E_{\tau}[2]$, the monodromy of the associated GLE$(\mathbf{n},  p^{\mathbf{n}}(\tau), A^{\mathbf{n}}(\tau), \tau)$ satisfies
(\ref{inftn})-(\ref{inft}), i.e. the global monodromy data is precisely
$(2\varepsilon_{k,1},2\varepsilon_{k,2},\mathcal{C})$.

\item[(2)] $\wp(p_{k,\mathcal{C}_{1}}^{\mathbf{n}}(\tau)|\tau)\equiv
\wp(p_{k,\mathcal{C}_{2}}^{\mathbf{n}}(\tau)|\tau)$ if and only if
$\mathcal{C}_{1}=\mathcal{C}_{2}$.
\end{itemize}
\end{theorem}

The rest of this section is devoted to the proofs of these theorems. First we
note that by applying the gauge transformation%
\begin{equation}
f(x)=\phi(x)F(x)\text{ \ with \ }\phi(x)=(x-\lambda)x^{\frac{\theta_{1}}{2}%
}(x-1)^{\frac{\theta_{2}}{2}}(x-t)^{\frac{\theta_{3}}{2}}, \label{II-98}%
\end{equation}
equation (\ref{90}) is normalized into a new Fuchsian ODE%
\begin{equation}
\frac{d^{2}F}{dx^{2}}+P_{1}(x)\frac{dF}{dx}+P_{2}(x)F=0, \label{II-99}%
\end{equation}
where
\[
P_{1}=p_{1}+2\frac{\phi^{\prime}}{\phi},\text{ \  \ }P_{2}=p_{2}+\frac
{\phi^{\prime}}{\phi}p_{1}+\frac{\phi^{\prime \prime}}{\phi}.
\]
Clearly the Riemann scheme of (\ref{II-99}) is%
\begin{equation}
\left(
\begin{array}
[c]{ccccc}%
0 & 1 & t & \lambda & \infty \\
-\frac{\theta_{1}}{2} & -\frac{\theta_{2}}{2} & -\frac{\theta_{3}}{2} & -1 &
\frac{3-\theta_{4}}{2}\\
\frac{\theta_{1}}{2} & \frac{\theta_{2}}{2} & \frac{\theta_{3}}{2} & 1 &
\frac{3+\theta_{4}}{2}%
\end{array}
\right)  , \label{91-1}%
\end{equation}
and $\lambda$ is still an apparent singularity of (\ref{II-99}). As in
\cite{Inaba-Iwasaki-Saito}, equation (\ref{II-99}) is called the
\textit{normal form} of (\ref{90}). By (\ref{91-1}) it is easy to see that the
normal form (\ref{II-99}) has its monodromy group contained in
$SL(2,\mathbb{C})$, which is an important advantage comparing to (\ref{90}).

We proceed to the monodromy representation. Take the base point $x_{0}%
=\frac{\wp(q_{0})-e_{1}}{e_{2}-e_{1}}\not \in \{0,1,t,\infty \}$ and let
$\gamma_{j}\in \pi_{1}(\mathbb{C}\backslash \{0,1,t\},x_{0})$ be a simple loop
encircling the singular point $0$ for $j=1$, $1$ for $j=2$, $t$ for $j=3$
respectively in the counterclockwise direction, and $\gamma_{4}$ be a simple
loop around $\infty$ clockwise such that%
\[
\gamma_{1}\gamma_{2}\gamma_{3}=\gamma_{4}^{-1}\quad\text{in}\;\pi_{1}(\mathbb{C}\backslash \{0,1,t\},x_{0}).
\]
Of course we require that all these loops do not intersect except at the base
point $x_{0}$. Let $M_{j}$ be the monodromy matrix along the loop $\gamma_{j}$
with respect to any fixed fundamental system of solutions $(F_{1}%
(x),F_{2}(x))$ of (\ref{II-99}). Then $\det M_{j}=1$, namely
$M_{j}\in SL(2,\mathbb{C})$ for all $j$. Define%
\begin{equation}
\varkappa_{1}:=\text{tr}(M_{2}M_{3}),\text{ \ }\varkappa_{2}:=\text{tr}%
(M_{1}M_{3}),\text{ \ }\varkappa_{3}:=\text{tr}(M_{1}M_{2}). \label{II-102}%
\end{equation}
Then $\varkappa=(\varkappa_{1},\varkappa_{2},\varkappa_{3})\in \mathbb{C}^{3}$
is \textit{independent} of the choice of solutions, and
is referred to as \textit{global monodromy data} of (\ref{90}) (or
(\ref{II-99})) in \cite{Inaba-Iwasaki-Saito}. Clearly $\varkappa_{j}%
=\varkappa_{j}(\theta,\lambda,\mu,t)$ is uniquely determined by equation
(\ref{90}) itself and so is a function of $(\theta,\lambda,\mu,t)$ for all $j$. Then
each B\"{a}cklund transformation $\kappa \in W(D_{4}^{(1)})$ induces a
transformation (still denoted by $\kappa$) from $\mathbb{C}^{3}$ to
$\mathbb{C}^{3}$:%
\begin{equation}
\kappa(\varkappa_{j}):=\varkappa_{j}(\kappa(\theta),\kappa(\lambda),\kappa
(\mu),t),\text{ \ }j=1,2,3. \label{II-103}%
\end{equation}
We recall an important result from \cite{Inaba-Iwasaki-Saito}; see also
\cite{Boalch} for a different proof.\medskip

\noindent \textbf{Theorem 5.A. \cite{Inaba-Iwasaki-Saito,Boalch} }\emph{The
global monodromy data }$\varkappa=(\varkappa_{1},\varkappa_{2},\varkappa_{3}%
)$\emph{ is invariant under the B\"{a}cklund transformations }$W(D_{4}^{(1)}%
)$\emph{. Namely for any B\"{a}cklund transformation }$\kappa \in W(D_{4}^{(1)}%
)$\emph{, }$\kappa(\varkappa_{j})=\varkappa_{j}$\emph{ for }$j=1,2,3$%
\emph{.\medskip}

Theorem 5.A can be also applied to GLE($\mathbf{n},$ $p,$ $A,$ $\tau$). Consider transformations as in \cite{Chen-Kuo-Lin}
\begin{equation}
x=\frac{\wp(z)-e_{1}}{e_{2}-e_{1}},\text{ \ }t=\frac{e_{3}-e_{1}}{e_{2}-e_{1}%
},\text{ \ }\lambda=\frac{\wp(p)-e_{1}}{e_{2}-e_{1}}, \label{II-104-0}%
\end{equation}
and%
\begin{equation}
(x-\lambda)^{-\frac{1}{2}}x^{-\frac{n_{1}}{2}}(x-1)^{-\frac{n_{2}}{2}%
}(x-t)^{-\frac{n_{3}}{2}}f(x)=y(z). \label{II-104}%
\end{equation}
Then $y(z)$ solves GLE$(\mathbf{n},p,A,\tau)$ if and only if $f(x)$ satisfies the Fuchsian
ODE (\ref{90}) on $\mathbb{CP}^{1}$ with parameter $\theta=\theta^{\mathbf{n}%
}$, where $\mu$ in (\ref{97}) is given by%
\begin{equation}
\mu=\frac{1}{8}\frac{\mathfrak{p}^{\prime}(\lambda)}{\mathfrak{p}(\lambda
)}+\frac{A\wp^{\prime}(p)}{(e_{2}-e_{1})^{2}\mathfrak{p}(\lambda)}+\frac
{n_{1}}{2\lambda}+\frac{n_{2}}{2(\lambda-1)}+\frac{n_{3}}{2(\lambda-t)},
\label{II-106}%
\end{equation}%
\begin{equation}
\text{where}\quad\mathfrak{p}(\lambda)=4\lambda(\lambda-1)(\lambda-t), \label{II-107}%
\end{equation}
and $K=K(\lambda,\mu,t)$ is given by (\ref{98}). Note that $\pm p\not \in
E_{\tau}[2]$ are apparent singularities of GLE$(\mathbf{n},p,A,\tau)$ is
equivalent to that $\lambda \not \in \{0,1,t,\infty \}$ is an apparent
singularity of (\ref{90}). See \cite[Theorem 4.1]{Chen-Kuo-Lin} for the proof.

By (\ref{II-105}), (\ref{II-98}) and (\ref{II-104}), we let
\begin{equation}
y(z)=\psi(x)F(x)\text{ \ with \ }\psi(x)=(x-\lambda)^{\frac{1}{2}}x^{\frac
{1}{4}}(x-1)^{\frac{1}{4}}(x-t)^{\frac{1}{4}}. \label{II-108}%
\end{equation}
Then the above argument shows that $y(z)$ is a solution to GLE$(\mathbf{n}%
,p,A,\tau)$ if and only if $F(x)$ satisfies the normal form (\ref{II-99}).

\begin{remark}
\label{loop}Recall the definition of $\gamma_{j}\in \pi_{1}(\mathbb{C}%
\backslash \{0,1,t\},x_{0})$. Under the transformation (\ref{II-104-0}), it is
easy to see that the fundamental cycle $\ell_{1}$ (resp. $\ell_{2}$) of $E_{\tau}$ is mapped to a simple loop in
$\pi_{1}(\mathbb{C}\backslash \{0,1,t\},x_{0})$ which separates $\{1,t\}$ from
$\{0,\infty \}$ (resp. separates $\{0,t\}$ from $\{1,\infty \}$), so $(\ell
_{1},\ell_{2})$ must be mapped to one of%
\[
(\gamma_{2}^{-1}\gamma_{3}^{-1},\gamma_{1}\gamma_{3}),(\gamma_{3}\gamma
_{2},\gamma_{3}^{-1}\gamma_{1}^{-1}),(\gamma_{2}\gamma_{3},\gamma_{3}%
\gamma_{1}),(\gamma_{3}^{-1}\gamma_{2}^{-1},\gamma_{1}^{-1}\gamma_{3}^{-1}).
\]
In this paper, by letting the base point $q_{0}$ lie inside the parallelogram
with vertices $\{0,\frac{-\omega_{1}}{2},\frac{-\omega_{2}}{2},\frac
{-\omega_{3}}{2}\}$, we can always assume that $(\ell_{1},\ell_{2})$ is
mapped to $(\gamma_{2}^{-1}\gamma_{3}^{-1},$ $\gamma_{1}\gamma_{3})$.
\end{remark}

Recalling the global monodromy data $\varkappa=(\varkappa_{1},\varkappa
_{2},\varkappa_{3})$ of the normal form (\ref{II-99}), we have the following
important result.

\begin{lemma}
\label{thm3}
\[
\mathit{tr}\rho(\ell_{1})=-\mathit{tr}(M_{2}M_{3})=-\varkappa_{1},
\]%
\[
\mathit{tr}\rho(\ell_{2})=-\mathit{tr}(M_{1}M_{3})=-\varkappa_{2},
\]%
\[
\mathit{tr(}\rho(\ell_{1})^{-1}\rho(\ell_{2}))=-\mathit{tr}(M_{1}%
M_{2})=-\varkappa_{3}.
\]

\end{lemma}

\begin{proof}
Let $(y_{1}(z),y_{2}(z))$ be any fundamental system of solutions to
GLE$(\mathbf{n}$, $p,A, \tau)$. Define a fundamental system of solutions
$(F_{1}(x),F_{2}(x))$ of (\ref{II-99}) via $(y_{1}(z),y_{2}(z))$ and
(\ref{II-108}). Recall the notation $N_{j}=\rho(\ell_{j})$. Under the
transformation (\ref{II-104-0}), it follows from Remark \ref{loop} that
$(\ell_{1},\ell_{2})$ is mapped to $(\gamma_{2}^{-1}\gamma_{3}^{-1},\gamma
_{1}\gamma_{3})$. Then{\allowdisplaybreaks%
\begin{align*}
N_{1}%
\begin{pmatrix}
y_{1}(z)\\
y_{2}(z)
\end{pmatrix}
&  =\ell_{1}^{\ast}%
\begin{pmatrix}
y_{1}(z)\\
y_{2}(z)
\end{pmatrix}
=\left(  \gamma_{2}^{-1}\gamma_{3}^{-1}\right)  ^{\ast}\psi(x)%
\begin{pmatrix}
F_{1}(x)\\
F_{2}(x)
\end{pmatrix}
\\
&  =-\psi(x)M_{2}^{-1}M_{3}^{-1}%
\begin{pmatrix}
F_{1}(x)\\
F_{2}(x)
\end{pmatrix} =-M_{2}^{-1}M_{3}^{-1}%
\begin{pmatrix}
y_{1}(z)\\
y_{2}(z)
\end{pmatrix},
\end{align*}}%
and similarly,
\[
N_{2}%
\begin{pmatrix}
y_{1}(z)\\
y_{2}(z)
\end{pmatrix}
=\ell_{2}^{\ast}%
\begin{pmatrix}
y_{1}(z)\\
y_{2}(z)
\end{pmatrix}
=-M_{1}M_{3}%
\begin{pmatrix}
y_{1}(z)\\
y_{2}(z)
\end{pmatrix}
,
\]
where the minus sign comes from the analytic continuation of $\psi
(x)$.\ Therefore, $N_{1}=-$ $M_{2}^{-1}M_{3}^{-1}$ and $N_{2}=-M_{1}M_{3}$.
Since $M_{j}\in SL(2,\mathbb{C})$, we have
\[
\text{tr}(M_{2}^{-1}M_{3}^{-1})=\text{tr}((M_{2}M_{3})^{-1})=\text{tr}%
(M_{2}M_{3})=\varkappa_{1},
\]
which proves tr$N_{1}=-\varkappa_{1}$ and similarly tr$N_{2}=-$tr$(M_{1}%
M_{3})=-\varkappa_{2}$.

On the other hand, recall (\ref{II-105}) that $\theta_{j}=n_{j}+\tfrac{1}{2}$
with $n_{j}\in \mathbb{Z}_{\geq0}$ for $j=1,2,3$, so (\ref{91-1}) implies the
existence of inverse matrices $P_{j}$ such that%
\[
M_{j}=P_{j}^{-1}%
\begin{pmatrix}
e^{-\pi i\theta_{j}} & 0\\
0 & e^{\pi i\theta_{j}}%
\end{pmatrix}
P_{j}=(-1)^{n_{j}}P_{j}^{-1}%
\begin{pmatrix}
-i & 0\\
0 & i
\end{pmatrix}
P_{j},
\]
which infers $M_{j}^{2}=-I_{2}$. Therefore,%
\begin{align*}
\text{tr}(N_{1}^{-1}N_{2})  &  =\text{tr}(M_{3}M_{2}M_{1}M_{3})=\text{tr}%
(M_{3}^{2}M_{2}M_{1})\\
&  =-\text{tr}(M_{2}M_{1})=-\text{tr}(M_{1}M_{2})=-\varkappa_{3}.
\end{align*}
The proof is complete.
\end{proof}

We are in the position to prove Theorems \ref{thm-II-8}-\ref{thm-II-8-0}.

\begin{proof}
[Proof of Theorems \ref{thm-II-8}-\ref{thm-II-8-0}]First, the assertions (2) of these two theorems follow directly from
Theorem \ref{thm5}-(ii), (\ref{uni-rec}) and (\ref{comred})-(\ref{nciomred}) (i.e. the invertibility of the B\"{a}cklund transformation implies the invertibility of the associated rational map).

Suppose $p^{\mathbf{n}}(\tau)$ is a solution of the elliptic form (\ref{124})
with parameter (\ref{parameter0}), and $p^{\mathbf{0}}(\tau)$ is the
corresponding solution of the elliptic form (\ref{hit}) such that under the
B\"{a}cklund transformation $\kappa^{\mathbf{n}}$, $p^{\mathbf{0}}(\tau)$ is
transformed to $p^{\mathbf{n}}(\tau)$. By Theorem 5.A and Lemma \ref{thm3},
the associated GLE$(\mathbf{n},p^{\mathbf{n}}(\tau),A^{\mathbf{n}}(\tau
),\tau)$ and GLE$(\mathbf{0},p^{\mathbf{0}}(\tau),A^{\mathbf{0}}(\tau),\tau)$
have the same $($tr$\rho(\ell_{1}),$tr$\rho(\ell_{2}))$. Together with
Corollary \ref{coro4},\ we conclude that $p^{\mathbf{n}}(\tau)$ is a
completely reducible solution (resp. not completely reducible) if and only if
$p^{\mathbf{0}}(\tau)$ is a completely reducible solution (resp. not
completely reducible).

Now we prove Theorem \ref{thm-II-8}-(1). Let $p^{\mathbf{n}}(\tau)$ be a
completely reducible solution, then so does $p^{\mathbf{0}}(\tau)$. Applying
Theorem \ref{thm5}, there exists $\left(  r,s\right)  \in \mathbb{C}%
^{2}\backslash \frac{1}{2}\mathbb{Z}^{2}$ such that $p^{\mathbf{0}}%
(\tau)=p_{r,s}^{\mathbf{0}}(\tau)$ and the monodromy of GLE$(\mathbf{0}%
,p^{\mathbf{0}}(\tau)$, $A^{\mathbf{0}}(\tau),\tau)$ satisfies (\ref{61.35}),
which implies%
\begin{align}
&  (\text{tr}\rho(\ell_{1}),\;\text{tr}\rho(\ell_{2}),\;\text{tr}(\rho(\ell
_{1})^{-1}\rho(\ell_{2})))\label{trac}\\
&  =(2\cos2\pi s,\;2\cos2\pi r,\;2\cos2\pi(r+s)).\nonumber
\end{align}
Thus, $p^{\mathbf{n}}(\tau)=p_{r,s}^{\mathbf{n}}(\tau)$ and Lemma \ref{thm3}
implies that (\ref{trac}) holds for GLE$(\mathbf{n},$ $p^{\mathbf{n}}%
(\tau),A^{\mathbf{n}}(\tau),\tau)$. Consequently, it follows from Theorem
\ref{thm0} and (\ref{aunique})-(\ref{aunique11}) that the monodromy of GLE$(\mathbf{n}%
,p^{\mathbf{n}}(\tau),A^{\mathbf{n}}(\tau),\tau)$ satisfies (\ref{61.35}).
This proves Theorem \ref{thm-II-8}-(1).

Finally, we prove Theorem \ref{thm-II-8-0}-(1). Let $p^{\mathbf{n}}(\tau)$ be
a not completely reducible solution, then so does $p^{\mathbf{0}}(\tau)$. By
Theorem \ref{thm3.3} and Propositions \ref{Prop-II-8}-\ref{Prop-II-9}, there
exist $k\in \{0,1,2,3\}$ and $\mathcal{C}\in \mathbb{C}\cup \{ \infty \}$ such
that $p^{\mathbf{0}}(\tau)=p_{k,\mathcal{C}}^{\mathbf{0}}(\tau)$ and
\begin{equation}
(\text{tr}\rho(\ell_{1}),\;\text{tr}\rho(\ell_{2}))
=(2\varepsilon_{k,1},\;2\varepsilon_{k,2}). \label{trac-0}%
\end{equation}
for GLE$(\mathbf{0},p^{\mathbf{0}}(\tau),A^{\mathbf{0}}(\tau),\tau)$. Thus
$p^{\mathbf{n}}(\tau)=p_{k,\mathcal{C}}^{\mathbf{n}}(\tau)$ and (\ref{trac-0})
holds for GLE$(\mathbf{n},p^{\mathbf{n}}(\tau),A^{\mathbf{n}}(\tau),\tau)$.

It remains to prove that the monodromy of GLE$(\mathbf{n},p^{\mathbf{n}}(\tau),A^{\mathbf{n}}(\tau),\tau)$ satisfies (\ref{inftn})-(\ref{inft}),
i.e. the global monodromy data is precisely $(2\varepsilon_{k,1}%
,2\varepsilon_{k,2},\mathcal{C})$. Note that we only need to prove this assertion for {\it some $\tau$} because of the isomonodromic deformation.
We take $k=1$ and $\mathcal{C}\not =\infty$
for example, and all the other cases can be proved in the same way. By Theorem
\ref{comp-non} and (\ref{II-128})-(\ref{ncomred}), we easily obtain%
\[
\wp(p_{1,\mathcal{C}}^{\mathbf{n}}(\tau)|\tau)=\lim_{s\rightarrow0}%
\wp(p_{\frac{1}{2}-\mathcal{C}s,s}^{\mathbf{n}}(\tau)|\tau),
\]%
\[
\mu_{1,\mathcal{C}}^{\mathbf{n}}(t)=\lim_{s\rightarrow0}\mu_{\frac{1}%
{2}-\mathcal{C}s,s}^{\mathbf{n}}(t).
\]
Fix any $\tau$ such that $p_{1,\mathcal{C}}^{\mathbf{n}}(\tau)\not \in
E_{\tau}[2]$. By Remark \ref{identify} we may assume%
\[
p_{1,\mathcal{C}}^{\mathbf{n}}(\tau)=\lim_{s\rightarrow0}p_{\frac{1}%
{2}-\mathcal{C}s,s}^{\mathbf{n}}(\tau)
\]
and then it follows from (\ref{II-106}) that the corresponding
\[
A_{1,\mathcal{C}}^{\mathbf{n}}(\tau)=\lim_{s\rightarrow0}A_{\frac{1}%
{2}-\mathcal{C}s,s}^{\mathbf{n}}(\tau).
\]
In the rest of the proof, we omit $\mathbf{n}$, $\tau$ in the notations for
convenience. Thus the associated GLE$(\mathbf{n},p_{1,\mathcal{C}%
},A_{1,\mathcal{C}})$ is a limit of GLE$(\mathbf{n},p_{\frac{1}{2}%
-\mathcal{C}s,s},A_{\frac{1}{2}-\mathcal{C}s,s})$. Denote by $\Phi_{e}(z)$ and
$\Phi_{e,s}(z)$ respectively, to be their corresponding unique even elliptic
solution stated in Theorem 2.A. Then%
\begin{equation}
\Phi_{e}(z)=\lim_{s\rightarrow0}\Phi_{e,s}(z). \label{phi}%
\end{equation}
Recall Theorem \ref{thm0-1} that%
\begin{equation}
\chi_{j}:=\int_{z}^{z+\omega_{j}}\frac{1}{\Phi_{e}(\xi)}d\xi \not =%
\infty,\text{ \ }j=1,2 \label{phi1}%
\end{equation}
are well-defined and independent of $z$. We claim that%
\begin{equation}
{\chi_{2}}/{\chi_{1}}=\mathcal{C}. \label{phi2}%
\end{equation}
Once (\ref{phi2}) is proved, then Theorem \ref{thm0-1} and (\ref{trac-0}) imply
that the monodromy of GLE$(\mathbf{n},p_{1,\mathcal{C}},A_{1,\mathcal{C}})$
satisfies (\ref{inftn})-(\ref{inft}) with $k=1$, hence completes the proof of Theorem
\ref{thm-II-8-0}-(1).

To prove (\ref{phi2}), we apply Theorem \ref{thm0} and Theorem \ref{thm5}-(i)
to GLE$(\mathbf{n},$ $p_{\frac{1}{2}-\mathcal{C}s,s},A_{\frac{1}%
{2}-\mathcal{C}s,s})$ and denote the corresponding $y_{\pm \boldsymbol{a}}(z)$
by $y_{\pm \boldsymbol{a}(s)}(z)$, which gives%
\[
\ell_{1}^{\ast}%
\begin{pmatrix}
y_{\boldsymbol{a}(s)}(z)\\
y_{-\boldsymbol{a}(s)}(z)
\end{pmatrix}
=%
\begin{pmatrix}
e^{-2\pi is} & 0\\
0 & e^{2\pi is}%
\end{pmatrix}%
\begin{pmatrix}
y_{\boldsymbol{a}(s)}(z)\\
y_{-\boldsymbol{a}(s)}(z)
\end{pmatrix}
,
\]%
\[
\ell_{2}^{\ast}%
\begin{pmatrix}
y_{\boldsymbol{a}(s)}(z)\\
y_{-\boldsymbol{a}(s)}(z)
\end{pmatrix}
=%
\begin{pmatrix}
e^{2\pi i(\frac{1}{2}-\mathcal{C}s)} & 0\\
0 & e^{-2\pi i(\frac{1}{2}-\mathcal{C}s)}%
\end{pmatrix}%
\begin{pmatrix}
y_{\boldsymbol{a}(s)}(z)\\
y_{-\boldsymbol{a}(s)}(z)
\end{pmatrix}
.
\]
By (\ref{a-a1}) there exists a nonzero constant $c(s)$ such that%
\[
\Phi_{e,s}(z)=c(s)y_{\boldsymbol{a}(s)}(z)y_{-\boldsymbol{a}(s)}(z).
\]
It follows from (\ref{phi}) that up to a subsequence, $\lim_{s\rightarrow
0}c(s)=c_{0}\not \in \{0,\infty \}$. Let%
\[
W(s):=y_{\boldsymbol{a}(s)}^{\prime}(z)y_{-\boldsymbol{a}(s)}%
(z)-y_{\boldsymbol{a}(s)}(z)y_{-\boldsymbol{a}(s)}^{\prime}(z)
\]
be the Wronskian, which is a nonzero constant independent of $z$. Since
GLE$(\mathbf{n},p_{\frac{1}{2}-\mathcal{C}s,s},A_{\frac{1}{2}-\mathcal{C}%
s,s})$ converges to GLE$(\mathbf{n},p_{1,\mathcal{C}},A_{1,\mathcal{C}})$
whose monodromy is not completely reducible, we have
\begin{equation}
\lim_{s\rightarrow0}W(s)=0. \label{wron}%
\end{equation}
Define%
\[
f_{s}(z):=\frac{y_{\boldsymbol{a}(s)}(z)}{y_{-\boldsymbol{a}(s)}(z)}.
\]
Then $f_{s}(z)$ has no branch points and hence single-valued in $\mathbb{C}$,
which satisfies%
\[
f_{s}(z+1)=e^{-4\pi is}f_{s}(z),\text{ \ }f_{s}(z+\tau)=e^{4\pi i(\frac{1}%
{2}-\mathcal{C}s)}f_{s}(z).
\]
Furthermore, a direct computation gives%
\[
\frac{d}{dz}\ln f_{s}(z)=\frac{c(s)W(s)}{\Phi_{e,s}(z)},
\]
and so%
\[
e^{-4\pi is}=\frac{f_{s}(z+1)}{f_{s}(z)}=\exp \left(  c(s)W(s)\int_{z}%
^{z+1}\frac{1}{\Phi_{e,s}(\xi)}d\xi \right)  ,
\]%
\[
e^{4\pi i(\frac{1}{2}-\mathcal{C}s)}=\frac{f_{s}(z+\tau)}{f_{s}(z)}%
=\exp \left(  c(s)W(s)\int_{z}^{z+\tau}\frac{1}{\Phi_{e,s}(\xi)}d\xi \right)  .
\]
Therefore, there exist $m_{1},m_{2}\in \mathbb{Z}$ such that%
\[
\int_{z}^{z+1}\frac{1}{\Phi_{e,s}(\xi)}d\xi=\frac{-4\pi is+2\pi im_{1}%
}{c(s)W(s)},
\]%
\[
\int_{z}^{z+\tau}\frac{1}{\Phi_{e,s}(\xi)}d\xi=\frac{4\pi i(\frac{1}%
{2}-\mathcal{C}s)+2\pi im_{2}}{c(s)W(s)}.
\]
Together with (\ref{phi})-(\ref{phi1}), we have
\[
\lim_{s\rightarrow0}\frac{-4\pi is+2\pi im_{1}}{c(s)W(s)}=\chi_{1},\text{
}\lim_{s\rightarrow0}\frac{4\pi i(\frac{1}{2}-\mathcal{C}s)+2\pi im_{2}%
}{c(s)W(s)}=\chi_{2}.
\]
This, together with $\lim_{s\rightarrow0}c(s)=c_{0}\not \in \{0,\infty \}$
and (\ref{wron}), yields $(m_{1},m_{2})=(0,-1)$ and so%
\[
\frac{\chi_{2}}{\chi_{1}}=\lim_{s\rightarrow0}\frac{4\pi i(\frac{1}%
{2}-\mathcal{C}s)-2\pi i}{-4\pi is}=\mathcal{C}.
\]
This proves (\ref{phi2}). The proof is complete.
\end{proof}

\section{Proofs of Theorem \ref{thm1}}

\label{unique-u}

This section is devoted to proving Theorem \ref{thm1}.  In this section, we denote $N=\sum_{k}n_k+1$.
First we prove the uniqueness of GLE$(\mathbf{n}, p, A, \tau)$ with respect to the monodromy data.

\begin{proof}
[Proof of Theorem \ref{thm1}-(1)]Fix $\mathbf{n}$ and $\tau_{0}$. Suppose
GLE$(\mathbf{n}, p_{j}, A_{j}, \tau_{0})$, $j=1,2,$ have the same global
monodromy data.  Let $(p_{j}^{\mathbf{n}}(\tau),A_{j}^{\mathbf{n}}(\tau))$ be
the solution of the Hamiltonian system (\ref{142-0}) with initial data
$(p_{j}^{\mathbf{n}}(\tau_{0}),$ $A_{j}^{\mathbf{n}}(\tau_{0}))$ $=$ $\left(
p_{j},A_{j}\right)  $, $j=1,2$. Then $p_{j}^{\mathbf{n}}(\tau)$ are solutions
of the elliptic form (\ref{124}) with parameter (\ref{parameter0}). There are
two cases.

\textbf{Case 1}. The monodromies of GLE$(\mathbf{n}, p_{j}, A_{j}, \tau_{0})$ are completely reducible with the same global monodromy data
$(r_{j},s_{j})\in \mathbb{C}^{2}\backslash \frac{1}{2}\mathbb{Z}^{2}$ with
$(r_{1},s_{1})\sim(r_{2},s_{2})$. Then Theorem \ref{thm-II-8} implies
$p_{j}^{\mathbf{n}}(\tau)=p_{r_{j},s_{j}}^{\mathbf{n}}(\tau)$ and hence
$\wp(p_{1}^{\mathbf{n}}(\tau)|\tau)\equiv \wp(p_{2}^{\mathbf{n}}(\tau)|\tau)$.
In a small neighborhood $U$ of $\tau_{0}$ we may assume $p_{1}^{\mathbf{n}%
}(\tau)=\pm p_{2}^{\mathbf{n}}(\tau)+m_{1}+m_{2}\tau$ for some $m_{j}%
\in \mathbb{Z}$. Then it follows from the first equation of the Hamiltonian system (\ref{142-0}) that
$A_{1}^{\mathbf{n}}(\tau)=\pm A_{2}^{\mathbf{n}}(\tau)$ for $\tau \in U$. In
particular, these hold for $\tau_{0}$ and we conclude from
(\ref{GLE-invariant}) that GLE$(\mathbf{n,}$ $p_{1},$ $A_{1},$ $\tau_{0}%
)=$GLE$(\mathbf{n,}$ $p_{2},$ $A_{2},$ $\tau_{0})$.

\textbf{Case 2}. The monodromies of GLE$(\mathbf{n}, p_{j}, A_{j}, \tau_{0})$ are not completely reducible with the same global monodromy data
$(2\varepsilon_{k,1},2\varepsilon_{k,2},\mathcal{C})$. Thanks to Theorem
\ref{thm-II-8-0}, the same argument as Case 1 implies GLE$(\mathbf{n}, p_{1},
A_{1}, \tau_{0})=\text{GLE}(\mathbf{n}, p_{2}, A_{2}, \tau_{0})$.
\end{proof}

To prove Theorem \ref{thm1} for H$(\mathbf{n},B,\tau)$, we need to apply the relation between H$(\mathbf{n},B,\tau)$ and GLE$(\mathbf{n}, p, A, \tau)$ studied in \cite{Chen-Kuo-Lin}.

Fix any $\tau_{0}\in \mathbb{H}$ and $c_{0}^{2}\in\{\pm i\frac{2n_{0}+1}{2\pi}\}$. Then for any $h\in \mathbb{C}$, it was
proved in \cite{Chen-Kuo-Lin} that there exists a solution $p_{h}^{\mathbf{n}}(\tau)$ of the elliptic form (\ref{124}) with parameters (\ref{parameter0}) satisfying the following asymptotic
behavior%
\begin{equation}
p_{h}^{\mathbf{n}}(\tau)=c_{0}(\tau-\tau_{0})^{\frac{1}{2}}%
(1+h(\tau-\tau_{0})+O(\tau-\tau_{0})^{2})\text{ as }\tau
\rightarrow \tau_{0}. \label{515-5}%
\end{equation}
Recall Remark \ref{identify} that we identify the solutions $p_{h}^{\mathbf{n}}(\tau)$ and $-p_{h}^{\mathbf{n}}(\tau)$, so (\ref{515-5}) gives
two $1$-parameter families (one family is given by $c_{0}^{2}= i\frac{2n_{0}+1}{2\pi}$ and the other by $c_{0}^{2}= -i\frac{2n_{0}+1}{2\pi}$) of solutions of the elliptic form (\ref{124}) satisfying
$p_{h}^{\mathbf{n}}(\tau)\rightarrow0$ as $\tau\rightarrow
\tau_{0}$. Moreover, these two $1$-parameter families of solutions give all
solutions $p^{\mathbf{n}}(\tau)$ of the elliptic form (\ref{124}) such that
$p^{\mathbf{n}}(\tau_{0})=0$. See \cite[Section 3]{Chen-Kuo-Lin} for the proof.

By using (\ref{515-5}), we proved that the associated GLE$(\mathbf{n}
, p_{h}^{\mathbf{n}}(\tau), A(\tau), \tau)$ converges to either
H$(\mathbf{n}^{+},B_{0}, \tau_0)$ or H$(\mathbf{n}^{-}, B_{0}, \tau_0)$ for some $B_{0}$
$\in \mathbb{C}$ as $\tau \rightarrow \tau_{0}$ where $\mathbf{n}^{\pm}$\emph{
}$=(  n_{0}\pm1,n_{1},n_{2},n_{3})  $. More precisely, we
have

\begin{theorem} \cite{Chen-Kuo-Lin} \label{thm-6A} Let
$\tau_{0}\in \mathbb{H}$ and $p^{\mathbf{n}%
}(\tau)$ be a solution of the elliptic form
(\ref{124}) with parameters (\ref{parameter0}) such that $p^{\mathbf{n}}(\tau_{0})=0$. Then $p^{\mathbf{n}}(\tau)=\pm p_{h}^{\mathbf{n}}(\tau)$ for
some $h\in \mathbb{C}$. Furthermore, the associated
GLE$(\mathbf{n}, p^{\mathbf{n}}(\tau), A(\tau), \tau)$
converges to either H$(\mathbf{n}^{+},B_{0}, \tau_0)$ if $c_{0}%
^{2}=-i\frac{2n_{0}+1}{2\pi}$ or H$(\mathbf{n}^{-},B_{0}, \tau_0)$ if
$c_{0}^{2}=i\frac{2n_{0}+1}{2\pi}$. Here
\begin{equation}
B_{0}=2\pi ic_{0}^{2}\left(  4\pi i h-\eta_{1}(\tau_{0})\right)
-\sum_{k=1}^{3}n_{k}(n_{k}+1)e_{k}(\tau_{0}). \label{B}%
\end{equation}
\end{theorem}

\begin{proof}
[{Proof of Theorem \ref{thm1}-(2)}]
Fix $\mathbf{n}$ and $\tau_{0}$. Suppose
H$(\mathbf{n}, B_{j}, \tau_{0})$, $j=1,2,$ have the same global
monodromy data. Our goal is to prove $B_1=B_2$.

Let $\mathbf{n}^{+}=(
n_{0}+1,n_{1},n_{2},n_{3})$ and $c_{0}^{2}=i\frac{2(
n_{0}+1)  +1}{2\pi}$.
Define $h_{j},$ $j=1,2$, by (\ref{B}) by replacing $B_{0}$ with
$B_{j}$ and consider the solutions $p_{h_{j}}^{\mathbf{n}^{+}}(\tau)$. By Theorem \ref{thm-6A}, the associated GLE$(\mathbf{n}^{+},
p_{h_{j}}^{\mathbf{n}^{+}}(\tau),A_{h_j}^{\mathbf{n}^{+}}( \tau), \tau)$ converges to H($\mathbf{n},B_{j},\tau_{0}$) as $\tau \rightarrow \tau_{0}$. The key step is to show that
\begin{align}\label{globalmono}
&\text{\it the global monodromy data of GLE$(\mathbf{n}^{+},
p_{h_{j}}^{\mathbf{n}^{+}}(\tau),A_{h_j}^{\mathbf{n}^{+}}( \tau), \tau)$}\\
&\text{\it and H$(\mathbf{n},B_{j},\tau_{0})$ are the same.}\nonumber
\end{align}
Once (\ref{globalmono}) is proved, then GLE$(\mathbf{n}^{+},
p_{h_{j}}^{\mathbf{n}^{+}}(\tau),A_{h_j}^{\mathbf{n}^{+}}( \tau), \tau)$, $j=1,2$, have the same global monodromy data and so Theorem \ref{thm1}-(1) yields that these two GLEs coincide, i.e. $\wp(p_{h_{1}}^{\mathbf{n}^{+}}(\tau)|\tau)\equiv \wp(p_{h_{2}}^{\mathbf{n}^{+}}(\tau)|\tau)$. From here and $p_{h_{j}}^{\mathbf{n}^{+}}(\tau_0)=0$ for $j=1,2$, we obtain
$p_{h_{1}}^{\mathbf{n}^{+}}(\tau)=\pm p_{h_{2}}^{\mathbf{n}^{+}}(\tau)$ near $\tau_0$.
This implies $h_1=h_2$ and so $B_1=B_2$.

We only need to prove (\ref{globalmono}) for $j=1$ and
in the following proof we write $(p_{h_{1}}^{\mathbf{n}^{+}}(\tau),A_{h_1}^{\mathbf{n}^{+}}( \tau))=(p(\tau),A(\tau))$ for convenience.

{\bf Case 1}. $p(\tau)=p_{r,s}^{\mathbf{n}^{+}}(\tau)$ for some $(r,s)\in\mathbb{C}^2\setminus\frac{1}{2}\mathbb{Z}^2$ is a completely reducible solution, i.e. the global monodromy data of GLE$(\mathbf{n}^{+}, p(\tau), A(\tau), \tau)$ is $(r,s)$.

Denote $\hat{N}=\sum n_k+2$. Then by Theorem \ref{thm0} and (\ref{aunique})-(\ref{aunique11}),
there exists $\boldsymbol{a}(\tau)=(a_{1}(\tau),\cdot \cdot \cdot,a_{\hat{N}}(\tau))$ satisfying
\begin{equation}\label{suma}
\sum_{i=1}^{\hat{N}}a_{i}(\tau)  -\sum_{k=1}^{3}\frac{n_{k}\omega_{k}%
}{2}=r+s\tau
\end{equation}
such that
{\allowdisplaybreaks
\begin{align}\label{ya}
y_{\boldsymbol{a}(\tau)  }(z)
=&\frac{e^{(r\eta_1(\tau)+s\eta_2(\tau))  z}
\prod_{i=1}^{\hat{N}}
\sigma(  z-a_{i}(\tau)|\tau)  }{\sigma(  z|\tau)^{n_{0}+2}
\prod_{k=1}^{3}
\sigma(  z-\frac{\omega_{k}}{2}|\tau)^{n_{k}}  }\\
&\times\frac{\sigma(z|\tau)}{\sqrt{\sigma(
z-p(\tau)|\tau)  \sigma (  z+p(\tau)|\tau)  }}\nonumber
\end{align}
}%
is a solution of GLE$(\mathbf{n}^{+}, p(\tau), A(\tau), \tau)$. By passing a subsequence, we
may assume
\begin{equation}\label{alimi}
\lim_{\tau \rightarrow \tau_{0}}\boldsymbol{a}(  \tau )  =\boldsymbol{a}
=(a_{1},\cdot \cdot \cdot,a_{\hat{N}})  \in E_{\tau}^{\hat{N}}.
\end{equation}
Then
\begin{equation}
\sum_{i=1}^{\hat{N}}a_{i}-\sum_{k=1}^{3}\frac{n_{k}\omega_{k}(
\tau_{0})  }{2}=r+s\tau_{0}, \label{la}%
\end{equation}
and $p(\tau)\to p(\tau_0)=0$ implies that
\begin{align}\label{ya-1}
&  y_{\boldsymbol{a}}(z)
:=\frac{e^{(r\eta_1(\tau_0)+s\eta_2(\tau_0))  z}
\prod_{i=1}^{\hat{N}}
\sigma(  z-a_{i}|\tau_0)  }{\sigma(  z|\tau_0)^{n_{0}+2}
\prod_{k=1}^{3}
\sigma(  z-\frac{\omega_{k}}{2}|\tau_0)^{n_{k}}  }
\end{align}
is a solution of H$(\mathbf{n}, B_1,\tau)$. Note that two of $a_{1},\cdot \cdot \cdot,a_{\hat{N}}$ must be $0$ since the local exponents of H$(\mathbf{n}, B_1,\tau)$ at $0$ are $-n_0, n_0+1$. By (\ref{la})-(\ref{ya-1}) and the transformation law (\ref{518}), we immediately obtain that with respect to $y_{\mathbf{a}}(z)$ and $y_{-\mathbf{a}}(z)$, the monodromy matrices $\rho(\ell_j), j=1,2$, are exactly  (\ref{61.35}). This proves that the global monodromy data of H$(\mathbf{n}, B_1,\tau)$ is also the same $(r,s)$ as that of GLE$(\mathbf{n}^{+}, p(\tau), A(\tau), \tau)$.

\textbf{Case 2. }$p(\tau)=p_{k,\mathcal{C}}^{\mathbf{n}^{+}}(\tau)$ for some $k\in \{0,1,2,3\}$ and $\mathcal{C}\in\mathbb{C}\cup \{\infty\}$ is a not completely reducible solution, i.e. the global monodromy data of GLE$(\mathbf{n}^{+}, p(\tau), A(\tau), \tau)$ is $(2\varepsilon_{k,1},2\varepsilon_{k,2},\mathcal{C})$.

Recalling Theorem \ref{thm0-1} and (\ref{trace}),
there exists $\boldsymbol{a}(\tau)=(a_{1}(\tau),\cdot \cdot \cdot,a_{\hat{N}}(\tau))$ satisfying (\ref{suma}) and
\begin{equation}\label{rs-inter}
(r,s)\equiv\begin{cases} (0,0) \mod\mathbb{Z}^2\quad\text{if } k=0,\\(\frac{1}{2},0) \mod\mathbb{Z}^2\quad\text{if } k=1,\\(0,\frac{1}{2}) \mod\mathbb{Z}^2\quad\text{if } k=2,\\(\frac{1}{2},\frac{1}{2}) \mod\mathbb{Z}^2\quad\text{if } k=3,\end{cases}
\end{equation}
such that $y_{\boldsymbol{a}(\tau)}(z)$ given by (\ref{ya}) is a solution of GLE$(\mathbf{n}^{+}, p(\tau), A(\tau), \tau)$. As in Case 1, we may assume (\ref{alimi}), then $y_{\boldsymbol{a}}(z)$ given by (\ref{ya-1}) is a solution of H$(\mathbf{n}, B_1,\tau)$. By (\ref{la}), (\ref{rs-inter}) and (\ref{518}), we easily obtain
\[y_{\boldsymbol{a}}(z+\omega_j)=\varepsilon_{k,j}y_{\boldsymbol{a}}(z),\quad j=1,2.\]
Since the proof of Theorem \ref{thm0-1} gives $\mathcal{C}=\frac{\int_{z}^{z+\omega_2}y_{\boldsymbol{a}(\tau)}(\xi)^{-2}d\xi}
{\int_{z}^{z+\omega_1}y_{\boldsymbol{a}(\tau)}(\xi)^{-2}d\xi}$, it follows from $y_{\boldsymbol{a}(\tau)}(z)^{-2}\to y_{\boldsymbol{a}}(z)^{-2}$ that
\[\frac{\int_{z}^{z+\omega_2}y_{\boldsymbol{a}}(\xi)^{-2}d\xi}
{\int_{z}^{z+\omega_1}y_{\boldsymbol{a}}(\xi)^{-2}d\xi}=\mathcal{C}.\]
Therefore, the global monodromy data of H$(\mathbf{n},B_1,\tau_0)$ is $(2\varepsilon_{k,1},2\varepsilon_{k,2},\mathcal{C})$, again the same as that of GLE$(\mathbf{n}^{+}, p(\tau), A(\tau), \tau)$.

The proof is complete.
\end{proof}

\begin{proof}[Proof of Theorem \ref{thm1}-(3)] Fix any $\mathbf{n}$, $\tau_0$ and $k\in \{0,1,2,3\}$. Suppose that the global monodromy datas of H$(\mathbf{n}, B_1, \tau_0)$ and H$(\mathbf{n}_k, B_2, \tau_0)$ are the same for some $B_1, B_2\in\mathbb{C}$. By changing variable $z\to z+\frac{\omega_k}{2}$, we only need to consider the case $k=0$. Then (\ref{n-0-0}) implies \begin{equation}\label{fc-11}\mathbf{n}_0^{-}=(n_0+1,n_1,n_2,n_3)=\mathbf{n}^{+},\quad\text{i.e. }(\mathbf{n}^{+})^+=\mathbf{n}_0.\end{equation}
Define $h_1$ by (let $c_{0}^{2}=i\frac{2n_{0}+3}{2\pi}$ and $B_0=B_1$ in (\ref{B}))
\[B_{1}=-(2n_0+3)\left(  4\pi i h_1-\eta_{1}(\tau_{0})\right)
-\sum_{k=1}^{3}n_{k}(n_{k}+1)e_{k}(\tau_{0}),\]
and $h_2$ by (let $c_{0}^{2}=-i\frac{2n_{0}+3}{2\pi}$ and $B_0=B_2$ in (\ref{B}))
\[B_{2}=(2n_0+3)\left(  4\pi i h_2-\eta_{1}(\tau_{0})\right)
-\sum_{k=1}^{3}n_{k}(n_{k}+1)e_{k}(\tau_{0}).\]
Then it follows from (\ref{515-5}) that there exist solutions $p_{h_j}^{\mathbf{n}^+}$, $j=1,2$, satisfying
\begin{equation}
p_{h_1}^{\mathbf{n}^+}(\tau)=c_{1}(\tau-\tau_{0})^{\frac{1}{2}}%
(1+h_1(\tau-\tau_{0})+O(\tau-\tau_{0})^{2})\text{ as }\tau
\rightarrow \tau_{0}, \label{515-55}
\end{equation}
\begin{equation}
p_{h_2}^{\mathbf{n}^+}(\tau)=c_{2}(\tau-\tau_{0})^{\frac{1}{2}}%
(1+h_2(\tau-\tau_{0})+O(\tau-\tau_{0})^{2})\text{ as }\tau
\rightarrow \tau_{0}, \label{515-555}
\end{equation}
with $c_1^2=i\frac{2n_{0}+3}{2\pi}=-c_2^2$. In particular, \begin{equation}\label{p1p2}\wp(p_{h_1}^{\mathbf{n}^+}(\tau)|\tau)\neq \wp( p_{h_2}^{\mathbf{n}^+}(\tau)|\tau)\quad\text{for $\tau\to \tau_0$}.\end{equation}

On the other hand, it follows from (\ref{fc-11})-(\ref{515-555}) and Theorem \ref{thm-6A} that the associated
GLE$(\mathbf{n}^{+}, p_{h_1}^{\mathbf{n}^+}(\tau), A_{h_1}^{\mathbf{n}^+}(\tau), \tau)$
converges to H$(\mathbf{n},B_{1}, \tau_0)$ and GLE$(\mathbf{n}^+, p_{h_2}^{\mathbf{n}^+}(\tau), A_{h_2}^{\mathbf{n}^+}(\tau), \tau)$
converges to H$(\mathbf{n}_0,B_{2}, \tau_0)$ as $\tau\to \tau_0$.
Then the same proof as Theorem \ref{thm1}-(2) shows that GLE$(\mathbf{n}^+, p_{h_1}^{\mathbf{n}^+}(\tau), A_{h_1}^{\mathbf{n}^+}(\tau), \tau)$ has the same global monodromy data as H$(\mathbf{n},B_{1}, \tau_0)$ and so do for GLE$(\mathbf{n}^+$, $p_{h_2}^{\mathbf{n}^+}(\tau), A_{h_2}^{\mathbf{n}^+}(\tau), \tau)$, H$(\mathbf{n}_0,B_{2}, \tau_0)$. Together with our assumption, we conclude that GLE$(\mathbf{n}^+, p_{h_j}^{\mathbf{n}^+}(\tau), A_{h_j}^{\mathbf{n}^+}(\tau), \tau)$ has the same global monodromy data for $j=1,2$. Then it follows from Theorem \ref{thm1}-(1) that these two GLEs coincide, i.e. $\wp(p_{h_1}^{\mathbf{n}^+}(\tau)|\tau)\equiv\wp( p_{h_2}^{\mathbf{n}^+}(\tau)|\tau)$, a contradiction with (\ref{p1p2}).

The proof is complete.
\end{proof}

We want to emphasize that the same proof as (\ref{globalmono}) improves Theorem \ref{thm-6A} as follows.

\begin{theorem}
Under the same notations and assumptions as Theorem \ref{thm-6A}, GLE$(\mathbf{n},$ $p^{\mathbf{n}}(\tau), A(\tau), \tau)$ has the same global monodromy data with
its limiting equation H$(\mathbf{n}^{+},B_{0}, \tau_0)$ for $c_{0}%
^{2}=-i\frac{2n_{0}+1}{2\pi}$ (resp. H$(\mathbf{n}^{-},B_{0}, \tau_0)$ for
$c_{0}^{2}=i\frac{2n_{0}+1}{2\pi}$).
\end{theorem}

\section{Applications}

\label{sec-app}

In this section, we give an application of Theorem \ref{thm1} to GLE$(\mathbf{n}, p, A, \tau)$. First we recall the basic theory of GLE$(\mathbf{n}, p, A, \tau)$ from its  hyperelliptic aspect in Part I \cite{CKL1}.

Recall $\Phi_e(z)$ in Theorem 2.A. It follows from (\ref{303-1}) that
\[Q_{\mathbf{n},p}(A):=\Phi_e'(z)^{ 2}-2\Phi_e^{\prime \prime }(z)\Phi_2(z)+4I_{\mathbf{n}}(z;
p,A,\tau)\Phi_e(z)^2\]
is a monic polynomial in $A$ of degree $2g+2$ and independent of $z$. Since $\Phi_e(z)=y_1(z)y_2(z)$ (recall $y_2(z)=y_1(-z)$), it is known (cf. Part I \cite[Theorem 2.7]{CKL1}) that the Wronskian $W$ of $y_1(z)$ and $y_2(z)$ satisfies $W^2=Q_{\mathbf{n},p}(A)$.
Define the hyperelliptic curve $\Gamma_{\mathbf{n,%
}p}=\Gamma_{\mathbf{n,}p}(\tau)$ by
\begin{equation}
\Gamma_{\mathbf{n,}p}(\tau):=\{(A,W)|W^2=Q_{\mathbf{n,}p}(A;\tau)\}.
\end{equation}
Since $\deg_{A}Q_{\mathbf{n,}p}(A;\tau)  $ is even, the curve
$\Gamma_{\mathbf{n,}p}(\tau)  $
has two points at infinity denoted by $\infty_{\pm}$, i.e. $\overline{\Gamma_{\mathbf{n,}%
p}(  \tau )  }=\Gamma_{\mathbf{n,}p}(  \tau )
\cup \{  \infty_{\pm} \}$.
Clearly $y_1(z)$ can be uniquely determined by the pair $(A, W)\in \Gamma_{\mathbf{n,}p}(\tau)$ by considering the correspondence (note that $-W$ is the Wronskian of $y_2(z)$ and $y_1(z)=y_2(-z)$)
\[(y_1(z), y_2(z))\leftrightarrow (A, W),\quad (y_2(z), y_1(z))\leftrightarrow (A, -W).\]
Denote $N=\sum_{k=0}^3n_k+1$ in the sequel. Recall Section 2.2 that there is  ${\boldsymbol a}=\{a_1,\cdots,a_N\}$ (unique mod $\Lambda_{\tau}$) such that $y_1(z)=y_{{\boldsymbol a}}(z)$. Then we can define a map $i_{\mathbf{n},p}:\Gamma_{\mathbf{n},p}\rightarrow \,\text{Sym}^NE_\tau$
by
\begin{equation}
i_{\mathbf{n},p}(A,W) :=\{[a_{1}],\cdot \cdot \cdot,[a_{N}]\}\in \text{Sym}^NE_\tau,
\label{i1}
\end{equation}
where $[a_i]:=a_i \ (\text{mod}\ \Lambda_{\tau}) \in E_{\tau}$. Clearly this $i_{{\bf n},p}$ is well-defined. Furthermore, if $W\neq0$, then we see from $y_2(z)=y_{-\mathbf{a}}(z)$ that
\begin{equation}
i_{\mathbf{n},p}( A,-W) =\{-[a_{1}],\cdot \cdot \cdot,-[a_{N}]\}.
\label{i2}
\end{equation}
We proved in Part I \cite{CKL1} that
$i_{{\bf n},p}$ is an embedding from $\Gamma_{\mathbf{n},p}$ into Sym$^NE_\tau$.
Let $Y_{{\bf n},p}(\tau)$ be the image of $\Gamma_{{\bf n},p}(\tau)$ in Sym$^{N}E_\tau$
under $i_{{\bf n},p}$, i.e.
\begin{equation}
Y_{{\bf n},p}(\tau)  =\left \{
\begin{array}
[c]{r}%
[\boldsymbol{a}]=\{[a_{1}],\cdot \cdot \cdot,[a_{N}]\}\text{ }\in $Sym$^NE_{\tau}|\text{
}y_{\boldsymbol{a}}(z)  \text{ defined in }\\
\text{(\ref{exp1}) is a solution of GLE$(\mathbf{n}, p, A, \tau)$ for some $A$}
\end{array}
\right \}, \label{set}%
\end{equation}
and define the addition map $\sigma_{{\bf n},p}:Y_{{\bf n},p}(\tau)\rightarrow
E_\tau$ by
\begin{equation}
\sigma_{{\bf n},p}([\boldsymbol{a}]):=\sum_{i=1}^N[a_i]-\sum_{k=1}^3 [\tfrac{n_k\omega_k}{2}].
\end{equation}
 Clearly
\[\sigma_{{\bf n},p}([-\boldsymbol{a}])=-\sum_{i=1}^N[a_i]-\sum_{k=1}^3 [\tfrac{n_k\omega_k}{2}]=-\sigma_{{\bf n},p}([\boldsymbol{a}]).\]
Furthermore, the degree $\deg \sigma_{{\bf n},p}= \#\sigma_{{\bf n}, p}^{-1}(z), z\in E_\tau$,
is well-defined and
\[
\deg \sigma_{\mathbf{n,}p}=\sum_{k=0}^{3}n_{k}(n_{k}+1)+1.
\]
Besides,
\begin{equation*}
\overline{Y_{\mathbf{n}%
,p}(\tau) }=Y_{\mathbf{n},p}( \tau ) \cup  \{
\infty_{+}(p),  \infty_{-}(p)\},
\end{equation*}
where
\begin{equation*}
\infty_{\pm}( p) :=\bigg( \overset{n_{0}}{\overbrace{0,\cdot \cdot
\cdot,0}},\overset{n_{1}}{\overbrace{\tfrac{\omega_{1}}{2},\cdot \cdot \cdot,%
\tfrac{\omega_{1}}{2}}},\overset{n_{2}}{\overbrace{\tfrac{\omega_{2}}{2},\cdot
\cdot \cdot,\tfrac{\omega_{2}}{2}}},\overset{n_{3}}{\overbrace {\tfrac{%
\omega_{3}}{2},\cdot \cdot \cdot,\tfrac{\omega_{3}}{2}}},\pm p\bigg).
\end{equation*}
The above theories can be found in Part I \cite{CKL1}.

Let $K(E_{\tau})$ and $K(  \overline{Y_{\mathbf{n,}
p}(\tau)})$ be the field of rational functions of
$E_{\tau}$ and $\overline{Y_{\mathbf{n,}
p}(\tau)}$,
respectively. Then $K(\overline{Y_{\mathbf{n,}
p}(\tau)})$ is a finite extension over $K(E_{\tau})$ and
\begin{equation}\label{1finite-ext}\left[K(\overline{Y_{\mathbf{n,}
p}(\tau)}): K(E_{\tau})\right]=\deg\sigma_{\mathbf{n},p}=\sum_{k=0}^{3}n_{k}(n_{k}+1)+1.\end{equation}
In this section, we consider the basic question \emph{what a primitive generator of this field extension is}.
Motivated by (\ref{61-37})-(\ref{61-38}), we define
\begin{align}\label{fc-f1}\mathbf{z}_{\mathbf{n},p}(a_1,\cdots,a_N):=&\zeta\Bigg(\sum_{i=1}^N a_i-\sum_{k=1}^{3}\frac{n_{k}\omega_{k}}{2}\Bigg)\\
&-\frac{1}{2}\sum_{i=1}^{N}(\zeta
(a_{i}+p)+\zeta(a_{i}-p))+\sum_{k=1}^{3}\frac{n_{k}\eta_{k}
}{2},\nonumber\end{align}
which is meromorphic and periodic in each $a_i$ and hence defines a rational function on $E_{\tau}^N$. By symmetry, it descends to a rational function on Sym$^N E_{\tau}$. We denote the restriction $\mathbf{z}_{\mathbf{n},p}|_{\overline{Y_{\mathbf{n,}
p}(\tau)}}$ also by $\mathbf{z}_{\mathbf{n},p}$, which is a rational function on $\overline{Y_{\mathbf{n,}
p}(\tau)}$. Here as an application of Theorem \ref{thm1}, we can prove that $\mathbf{z}_{\mathbf{n},p}(\boldsymbol{a})$ is a primitive generator. The same statement as the following result was proved in \cite{LW2} for the Lam\'{e} equation and later generalized to H$(\mathbf{n},B,\tau)$ in Part II \cite{CKL2}.

\begin{theorem}
$\mathbf{z}_{\mathbf{n},p}$ is a primitive generator of the finite extension of rational function field $K(\overline{Y_{\mathbf{n,}
p}(\tau)})$ over $K(E_{\tau})$, i.e. the minimal polynomial $W_{\mathbf{n},p}(\mathbf{z})\in K(E_{\tau})[\mathbf{z}]$ of $\mathbf{z}_{\mathbf{n},p}$ satisfies $\deg W_{\mathbf{n},p}=\deg \sigma_{\mathbf{n},p}$.
\end{theorem}

\begin{proof}
Since $\mathbf{z}_{\mathbf{n},p}\in K(  \overline{Y_{\mathbf{n,}
p}(\tau)})$, its minimal polynomial $W_{\mathbf{n},p}(\mathbf{z})\in K(E_{\tau})[\mathbf{z}]=\mathbb{C}(\wp(\sigma),\wp'(\sigma))[\mathbf{z}]$ exists with degree
$d_{\mathbf{n},p}:=\deg W_{\mathbf{n},p}|\deg \sigma_{\mathbf{n},p}$ by (\ref{1finite-ext}).

Note that if $\boldsymbol{a}=-\boldsymbol{a}$, then $\sigma_{\mathbf{n},p}(\boldsymbol{a})\in E_{\tau}[2]$.
To prove $d_{\mathbf{n},p}=\deg \sigma_{\mathbf{n},p}$, i.e. $\mathbf{z}_{\mathbf{n},p}(\boldsymbol{a})$ is a primitive generator, we take $\sigma_0\in E_{\tau}\setminus E_{\tau}[2]$ outside the branch loci of $\sigma_{\mathbf{n},p}: \overline{Y_{\mathbf{n,}
p}(\tau)}\to E_{\tau}$ such that there are precisely $\deg \sigma_{\mathbf{n},p}$ different points $\boldsymbol{a}^{k}\in Y_{\mathbf{n,}
p}(\tau)$ satisfying $\sigma_{\mathbf{n},p}(\boldsymbol{a}^k)=\sigma_0$ and $\pm [p]\notin \boldsymbol{a}^{k}$  for $1\leq k\leq \deg \sigma_{\mathbf{n},p}$. We claim that
\begin{equation}\label{fc-fc}
 \mathbf{z}_{\mathbf{n},p}(\boldsymbol{a}^{k_1})\neq  \mathbf{z}_{\mathbf{n},p}(\boldsymbol{a}^{k_2}),\quad \forall k_1\neq k_2.
\end{equation}
Suppose for some $k_1\neq k_2$ we have $\mathbf{z}_{\mathbf{n},p}(\boldsymbol{a}^{k_1})=  \mathbf{z}_{\mathbf{n},p}(\boldsymbol{a}^{k_2})$.
Then we can take $(a_1,\cdots,a_N)$, $(b_1,\cdots,b_N)\in\mathbb{C}^N$ to be representatives of $\boldsymbol{a}^{k_1}, \boldsymbol{a}^{k_2}$ such that
\[\sum_{i=1}^N a_i=\sum_{i=1}^N b_i,\;\sum_{i=1}^{N}(\zeta
(a_{i}+p)+\zeta(a_{i}-p))=\sum_{i=1}^{N}(\zeta
(b_{i}+p)+\zeta(b_{i}-p)).\]
By (\ref{set}), there exist $A_1, A_2$ such that $y_{\boldsymbol{a}^{k_1}}(z)$ (resp. $y_{\boldsymbol{a}^{k_2}}(z)$) is a solution of GLE$(\mathbf{n}, p, A_1, \tau)$ (resp. GLE$(\mathbf{n}, p, A_2, \tau)$). Then (\ref{61.35})-(\ref{61-38}) imply that GLE$(\mathbf{n}, p, A_1, \tau)$ and GLE$(\mathbf{n}, p, A_1, \tau)$ have the same global monodromy data $(r,s)\notin\frac{1}{2}\mathbb{Z}^2$, namely $y_{\boldsymbol{a}^{k_1}}(z)$ and $y_{\boldsymbol{a}^{k_2}}(z)$ satisfy the same transformation law:
\begin{equation}\label{1same-trans}\ell_1^*y(z)=e^{-2\pi i s}y(z),\quad \ell_2^*y(z)=e^{2\pi i r}y(z).\end{equation}
Consequently, Theorem \ref{thm1} implies GLE$(\mathbf{n}, p, A_1, \tau)=$GLE$(\mathbf{n}, p, A_2, \tau)$, i.e. $y_{\boldsymbol{a}^{k_1}}(z)$ and $y_{\boldsymbol{a}^{k_2}}(z)$ are solutions of the same GLE$(\mathbf{n}, p, A_1, \tau)$ and satisfies the same transformation law (\ref{1same-trans}).
It follows from $(r,s)\notin\frac{1}{2}\mathbb{Z}^2$ and (\ref{61.35}) that $y_{\boldsymbol{a}^{k_1}}(z)=y_{\boldsymbol{a}^{k_2}}(z)$, so $\boldsymbol{a}^{k_1}=\boldsymbol{a}^{k_2}$, a contradiction.

This proves (\ref{fc-fc}), which infers that these $\deg \sigma_{\mathbf{n},p}$ different points $\boldsymbol{a}^k$'s give $\deg \sigma_{\mathbf{n},p}$ different values $\mathbf{z}_{\mathbf{n},p}(\boldsymbol{a}^k)$'s. That is for $\sigma=\sigma_0$, the polynomial $W_{\mathbf{n},p}(\mathbf{z})\in \mathbb{C}(\wp(\sigma),\wp'(\sigma))[\mathbf{z}]$ of degree $d_{\mathbf{n},p}|\deg \sigma_{\mathbf{n},p}$ has $\deg \sigma_{\mathbf{n},p}$ distinct zeros $\mathbf{z}_{\mathbf{n},p}(\boldsymbol{a}^k)$'s, which implies $d_{\mathbf{n},p}=\deg \sigma_{\mathbf{n},p}$. The proof is complete.
\end{proof}

\begin{remark}
For $(r,s)\in\mathbb{C}^2\setminus \frac{1}{2}\mathbb{Z}^2$, as in \cite{CKL2,LW2} we define
\[Z_{r,s}(\tau):=\zeta(r+s\tau|\tau)-r\eta_1(\tau)-s\eta_2(\tau).\]
Then it follows from (\ref{fc-f1}) and (\ref{61.35})-(\ref{61-38}) that $\mathbf{z}_{\mathbf{n},p}(\boldsymbol{a})=Z_{r,s}(\tau)$ with $\sigma_{\mathbf{n},p}(\boldsymbol{a})=r+s\tau$. Therefore, like the Lam\'{e} case proved in \cite{LW2} and the general Darboux-Treibich-Verdier case proved in Part II \cite{CKL2}, the monodromy data $(r,s)$ of GLE$(\mathbf{n},p,A,\tau)$ in (\ref{61.35})-(\ref{61-38}) can be characterized by
\begin{equation}\label{fc-f2}W_{\mathbf{n},p}(Z_{r,s}(\tau))=0
\quad\text{with}\quad\sigma=r+s\tau.
\end{equation}
Let us consider the special case $\mathbf{n}=\mathbf{0}$ for example. Then
\[\mathbf{z}_{\mathbf{0},p}(a)=\zeta(a)-\tfrac{1}{2}(\zeta
(a+p)+\zeta(a-p))=\frac{\wp'(a)}{2(\wp(p)-\wp(a))}\in K(E_{\tau}),\]
i.e. its minimal polynomial $W_{\mathbf{0},p}(\mathbf{z})=\mathbf{z}-\mathbf{z}_{\mathbf{0},p}(a)$. So (\ref{fc-f2}) is just
\[Z_{r,s}(\tau)-\frac{\wp'(r+s\tau)}{2(\wp(p)-\wp(r+s\tau))}=0,\]
which recovers Hitchin's formula
\[\wp(p|\tau)=\wp(r+s\tau|\tau)+\frac
{\wp^{\prime}(r+s\tau|\tau)}{2Z_{r,s}(\tau)}.\]
Therefore, (\ref{fc-f2}) should be closely related to the formula of solutions of Painlev\'{e} VI equation with parameter (\ref{parameter0})-(\ref{parameter}) for general $\mathbf{n}$, which will be studied elsewhere.
\end{remark}

\medskip

{\bf Acknowledgements} We thank Prof. Treibich for his interest and valuable comments to this work.
The research of Z. Chen was supported by NSFC (No. 11871123) and Tsinghua University Initiative Scientific Research Program (No. 2019Z07L02016).
The research of T-J. Kuo was supported by MOST (MOST 107-2628-M-003-002-MY4).


\begin{thebibliography}{99}


\bibitem {Babich-Bordag} M. V. Babich and L. A. Bordag; \textit{Projective differential geometrical structure of the Painlev\'{e} equations}. J. Differ. Equ. \textbf{157} (1999), 452-485.

\bibitem {Boalch}P. Boalch; \textit{From Klein to Painlev\'{e} via Fourier,
Laplace and Jimbo}. Proc. Lond. Math. Soc. \textbf{90} (2005), 167-208.



\bibitem {CLW}{ C.L. Chai, C.S. Lin and C.L. Wang; \textit{Mean field
equations, Hyperelliptic curves, and Modular forms: I}. Camb. J.
Math. \textbf{3} (2015), 127-274.}

\bibitem {Chen-Kuo-Lin}Z. Chen, T.J. Kuo and C.S. Lin; \textit{Hamiltonian
system for the elliptic form of Painlev\'{e} VI equation}. J. Math. Pures
Appl. \textbf{106} (2016), 546-581.

\bibitem{CKL-JGP} Z. Chen, T.J. Kuo and C.S. Lin; \textit{Unitary monodromy
implies the smoothness along the real axis for some Painlev\'{e} VI
equation, I}. J. Geom. Phys. \textbf{116} (2017), 52-63.

\bibitem {CKL1} Z. Chen, T.J. Kuo and C.S. Lin; \textit{The geometry of generalized Lam\'{e} equation, I}. J. Math. Pures
Appl.  \textbf{127} (2019), 89-120.

\bibitem {CKL2} Z. Chen, T.J. Kuo and C.S. Lin; \textit{The geometry of generalized Lam\'{e} equation, II: Existence of pre-modular forms and application}. J. Math. Pures
Appl. \textbf{132} (2019), 251-272.

\bibitem {CKLW}Z. Chen, T.J. Kuo, C.S. Lin and C.L. Wang; \textit{Green
function, Painlev\'{e} VI equation, and Eisenstein series of weight one}. J.
Differ. Geom. \textbf{108} (2018), 185-241.

\bibitem{Dahmen} S. Dahmen; \textit{Counting integral Lam\'{e} equations by
means of dessins d'enfants}. Trans. Amer. Math. Soc. \textbf{359} (2007),
909-922.

\bibitem{Darboux} G. Darboux; Sur une \'{e}quation lineare. C. R. Acad. Sci. Paris, t. \textbf{XCIV}(25), (1882), 1645-1648.

\bibitem{FIK} A.S. Fokas, A.R. Its, A.A.Kapaev and V.Yu.Novokshenov; \textit{Painlev\'{e} transcendents. The Riemann-Hilbert
approach}. Mathematical Surveys and Monographs, Vol. 128, Amer. Math. Soc., Providence, RI, 2006.

\bibitem {Fuchs}R. Fuchs; \"{U}ber lineare homogene Differentialgleichungen
zweiter Ordnung mit drei im Endlichen gelegenen wesentlich singul\"{a}ren
Stellen, Math. Ann. 63 (1907), 301-321.



\bibitem {GW1}F. Gesztesy and R. Weikard; \textit{Treibich-Verdier potentials
of the stationary (m)KdV hierarchy}. Math. Z. \textbf{219} (1995), 451-476.


\bibitem {Halphen}{G.H. Halphen; \textit{Trait\'{e} des Fonctions Elliptiques
et de leurs Applications II}, Gauthier-Villars et Fils, Paris,\ 1888.}

\bibitem {Hit1}{N.J. Hitchin; \textit{Twistor spaces, Einstein metrics and
isomonodromic deformations}. J. Differ. Geom. \textbf{42} (1995), no.1,
30-112.}

\bibitem {Inaba-Iwasaki-Saito}M. Inaba, K. Iwasaki and M. Saito;
\textit{B\"{a}cklund transformations of the sixth Painlev\'{e} equation in
terms of Riemann-Hilbert correspodence}. Inter. Math. Res. Not. \textbf{1}
(2004), 1-30.

\bibitem{Ince} E. Ince; \textit{Further investigations into the periodic Lam%
\'{e} functions}. Proc. Roy. Soc. Edinburgh \textbf{60} (1940), 83-99.

\bibitem {GP}{K. Iwasaki, H. Kimura, S. Shimomura and M. Yoshida; \textit{From
Gauss to Painlev\'{e}: A Modern Theory of Special Functions}. Springer vol.
E16, 1991.}

\bibitem{Kawai} S. Kawai; \textit{Isomonodromic deformation of Fuchsian projective connections on elliptic curves}. Nagoya Math. J. \textbf{171} (2003), 127-161.

\bibitem {LW}{ C.S. Lin and C.L. Wang; \textit{Elliptic functions, Green
functions and the mean field equations on tori}. Ann. Math. \textbf{172}
(2010), no.2, 911-954.}


\bibitem {LW2} C.S. Lin and C.L. Wang; \textit{Mean field equations,
Hyperelliptic curves, and Modular forms: II}.  J. \'{E}c. polytech. Math.
\textbf{4} (2017), 557-593.

\bibitem {Maier}R. Maier; \textit{Lam\'{e} polynomail,
hyperelliptic reductions and Lam\'{e} band structure. }Philos.
Trans. R. Soc. A \textbf{366} (2008), 1115-1153.

\bibitem {Y.Manin}{ Y. Manin; \textit{Sixth Painlev\'{e} quation, universal
elliptic curve, and mirror of} $\mathbb{P}^{2}$. Amer. Math. Soc. Transl. (2),
\textbf{186} (1998), 131--151.}

\bibitem {Okamoto2}{ K. Okamoto; \textit{Isomonodromic deformation and
Painlev\'{e} equations, and the Garnier system}. J. Fac. Sci.\ Univ. Tokyo
Sec. IA Math. \textbf{33} (1986), 575-618. }

\bibitem {Okamoto1}{ K. Okamoto; \textit{Studies on the Painlev\'{e}
equations. I. Sixth Painlev\'{e} equation} $P_{VI}$. Ann. Mat. Pura Appl.
\textbf{146} (1986), 337-381. }



\bibitem {Poole}{ E. Poole; \textit{Introduction to the theory of linear
differential equations}. Oxford University Press, 1936.}

\bibitem{ST} T. Shioda and K. Takano; \textit{On some Hamiltonian structures of Painleve systems, I}. Funkc. Ekvac. \textbf{40} (1997), 271-291.

\bibitem {Takemura}K. Takemura; \textit{The Hermite-Krichever Ansatz for
Fuchsian equations with applications to the sixth Painlev\'{e} equation and to
finite gap potentials}. Math. Z. \textbf{263} (2009), 149-194.

\bibitem {Takemura1}K. Takemura; \textit{The Heun equation and the
Calogero-Moser-Sutherland system I: the Bethe Ansatz method}. Comm. Math.
Phys. \textbf{235} (2003), 467-494.

\bibitem {Takemura2}K. Takemura; \textit{The Heun equation and the
Calogero-Moser-Sutherland system II: perturbation and algebraic solution}.
Elec. J. Differ. Equ. \textbf{2004} (2004), no. 15, 1-30.

\bibitem {Takemura3}K. Takemura; \textit{The Heun equation and the
Calogero-Moser-Sutherland system III: the finite-gap property and the
monodromy}. J. Nonl. Math. Phys. \textbf{11} (2004), 21-46.

\bibitem {Takemura4}K. Takemura; \textit{The Heun equation and the
Calogero-Moser-Sutherland system IV: the Hermite-Krichever Ansatz}. Comm.
Math. Phys. \textbf{258} (2005), 367-403.

\bibitem {Takemura5}K. Takemura; \textit{The Heun equation and the
Calogero-Moser-Sutherland system V: generalized Darboux transformations}. J.
Nonl. Math. Phys. \textbf{13} (2006), 584-611.

\bibitem{Takemura09} K. Takemura; \textit{Middle convolution and Heun's equation}. SIGMA \textbf{5} (2009), 040.

\bibitem {Tsuda-Okamoto-Sakai}T. Tsuda, K. Okamoto and H. Sakai;
\textit{Folding transformations of the Painlev\'{e} equations}. Math. Ann.
\textbf{331} (2005), 713-738.

\bibitem{Treibich}A. Treibich; \textit{New elliptic potentials}. Acta Appl. Math. \textbf{36} (1994), 27-48.

\bibitem {TV}A. Treibich and J. L. Verdier; \textit{Revetements exceptionnels
et sommes de 4 nombres triangulaires}. Duke Math. J. \textbf{68} (1992), 217-236.

\bibitem {Whittaker-Watson}{ E. Whittaker and G. Watson, \textit{A course of
modern analysis}. Cambridge University Press, 1996.}
\end{thebibliography}
\end{document}